\documentclass[]{article}
\usepackage{setspace}
\usepackage{pdfpages}
\usepackage{float}
\usepackage[english]{babel}
\usepackage[utf8]{inputenc}
\usepackage{babel,csquotes,xpatch}
\usepackage[margin=1in]{geometry}
\usepackage{fancyhdr}
\usepackage{amsmath}
\usepackage{amssymb}
\usepackage{centernot}
\usepackage{mathtools}
\usepackage{tikz-cd}
\usepackage{tkz-euclide}
\usepackage[toc,page]{appendix}
\usetikzlibrary{decorations.markings}
\usepackage{esint}
\usepackage{cases}
\usepackage{nicematrix}

\usepackage{url}
\usepackage{listings}
\usetikzlibrary{intersections}
\usetikzlibrary{calc}
\usepackage{amsthm}
\usepackage{dsfont}
\usepackage{makecell}
\usepackage{color} %
\definecolor{mygreen}{RGB}{28,172,0} %
\definecolor{mylilas}{RGB}{170,55,241}
\usepackage{upquote} %
\usepackage{eurosym} %
\usepackage{grffile}
\usepackage{subcaption}
\usepackage{hyperref}
\hypersetup{
    colorlinks=true,
    citecolor=blue,
    linkcolor=blue,
    filecolor=magenta,      
    urlcolor=blue,
}
\usepackage{longtable} %
\usepackage{booktabs}  %
\usepackage[inline]{enumitem} %
\usepackage[normalem]{ulem} %
\usepackage{mathrsfs}
\usepackage[linesnumbered, ruled, lined, shortend]{algorithm2e}
\usepackage[style=trad-abbrv, backend=biber, doi = false, url = false, isbn = false, giveninits=true]{biblatex}
\AtEveryBibitem{\clearfield{issn}}

\usepackage{titlesec}

\titleformat*{\section}{\large\bfseries}
\titleformat*{\subsection}{\bfseries}

\title{Numerical Analysis of differential equations on weighted Sobolev spaces: beyond classical orthogonal polynomials}
\author{Maxime Breden\thanks{
CMAP, CNRS, \'Ecole polytechnique, Institut Polytechnique de
Paris, 91120 Palaiseau, France. \href{mailto:maxime.breden@polytechnique.edu}{maxime.breden@polytechnique.edu}, \href{mailto:hugo.chu@polytechnique.edu}{hugo.chu@polytechnique.edu}} $\quad$ Hugo Chu\footnotemark[1]$\;$\thanks{Department of Mathematics, Imperial College London, London SW7 2AZ, United Kingdom.} }

\newtheorem{thm}{Theorem}

\newtheorem{notation}[thm]{Notation}
\newtheorem{cor}[thm]{Corollary}
\newtheorem{rmk}[thm]{Remark}
\newtheorem{lem}[thm]{Lemma}

\newtheorem{ex}[thm]{Example}

\newtheorem{prop}[thm]{Proposition}

\newtheorem{manualtheoreminner}{\bf Claim}

\newtheorem{manualtheoreminnerb}{Step}
\newenvironment{step}[1]{%
  \renewcommand\themanualtheoreminnerb{#1}%
  \manualtheoreminnerb
}{\endmanualtheoreminner}

\renewcommand{\d}{\mathrm{d}}

\newcommand{\pdiff}[2]{\frac{\partial#1}{\partial#2}}

\newcommand{\Ind}[1]{\mathds{1}_{#1}}
\newcommand{\R}{\mathbb{R}}

\newcommand{\Z}{\mathbb{Z}}

\newcommand{\iPL}{\Pi_{>n}^{L^2}}
\newcommand{\fPL}{\Pi_{\leq n}^{L^2}}
\newcommand{\iPH}{\Pi_{>n}^{H^1}}
\newcommand{\fPH}{\Pi_{\leq n}^{H^1}}
\newcommand{\fh}{h_{\leq n}}
\newcommand{\ih }{h_{>n}}

\newcommand{\T}{\mathbb{T}}

\renewcommand{\d}{\mathrm{d}}

\newcommand{\N}{\mathbb{N}}
\newcommand{\bu}{\bar{u}}
\newcommand{\bv}{\bar{v}}

\newcommand{\cL}{\mathcal{L}}
\newcommand{\cP}{\mathcal{P}}
\newcommand{\cQ}{\mathcal{Q}}
\newcommand{\cX}{\mathcal{X}}
\newcommand{\cC}{\mathcal{C}}
\newcommand{\cJ}{\mathcal{J}}
\newcommand{\cF}{\mathcal{F}}
\newcommand{\cZ}{\mathcal{Z}}
\newcommand{\cE}{\mathcal{E}}
\newcommand{\cR}{\mathcal{R}}

\newcommand{\tL}{\tilde{\mathcal{L}}}
\newcommand{\id}{\mathrm{id}}
\newcommand{\Diag}{\mathrm{Diag}}

\begin{document}

\maketitle


\begin{abstract}
    We lay mathematical foundations for the Numerical Analysis of differential equations on Sobolev spaces weighted by a Gibbs probability measure $\nu(\d x) = e^{-V(x)}\d x/\mathcal{Z}$ on the real line. Over recent decades, the Functional Analysis of these spaces has been thoroughly developed to study Schrödinger-type equations and diffusion processes. While such equations should therefore be amenable to a numerical resolution with respect to orthogonal polynomials, this feat has only ever been achieved with respect to classical bases. We bridge this gap by showing that such equations can be solved with respect to suitable bases by factorising their leading linear component. In particular, we propose a new natural notion of Sobolev orthogonal polynomials, simpler and more tractable than those arising from the usual Sobolev inner product.
    
    In the case of $V$ being an even polynomial, we further establish quantitative estimates for the compactness of the embedding $H^1(\nu)\hookrightarrow L^2(\nu)$, uncovering a connection with the growth of the Jacobi recurrence coefficients, which are solutions of corresponding Painlevé-type discrete equations. As an application, we rigorously and tightly enclose solutions of the Gross--Pitaevskii equation with sextic potential and rigorously demonstrate the phenomenon of stochastic resonance via a computer-assisted proof.
\end{abstract}

\begin{center}
{\bf \small Keywords} \\ \vspace{.05cm}
{ \small Sobolev orthogonal polynomials $\cdot$ Weighted Sobolev spaces $\cdot$ Painlevé equations $\cdot$ Freud weight \\ Compactness estimates $\cdot$ Computer-assisted proofs $\cdot$ Stochastic resonance}
\end{center}

\begin{center}
{\bf \small Mathematics Subject Classification (2020)}  \\ \vspace{.05cm}
{\small   	35B45 $\cdot$ 42C05 $\cdot$ 46E35 $\cdot$ 47B36 $\cdot$ 41A81 $\cdot$ 35J61 $\cdot$ 65G20 $\cdot$ 65Q30 $\cdot$ 65N15 $\cdot$ 37H30 } 
\end{center}


\section{Introduction}

For an even potential $V\in \mathcal{C}^2(\R)$ growing at least linearly at infinity, consider the Gibbs probability measure $\nu(\d x) := e^{-V}\d x/\cZ$. It is then well known~\cite{Ernst2012OnExpansions} that one may construct a Hilbert basis of orthonormal polynomials $\mathcal{P} = \{p_n\}_{n\in\N}$ on $(L^2(\nu),\langle \cdot, \cdot\rangle)$, that is
$$\langle p_n,p_m\rangle = \int_{\R}p_np_m \d\nu = \delta_{mn}, \qquad \cZ :=\int_{\R}e^{-V(x)}\d x. $$
Orthogonal polynomials~\cite{Abramowitz1970HandbookSeries,Deift2000OrthogonalApproach, Gautschi2004OrthogonalPolynomials,1990OrthogonalPolynomials,Szego1939OrthogonalPolynomials} are among the oldest and most fundamental objects in Numerical Analysis, and much is known about their asymptotic behaviour, zeroes, relations with special functions and random matrices, etc. Moreover, they admit a simple construction via a three-term recurrence relation
\begin{equation}\label{eqn:jacobi-rec-rel_intro}
    xp_n = a_{n+1}p_{n+1} +a_n p_{n-1}, \qquad n\geq 1.
\end{equation}
However, somewhat disappointingly, almost all the resolutions of differential equations using orthogonal polynomials rely on the classical families (Fourier, Hermite, Jacobi and Laguerre)~\cite{Driscoll2014}. The main reason is that these classical families have the property that the derivative of a classical orthogonal polynomial is a polynomial in the same family. Therefore, these bases allow for a simple representation of the differentiation operator $\partial_x$ (typically diagonal, maybe up to a shift), which makes them very convenient for the study of differential equations.

Nevertheless, even if $\nu$ is not a classical weight, it is known~\cite{Bakry1994LhypercontractiviteSemigroupes,Bakry2014AnalysisOperators} that if $V$ grows fast enough at infinity, the Sobolev embedding $H^1(\nu)\hookrightarrow L^2(\nu)$ is compact (one actually expects that the faster $V$ grows, the better the compactness). In particular, such compactness results have been useful in the pen-and-paper analysis of (possibly nonlinear) Fokker--Planck equations~\cite{Escobedo1987VariationalEquation} and Schr\"odinger type equations~\cite{Kavian1994Self-similarEquation}. These equations can be reformulated~\cite{Pavliotis2014StochasticApplications} as
\begin{equation}\label{eqn:Lu=f(u)}
    \mathcal{L}u = f(u), \qquad \mathcal{L} := V'(x)\partial_x -\partial_{xx},
\end{equation}
which is then naturally posed in $H^1(\nu)$, where $\mathcal{L}$ is also a well-studied operator in Probability Theory~\cite{Bakry2014AnalysisOperators,Kavian1993SomeUltracontractivity,Pavliotis2014StochasticApplications} (as the generator of a diffusion process). 
In principle, the compactness of the embedding $H^1(\nu)\hookrightarrow L^2(\nu)$ also allows one to approximate compact operators like $\cL^{-1}$ on these spaces with finite-dimensional matrices on the computer, and therefore to numerically solve such equations.
%
%
Ideally, for both theoretical and practical purposes, we would like $\mathcal{L}$ to be diagonal with respect to the associated orthonormal polynomial basis $\cP =\{p_n\}_{n\in \N^*}$. Unfortunately, it is established~\cite{Bakry2014AnalysisOperators,Bakry2003CharacterizationPolynomials,Bakry2022OrthogonalOperators,Mazet1997ClassificationOrthogonaux} that this only occurs if the $p_n$'s belong to a classical family, suggesting that working with respect to non-classical weights is quite more challenging. 

The goal of this work is to overcome the apparent hurdles of making use of non-classical orthogonal polynomials, and to lay some mathematical foundations for the numerical resolution of problems of the form~\eqref{eqn:Lu=f(u)} when $V$ is not necessarily quadratic. We will show that such bases, because well-adapted to the problem, can be remarkably successful at solving problems of the form~\eqref{eqn:Lu=f(u)}, clearly outperforming some classical bases such as Hermite polynomials.


\subsection{A Numerical Analysis on weighted Sobolev spaces} 

Numerical methods for differential equations often require solutions to be sought in function spaces with some regularity, such as $H^1(\nu)$, and thus a Galerkin representation with respect to bases of such spaces. When $V(x)=x^2/2$, the $p_n$'s (which are then Hermite polynomials) enjoy the convenient property that they remain orthogonal with respect to $H^1(\nu)$-inner products of the form 
\begin{equation}\label{eqn:usual-ip}\int_{\R}u'v'\d\nu +\Lambda \int_{\R}uv\d\nu, \qquad \Lambda>0.\end{equation}
Unfortunately, as soon as $\nu$ is not a classical weight, the $p_n$'s are no longer orthogonal for such $H^1(\nu)$-inner products.
There has therefore been a great effort since the 1960s to construct polynomial bases orthogonal with respect to~\eqref{eqn:usual-ip}, so-called Sobolev orthogonal polynomials, leading to a very rich literature (see~\cite{Iserles1991OnProducts,Marcellan2017WHATPolynomial, Marcellan2015OnPolynomials, Martnez-Finkelshtein2001AnalyticRevisited, VanBuggenhout2023OnPolynomials} and references therein). While this is a fundamental question, an immediate problem one encounters is that these polynomials are not easy to construct in practice~\cite{Gautschi1995ComputingSpaces, VanBuggenhout2023OnPolynomials}, and that they depend nontrivially on the chosen value of $\Lambda$~\cite{Marcellan2015OnPolynomials}. This has led to many algebraic studies of these bases, but their use has remained limited and impractical in their original intended field of application: the resolution of differential equations. In 2015, Marcell\'an and Xu write in their seminal survey~\cite{Marcellan2015OnPolynomials}:

\begin{quote}``It is now time to go back to the beginning, studying the Fourier expansions and approximation by polynomials in Sobolev spaces, [...] to solve problems in other fields. We end this survey with this call of action.''
\end{quote}


It seems this call of action has yet to be heeded, presumably in part because these Sobolev orthogonal polynomials prove difficult to work with in practice for studying expansions, getting a priori truncation error estimates, and therefore solving differential equations. We propose in this work a very simple but seemingly unexplored alternative, which is to use a different scalar product on $H^1(\nu)$, rooted in the functional analysis of weighted Sobolev spaces. One of the motivations for our construction is to obtain orthogonal polynomials on $H^1(\nu)$ which relate in a simple way to the orthogonal polynomials $p_n$ on $L^2(\nu)$. A naive but natural approach would therefore be to take $\Lambda=0$ in~\eqref{eqn:usual-ip}, because then the antiderivatives of the $p_n$'s become orthogonal. However,~\eqref{eqn:usual-ip} is not a scalar product on $H^1(\nu)$ when $\Lambda=0$ (all the constants would have norm zero), which stems from the fact that there does not hold a classical Poincaré inequality on $H^1(\nu)$: the semi-norm $\|\partial_x \cdot\|_{L^2(\nu)}$ does not dominate the $L^2(\nu)$-norm $\|\cdot\|_{L^2(\nu)}$. Fortunately, for a wide class of potentials $V$ (see Proposition~\ref{prop:fonct-poinc}) a modified Poincaré inequality holds on $H^1(\nu)$:
\begin{equation}\label{eqn:poincaré-eq}\tag{PI}
\mbox{there exists }C_P>0 ,\qquad \int_{\R}u^2\d\nu - \left(\int_{\R}u \d \nu\right)^2\leq C_P\int_{\R}(u')^2\d \nu \quad\mbox{for all } u\in H^1(\nu).
\end{equation}
From now on, we will simply refer to~\eqref{eqn:poincaré-eq} as the Poincaré inequality on $H^1(\nu)$, a fundamental tool in the Functional Analysis of weighted Sobolev spaces and their associated equations and stochastic processes~\cite{Bonnefont2022PoincareDimension, Bakry2014AnalysisOperators, Pavliotis2014StochasticApplications}. As soon as~\eqref{eqn:poincaré-eq} holds, we can consider the following inner product on $H^1(\nu)$
\begin{equation}\label{eqn:good-inner-product}
    (u, v) := \int_{\R} u'v'\d\nu +\int_{\R}u\d\nu \int_{\R}v\d\nu,
\end{equation}
whose induced norm (henceforth denoted $\|\cdot\|_{H^1(\nu)}$) can be shown to be Lipschitz equivalent to the standard norm on $H^1(\nu)$; demonstrating that $(H^1(\nu),(\cdot, \cdot))$ is indeed a Hilbert space.

While this new inner product~\eqref{eqn:good-inner-product} is not of the form~\eqref{eqn:usual-ip} and may thus appear unusual at first sight, it is actually very natural; for instance, it reduces to the inner product $\langle \partial_x\cdot, \partial_x\cdot\rangle$ on the subspace of $\nu$-mean zero functions of $H^1(\nu)$. Moreover, the construction of orthogonal polynomials $\mathcal{Q} =\{q_n\}_{n\in\N}$ on $H^1(\nu)$ with respect to $(\cdot, \cdot)$ becomes very simple and elegant.
\begin{prop}\label{prop:Q-construct} 
    Assume the potential $V$ is such that~\eqref{eqn:poincaré-eq} holds. For all $n\geq 1$, the polynomials $\mathcal{Q} =\{q_n\}_{n\in\N}$ orthonormal with respect to $(\cdot,\cdot)$ are uniquely defined
    by the properties
    $$q_n' = p_{n-1} \qquad \textrm{and} \qquad \int_{\R}q_n\d\nu = 0.$$
\end{prop}
\begin{proof}
    The $q_n$'s are clearly orthonormal with respect to $(\cdot, \cdot)$, and the uniqueness (up to sign) follows from the Gram--Schmidt algorithm.
\end{proof}
While the basis $\mathcal{Q} =\{q_n\}_{n\in \N}$ is not orthogonal with respect to an inner product of the form~\eqref{eqn:usual-ip}, we will from now on refer to the $q_n$'s as Sobolev orthogonal polynomials.

Our goal is to use the bases $\mathcal{P}$ and/or $\mathcal{Q}$ in order to solve problems of the form~\eqref{eqn:Lu=f(u)}. The following properties show that these bases are indeed well suited, both for numerical and for theoretical purposes, even though they do not make $\cL$ diagonal.
\begin{prop}
\label{prop:decompo-intro}
    Let $\mathcal{P}^*= \cP \backslash\{1\}= \{p_n\}_{n\in \N^*}$ and $Q^* = \cQ \backslash\{1\}= \{q_n\}_{n\in \N^*}$ and let $D_{>0} = [\partial_x]_{\cP^*\to \cP}$ denote the Galerkin representation of $\partial_x$ with respect to $\cP^*$ and $\cP$, which is upper triangular. Then the Galerkin representations of $\mathcal{L} = V'\partial_x -\partial_{xx}$ have the following factorisations
    $$[\cL]_{\cP^*\to\cP^*} = D_{>0}^TD_{>0}, \qquad [\cL]_{\cQ^*\to\cQ^*} = D_{>0}D_{>0}^T,\qquad \text{and}\qquad[\mathcal{L}]_{\cQ^*\to\cP^*} = D_{>0}^T.$$
    Furthermore, if $V$ is an even polynomial of degree $2k$, then $D_{>0}$ is banded with $k$ non-zero superdiagonals, and its entries have an explicit dependence on the coefficients $(a_n)_{n\geq 0}$ of the three-term recurrence relation~\eqref{eqn:jacobi-rec-rel_intro} (with the $a_n$'s themselves satisfying a polynomial recurrence relation).
\end{prop}

The first decomposition, which provides us with the Cholesky (or LU) decomposition of $\cL$ with respect to $\cP$, is in fact simply the matrix reformulation of the well-known identity $\langle \cL u,v\rangle = \langle \partial_x u,\partial_x v\rangle$. The second one is more remarkable, as it provides us with a \emph{reverse Cholesky} (or UL) decomposition of $\cL$ with respect to $\cQ$. Moreover, in the case where $V$ is an even polynomial of degree $2k$, we show in this work that the banded triangular operator $D_{>0}$ can be inverted rather explicitly, provided we are able to quantify the growth of the coefficients $(a_n)_{n\geq 0}$ appearing in the three-term recurrence relation~\eqref{eqn:jacobi-rec-rel_intro}. This shall then enable the ``inversion'' of $\mathcal{L}$, estimation of the Poincaré constant $C_P$ in~\eqref{eqn:poincaré-eq} and more importantly to quantify the compactness of the embedding $H^1(\nu)\hookrightarrow L^2(\nu)$ (or equivalently of ``$\mathcal{L}^{-1}$'') which is crucial to justify the numerical resolution of PDEs of the form~\eqref{eqn:Lu=f(u)} on weighted Sobolev spaces. This establishes a \emph{direct connection} between the growth of the sequence $a_n$ (for which we actually have $a_n\asymp n^{1/{2k}}$~\cite{Deift1999StrongWeights}) and the optimal rate of compactness/approximation estimates on $H^1(\nu)$ (see Theorem~\ref{thm:compact}).

In this paper, for clarity, and motivated by the examples driving this work, we mostly focus on the case where $V$ is a quartic (or double-well) potential, i.e.
\begin{equation}
    \label{eq:Vquartic}
    V(x) = \frac{x^4}{4}-\kappa\frac{x^2}{2},\qquad \kappa \in \R.
\end{equation}
The treatment of this quartic case represents the first step departing from the classical paradigm (for which $D_{>0}$ is diagonal). We will show that this quartic case already leads to a rich theory bringing together several ideas in Numerical and Functional Analysis. However, we stress that analogous ideas also enable us to treat the general case of $V$ being a polynomial of degree $2k$. We elaborate on this in Section~\ref{sec:gen-poly-case}, where we show that inverting $D_{>0}$ and studying its entries reduces to proving the invertibility of a Toeplitz matrix, which we achieve by proving the positivity of a finite Taylor series.
\begin{figure}[h]
\begin{subfigure}{0.5\textwidth}
\includegraphics[width=0.9\linewidth]{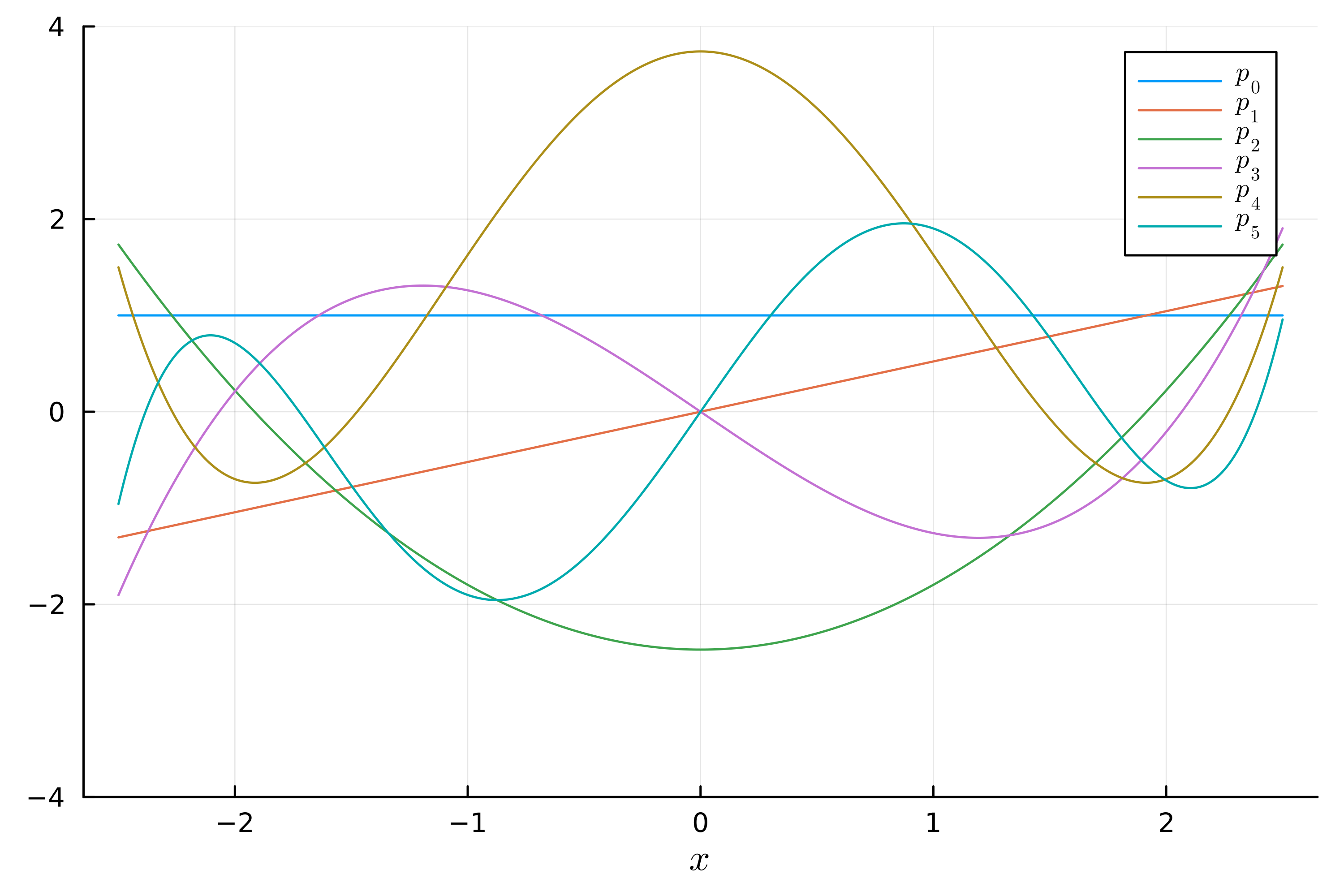} 
\caption{$p_0, \ldots, p_5$ orthonormal with respect to $\langle\cdot, \cdot\rangle$}
\end{subfigure}
\begin{subfigure}{0.5\textwidth}
\includegraphics[width=0.9\linewidth]{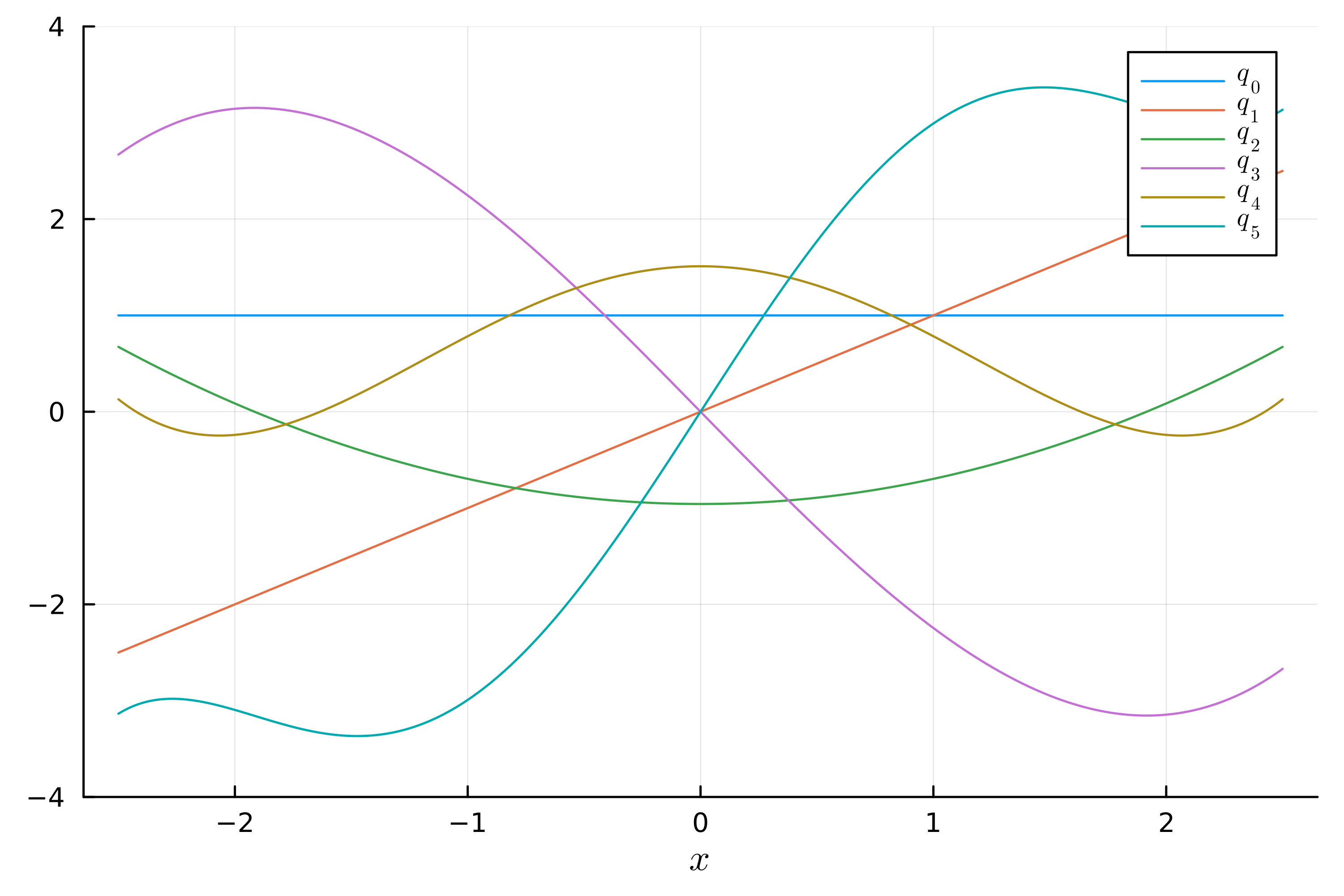} 
\caption{$q_0, \ldots, q_5$ orthonormal with respect to $(\cdot, \cdot)$}
\end{subfigure}
\caption{\centering Plots of orthonormal polynomials with respect to $\nu(\d x) = e^{-V(x)}\d x/\cZ$ with $V(x) = x^4/4-\kappa x^2/2$ and $\kappa = 4$.}
\end{figure}

With the quartic potential~\eqref{eq:Vquartic}, the operator $\mathcal{L} = -\partial_{xx} + V'\partial_x$ of Eq.~\eqref{eqn:Lu=f(u)} becomes
$$\mathcal{L} = (x^3-\kappa x)\partial_x -\partial_{xx},$$
and appears in several models of \emph{noise-induced phenomena} in Random Dynamics such as: the Hopf normal form with additive noise~\cite{Baxendale2025LyapunovNoise, DeVille2011StabilitySystem, Chemnitz2023PositiveNoise}, the Desai--Zwanzig model~\cite{Desai1978StatisticalModel} or the (non)-destruction of the pitchfork bifurcation by additive noise~\cite{BlessingNeamtu2025, Callaway2017TheNoise, Crauel1998AdditiveBifurcation}.

Our approach to give a thorough and precise Numerical Analysis on the Sobolev space $H^1(\nu)$ will rely on exploiting the structure of the bidiagonal differentiation operator $D_{>0}$, and some of our estimates are inspired from the related work~\cite{Breden2015RigorousPart} on tridiagonal operators. Our framework enables the computation of tight enclosures for the optimal Poincaré constant $C_P$ on $H^1(\nu)$, as stated in the following theorem.


\begin{thm}\label{thm:quant-poinc}
    Consider $V(x) = x^4/4 -\kappa x^2/2$ with $\kappa = 4$. The best Poincaré inequality constant $C_P$ in~\eqref{eqn:poincaré-eq} is given by
    \begin{equation}\label{eqn:cp-val}
        C_P = \|D_{>0}^{-1}\|_{\ell^2\to \ell^2}^2 = 33.58004242 \pm 2.3\times 10^{-7}.
    \end{equation}
\end{thm}
Note that, in general, even giving a non-sharp upper bound on $C_P$ is not trivial~\cite{Bakry2014AnalysisOperators,Bobkov1999IsoperimetricMeasures,Bonnefont2022PoincareDimension, Brascamp1976OnEquation, Kannan1995IsoperimetricLemma, Veysseire2010AManifolds} especially when $V$ is not convex (which is the case if $\kappa>0$). In contrast, in this particular case, we are able to give a tight enclosure for $C_P$.

In fact, we can do much better by considering projections of the operator $D_{>0}^{-1}$, and explicitly quantify the compactness of the embedding $H^1(\nu)\hookrightarrow L^2(\nu)$, a crucial estimate to justify the spectral approach with respect to the basis $\mathcal{Q} = \{q_n\}_{n\in \N}$.

\begin{thm}\label{thm:quant-compact}
    Consider $V(x) = x^4/4 -\kappa x^2/2$. There exist $C>0$ and $N\in \mathbb{N}$ such that for all $n\geq N$
    \begin{equation}\label{eqn:intro=cmopact-est}
        \| u\|_{L^2(\nu)}\leq \frac{C}{n^{3/4}}\|u\|_{H^1(\nu)}, \qquad \mbox{for all $u\in \overline{\mathrm{Span}\{q_j\}_{j>n}}^{H^1(\nu)}$}.
    \end{equation}
    In particular, if $\kappa = 4$, the above holds for $C = 1.63$ and $N = 50$. 
\end{thm}

\begin{rmk}
    The explicit value of $C$ obtained in Theorem~\ref{thm:quant-compact}, which influences the proofs of some statement to come, like Theorem~\ref{thm:GP1} and~\ref{thm:GP2}, is not sharp in the asymptotic regime $n\to\infty$. If need be, one could obtain a smaller constant $C$, by taking $N$ larger.
\end{rmk}

It is well-established~\cite{Allaire2007NumericalOptimization, Brenner2008TheMethods,Canuto1982ApproximationSpaces,Ciarlet2002TheProblems,Funaro1991ApproximationFunctions, Nakao1988AProblems} that estimates like Theorem~\ref{thm:quant-compact} are useful in the numerical resolution of partial differential equations (PDEs), as they are key to determining the required truncation level to approximate a differential equation by its finite Galerkin approximation. To the best of our knowledge, this is the first derivation of such a fully explicit compactness estimate with respect to a non-classical Sobolev orthogonal polynomial basis. Indeed, while in classical polynomial settings, these estimates can be obtained trivially from a diagonal differentiation operator (e.g.~see Example~\ref{ex:hermite}), much work is needed to obtain Theorem~\ref{thm:quant-compact}. 
As will be discussed below, our derivation of an explicit constant $C$ in Theorem~\ref{thm:quant-compact} relies on the fact that we focused on a specific potential $V$. However, most of the ideas used in the proof can be generalised for any polynomial potential $V$, allowing us to quantify the fact that the compactness of the embedding $H^1(\nu)\hookrightarrow L^2(\nu)$ becomes better as the potential $V$ grows faster at infinity.
\begin{thm}\label{thm:compact}
    Consider $V(x) = \sum_{j = 1}^k c_j\frac{x^{2j}}{2j}$ with $c_k=1$.
    There exists $C>0$ such that for all $n\geq 1$
    \begin{equation}\label{eqn:conj-compact}
        \| u\|_{L^2(\nu)}\leq \frac{C}{n^{\frac{2k-1}{2k}}}\|u\|_{H^1(\nu)}, \qquad \mbox{for all $u\in \overline{\mathrm{Span}\{q_j\}_{j>n}}^{H^1(\nu)}$}.
    \end{equation}
\end{thm}
\noindent We will see in the proof of Theorem~\ref{thm:compact} that the rate $n^{-\frac{2k-1}{2k}}$ is optimal.

%

All the previous theorems require understanding the growth of the coefficients $(a_n)_{n\in \N}$ in the three-term recurrence relation~\eqref{eqn:jacobi-rec-rel_intro}, which themselves solve a polynomial recurrence equation. In the case of $V$ being a quartic potential of the form~\eqref{eq:Vquartic}, the sequence $(b_n)_{n\in\N} = (a^2_n)_{n\in\N}$ solves the so-called \emph{discrete Painlevé I equation}~\cite{Bonan1984OrthogonalI}
\begin{equation}
    \frac{n}{b_n} = b_{n-1}+b_n +b_{n+1} - \kappa, \qquad n\geq 1.\label{eqn:intro-bread}
\end{equation}
A quick heuristic calculation might suggest that $b_n\sim \sqrt{n}/\sqrt{3}$ regardless of the initial conditions, but it turns out that~\eqref{eqn:intro-bread} has high sensitivity on initial conditions with a unique positive solution. Therefore, while the coefficients $a_n^2$ do in fact behave asymptotically like $\sqrt{n}/\sqrt{3}$, obtaining quantitative estimates is challenging. We nonetheless give a rigorous computational method to explicitly control the growth of this solution.
\begin{prop}\label{prop:bread-bounds}
    Let $(a^2_n)_{n\in\N}$ be the unique positive solution to the recurrence relation~\eqref{eqn:intro-bread}. For all $c^-<1$ and $c^+>1$, there exists $N_0\in \N$ such that for all $n\geq N_0$
    \begin{equation}\label{eqn:intro-bn_bounds}
    c^-\frac{\sqrt{n}}{\sqrt{3}}\leq a^2_n\leq c^+\frac{\sqrt{n}}{\sqrt{3}}.
    \end{equation}
    In particular, if $\kappa = 4$, $c^-=0.987$ and $c^+ = 1.177$, the above holds for $N_0 = 50$. 
\end{prop}

In summary, starting from the scalar product~\eqref{eqn:good-inner-product}, we obtain a natural and mathematically rigorous framework for the resolution of differential equations of the form~\eqref{eqn:Lu=f(u)}. We now present two examples, for which we actually used the bases $\cP$ and $\cQ$ in order to discretise the problem and obtain numerical solutions. To showcase that our approach does indeed provide an excellent theoretical basis beyond purely numerical experiments, we incorporate our estimates into computer-assisted proofs (CAPs)~\cite{vandenBerg2015RigorousDynamics,Nakao2019NumericalEquations,Gomez-Serrano2019Computer-assistedSurvey}, which yield the existence of a \emph{true} solution to~\eqref{eqn:Lu=f(u)} in a  \emph{certified} explicit neighbourhood of our approximate solutions.

\subsection{The Gross--Pitaevskii equation with sextic potential}

As a first example, we consider the Gross--Pitaevskii equation with sextic potential
\begin{equation}\label{eqn:para-GP}
    i\partial_t v = - \partial_{xx} v +W(x)v +|v|^2v,
\end{equation}
where $W(x) = x^6/4 - \kappa x^4/2 + cx^2 +d$. We look for standing wave solutions $v(t,x) = e^{i\omega t}\varphi(x)$ of Eq.~\eqref{eqn:para-GP} where $\varphi$ satisfies the elliptic semilinear equation

\begin{equation}\label{eqn:GP-intro}
-\partial_{xx}\varphi + W(x)\varphi +\varphi^3 = \omega \varphi,
\end{equation}
which can be reformulated as an equation of the form~\eqref{eqn:Lu=f(u)}.
Within our framework which makes use of Sobolev orthogonal polynomials and estimates such as~\eqref{eqn:intro=cmopact-est}, we not only obtain accurate approximate solutions to this equation, plotted in Figure~\ref{fig:GP}, but also prove the existence of \emph{true} solutions of Eq.~\eqref{eqn:GP-intro} nearby. The very small error bounds obtained below showcase that using Sobolev orthogonal polynomials to numerically solve such equations leads to very efficient numerical methods.

\begin{thm}
\label{thm:GP1}
    Let $c= 5/2$, $d=2$, $\kappa = 4$, $\omega = 8$ and $\bar{\varphi}_1:\R\to\R$ be the function whose restriction to the interval $[-5,5]$ is depicted in Figure~\ref{fig:GP}(a), and whose precise description is available at~\cite{Chu2025Huggzz/Freud} in the file \texttt{GP\_eq/ubar1}. There exists an even \emph{strictly positive} solution $\varphi^{\star}_1\in H^1(\R)$ to Eq.~\eqref{eqn:GP-intro} such that
    $$\|\varphi^{\star}_1- \bar{\varphi}_1\|_{H^1(\R)}\leq 2.9 \times10^{-100},$$
    where $\|\cdot\|_{H^1(\R)}$ is the standard $H^1$-norm on $\R$ equipped with the Lebesgue measure $\d x$, i.e.~$\|\varphi\|_{H^1(\R)} = \left(\int_\R \varphi^2 \d x + \int_\R (\varphi')^2 \d x\right)^{1/2}$. 
\end{thm}
\begin{thm}
\label{thm:GP2}
    Let $c= 5/2$, $d=2$, $\kappa = 4$, $\omega = 8$ and $\bar{\varphi}_2:\R\to\R$ be the function whose restriction to the interval $[-5,5]$ is depicted in Figure~\ref{fig:GP}(b), and whose precise description is available at~\cite{Chu2025Huggzz/Freud} in the file \texttt{GP\_eq/ubar2}. There exists an even solution $\varphi^{\star}_2\in H^1(\R)$ to Eq.~\eqref{eqn:GP-intro} such that
    $$\|\varphi^{\star}_2- \bar{\varphi}_2\|_{H^1(\R)}\leq 4.9\times 10^{-140}.$$
\end{thm}

\begin{figure}[h]
\begin{subfigure}{0.5\textwidth}
\includegraphics[width=0.9\linewidth]{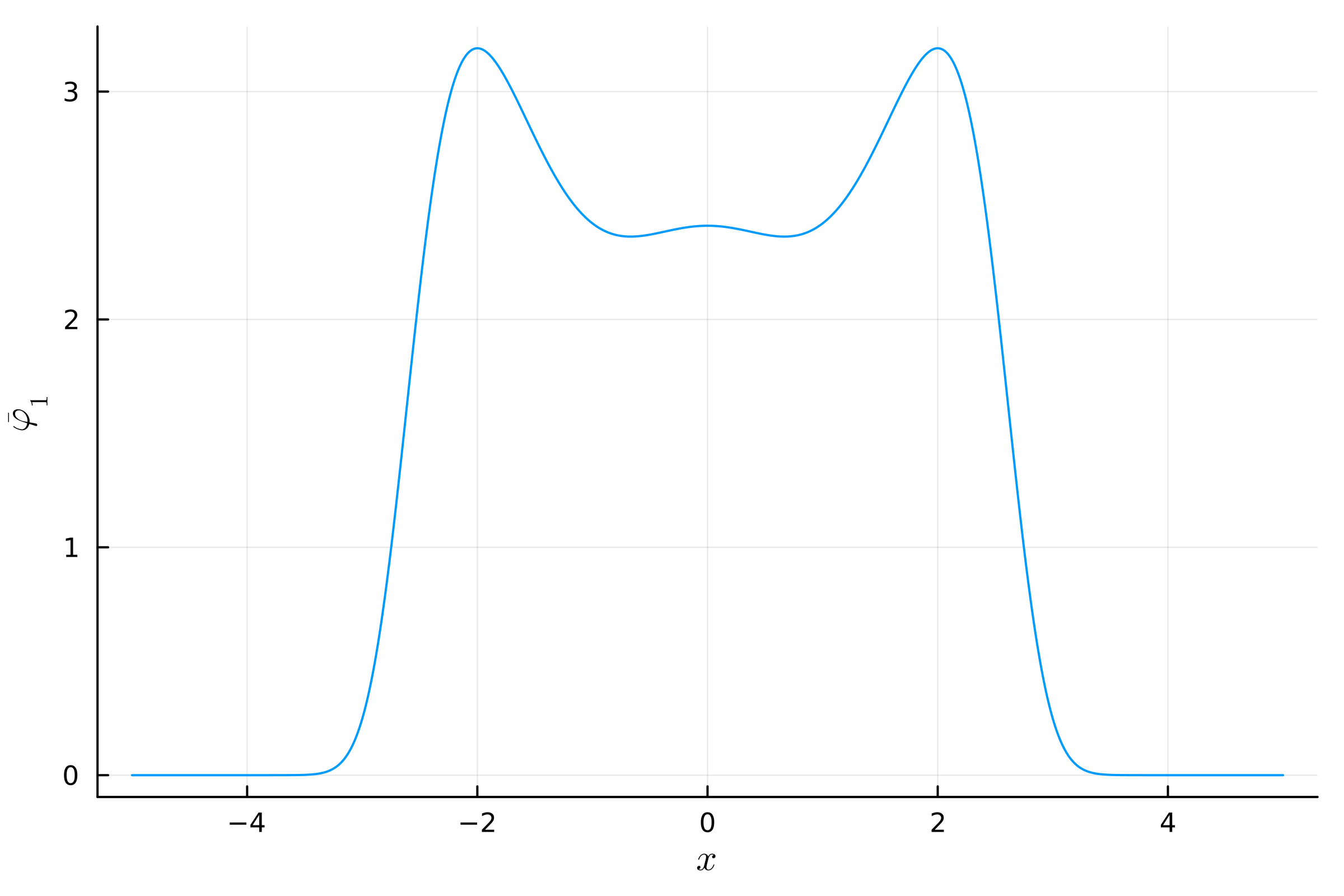} 
\caption{$\bar{\varphi}_1$ in Theorem~\ref{thm:GP1}}
\label{fig:GP1}
\end{subfigure}
\begin{subfigure}{0.5\textwidth}
\includegraphics[width=0.9\linewidth]{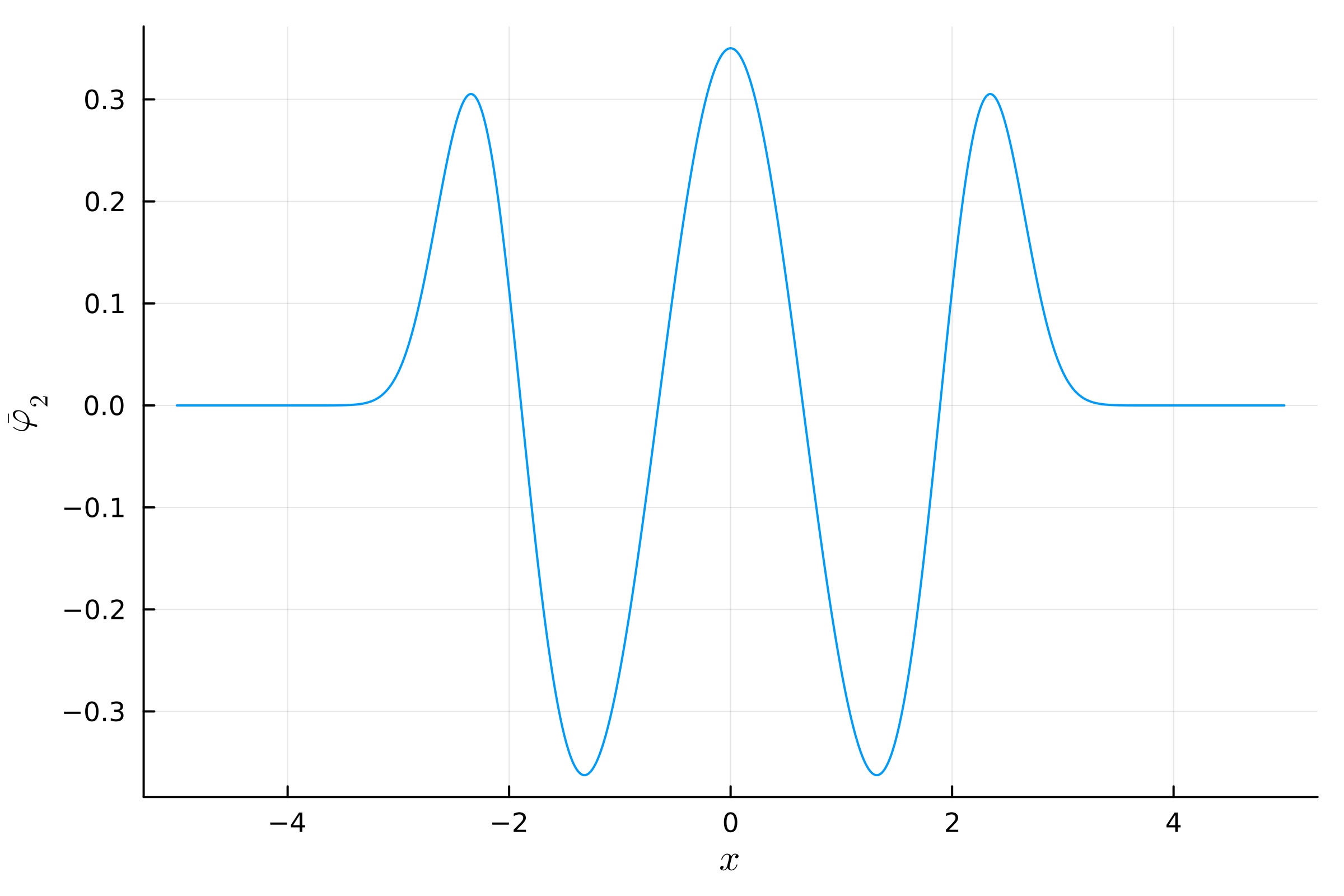} 
\caption{$\bar{\varphi}_2$ in Theorem~\ref{thm:GP2}}
\label{fig:GP2}
\end{subfigure}
\caption{Plots of the numerical approximations $\bar{\varphi}_1$ and $\bar{\varphi}_2$ in Theorems~\ref{thm:GP1} and~\ref{thm:GP2}.\label{fig:GP}}
\end{figure}

Such computer-assisted proofs have become more and more common in the study of differential equations, and we refer to the expository article~\cite{vandenBerg2015RigorousDynamics}, the surveys~\cite{Gomez-Serrano2019Computer-assistedSurvey,CAPD}, the books~\cite{Nakao2019NumericalEquations,Tuc11} and the references therein for an overview of the field. A notable feature of the computer-assisted proofs performed in this work is that they deal with problems on unbounded domains, which pose additional challenges (see~\cite[Part II]{Nakao2019NumericalEquations}, and the more recent~\cite{Breden2025ConstructiveRd,Cadiot2023RigorousMethods, Dahne2024Self-SimilarEquations, Donninger2026SelfSimilarEquation,Hou2025NonuniquenessEquation, vandenBerg2023ConstructiveRd} for some examples). Furthermore, while our proofs belong to the family of CAPs based on spectral methods, they feature the unusual twist that we work between two different bases $\cP = \{p_n\}_{n\in \N}$ and $\cQ = \{q_n\}_{n\in \N}$, and we have to deal with the fact that $\cL$ is not diagonal (even using different bases).  



\subsection{A computer-assisted proof of stochastic resonance}
\label{sec:intro-resonance}
As a second application, we turn to the study of the dynamics of the stochastic Duffing oscillator
\begin{equation}\label{eqn:resonance-intro} \d x_t =(-V'(x_t) +\eta\cos(\omega t))\d t + \sigma \d B_t, \qquad V(x) = \frac{x^4}{4}- \kappa \frac{x^2}{2},\end{equation}
i.e., a Brownian particle evolving in a double-well potential $V$ periodically tilted by the forcing $\eta \cos(\omega t)$. In the absence of noise ($\sigma = 0$), if the periodic forcing is weak ($\eta \ll \sqrt{\kappa}$), then the particle is restricted to remain in a single well. However, in their seminal papers~\cite{Benzi1981TheResonance, Benzi1983AChange}, Benzi, Parisi, Sutera and Vulpiani observe that for a well-chosen noise strength $\sigma$, the noise can amplify the weak periodic forcing instead of perturbing it, making the signal $(x_t)_{t\geq 0}$ appear almost periodic at an amplitude around $\sqrt{\kappa}\gg \eta$. They therefore named this noise-induced phenomenon \emph{stochastic resonance}.

Stochastic resonance has since been observed in other models, and equation~\eqref{eqn:resonance-intro} has been studied in various areas of application (see~\cite{McNamara1989TheoryResonance, Wiesenfeld1995StochasticSQUIDs} and the references therein). However, previous investigations of stochastic resonance have mostly relied on numerical simulations and informal perturbative arguments, and mathematically rigorous evidence for this phenomenon is still lacking.

The phenomenon of stochastic resonance is usually characterised via the periodic stationary probability density function $\varrho :(\T/\omega)\times\R\to\R$ solving the Fokker--Planck equation
\begin{equation}\label{eqn:FP-intro}
\pdiff{\varrho}{t} = -\pdiff{}{x}\Big((-V'(x)+\eta\cos\omega t)\varrho\Big)+\frac{\sigma^2}{2}\pdiff{^2\varrho}{x^2}\qquad \mbox{on $(\T/\omega)\times\R \simeq [0, 2\pi/\omega)\times \R$},
\end{equation}
which gives the asymptotic periodic statistics of $(x_t)_{t\geq 0}$. This is well illustrated in Figure~\ref{fig:resonance-density} which shows that, for a well-chosen value of $\sigma$, almost all of the probability mass of $(x_t)_{t\geq 0}$ is concentrated in the same well at a given time, and switches well periodically.
\begin{figure}[h]
\begin{subfigure}{0.32\textwidth}\centering
\includegraphics[height=0.9\linewidth]{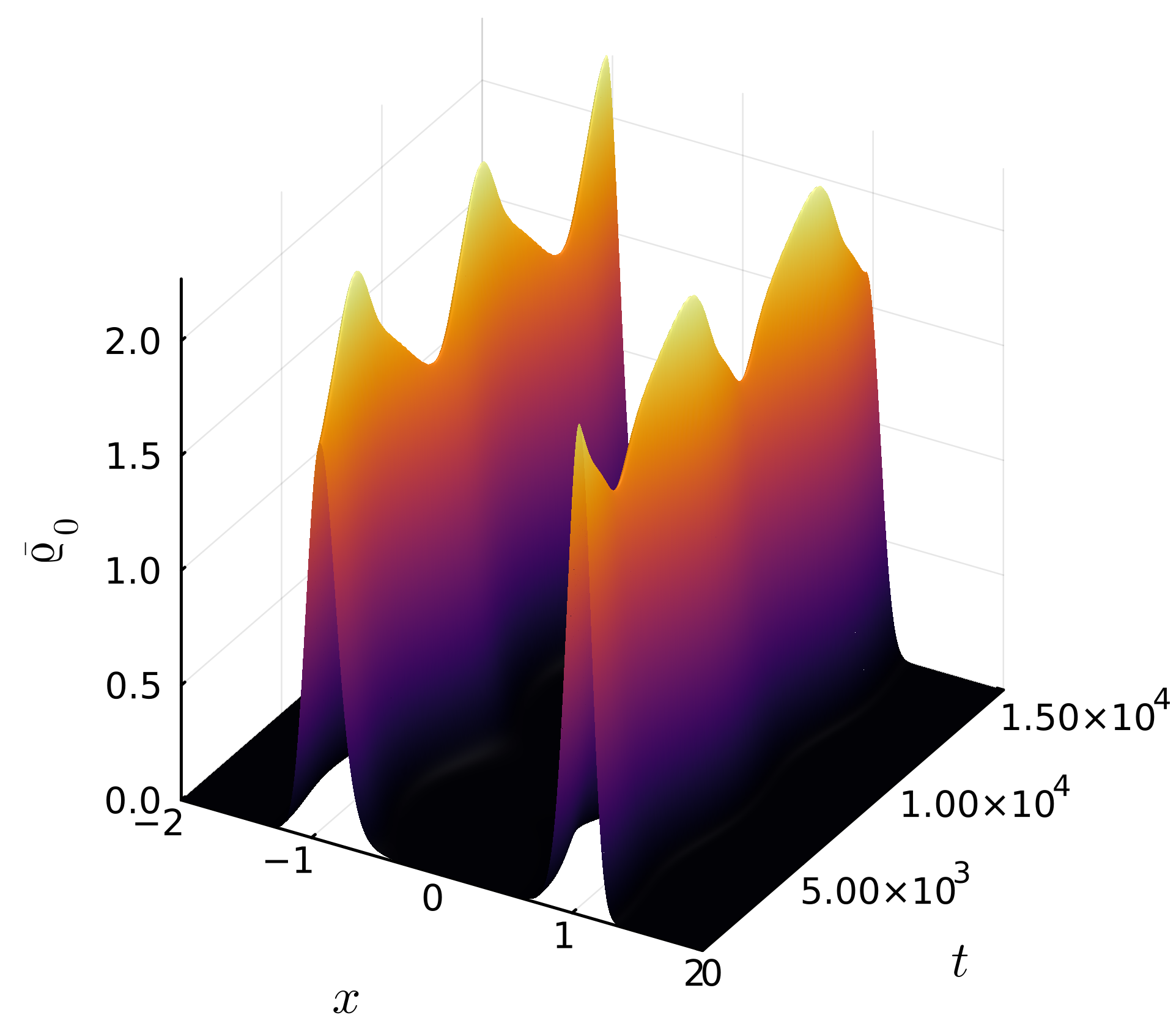} 
\caption{$\sigma_0 = 0.2$.}
\label{fig:resonance-paths-left}
\end{subfigure}
\begin{subfigure}{0.32\textwidth}\centering
\includegraphics[height=0.9\linewidth]{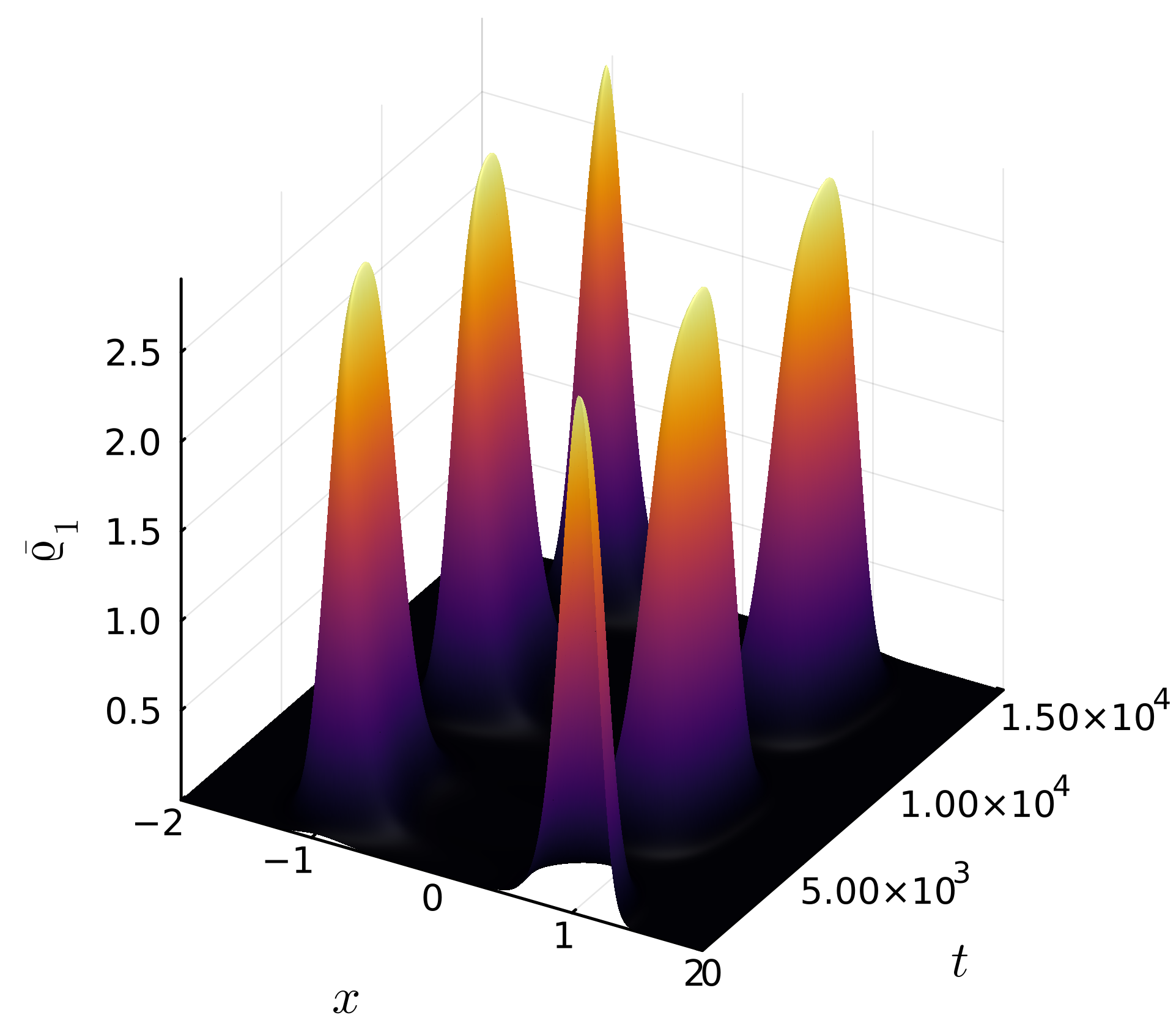} 
\caption{$\sigma_1 = 0.287129152$.}
\label{fig:resonance-density-middle}
\end{subfigure}
\begin{subfigure}{0.32\textwidth}\centering
\includegraphics[height=0.9\linewidth]{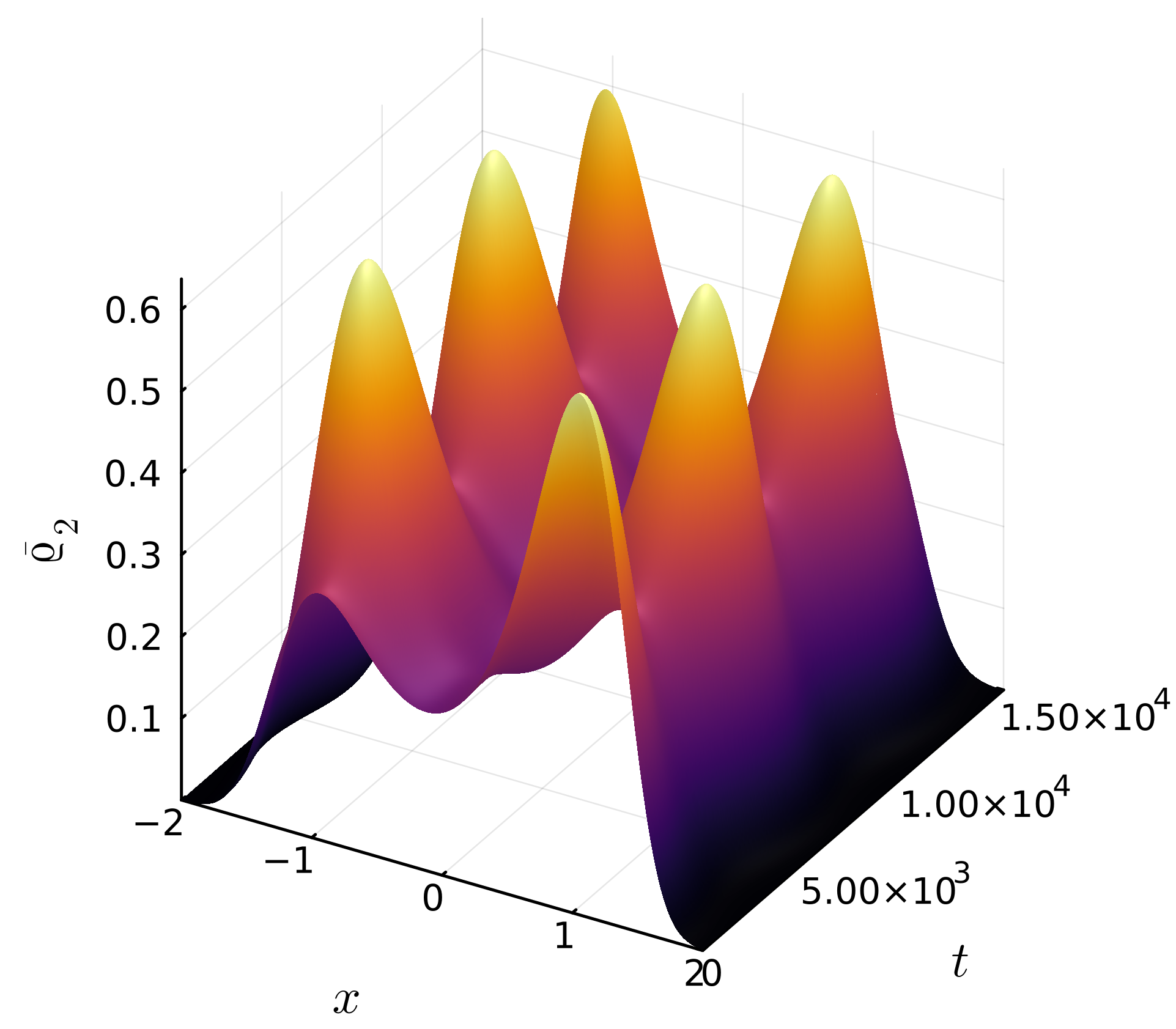} 
\caption{$\sigma_2 = 0.8$.}
\label{fig:resonance-density-right}
\end{subfigure}
\caption{\centering Numerical periodic solutions to the Fokker--Planck equation~\eqref{eqn:FP-intro} for $\eta = 0.12$, $\kappa = 1$ and $\omega = 0.001$ and various values of $\sigma$. Stochastic resonance is strongest for $\sigma$ around $\sigma_1 = 0.287129152$.\label{fig:resonance-density}}
\end{figure}

We therefore aim to demonstrate rigorously the occurrence of stochastic resonance in the system~\eqref{eqn:resonance-intro} by computing the density $\varrho$ with explicit error bounds. However, the Numerical Analysis of Eq.~\eqref{eqn:FP-intro} is more intricate than the one of~\eqref{eqn:GP-intro} given the extra dependence in the time variable: now, we need to derive compactness estimates in space and time simultaneously. Nevertheless, we show that the core ideas behind the proof of Theorem~\ref{thm:quant-compact} carry over and we obtain:
\begin{thm}\label{thm:rho-intro} For $i=0, 1, 2$, let $\sigma_i$ be as in Figure~\ref{fig:resonance-density}, let $\varrho_i :(\T/\omega)\times \R \to \R$ be the unique probability solution of Eq.~\eqref{eqn:FP-intro} with $\kappa = 1$, $\eta = 0.12$, $\omega = 0.001$, $\sigma = \sigma_i$, and let $\bar{\varrho}_i : (\T/\omega)\times \R \to \R$ be the function depicted in Figure~\ref{fig:resonance-density} (and whose precise description is available at~\cite{Chu2025Huggzz/Freud} in the file~\texttt{stoc\_res/u\_bars\_plots}). Then, for $i=0, 1, 2$
$$\sup_{(t, x)\in (\T/\omega)\times\R}|e^{V/2}(\varrho_i-\bar{\varrho}_i)|\leq 3 \times 10^{-22}.$$
\end{thm}

In Section~\ref{sec:stoc-res}, we give a more detailed description of the phenomenon of stochastic resonance, together with a rigorous quantitative statement concerning the range for $\sigma$ at which its effect is maximised.

Theorem~\ref{thm:rho-intro} amounts to a computer-assisted proof of time-periodic solutions of the Fokker--Planck equation~\eqref{eqn:FP-intro}. We note that, while the literature on computer-assisted proofs for PDEs contains several techniques allowing to study time-periodic solutions~\cite{AriKoc10bis,CasGamLes18,FigLla17,GamLes17,BerBreLesVee21,BerQue22,Zgl04,Zgl10}, all these works rely quite heavily on the fact that they only deal with solutions which are also periodic in space (because the spatial domain is either a torus, or a compact interval with boundary conditions and symmetries allowing for a periodic expansion of the solution). In contrast, in this work we deal with a problem posed on the whole real line. In particular, we use a Galerkin representation of the problem with Fourier series in time and the $L^2(\nu)$-orthogonal basis $\mathcal{P} = \{p_n\}_{n\in \N}$ in space, rather than using space-time Fourier series as in several of the aforementioned works, which makes the derivation of the estimates required for the computer-assisted proof significantly more challenging.

\subsection*{Organisation of the paper}

In Section~\ref{sec:weighted-Sobolev-spaces}, we first recall some results about weighted Sobolev spaces and Poincaré inequalities, and then discuss our construction of Sobolev orthogonal polynomials and the structure of operators with respect to orthogonal polynomial bases. Building on these insights, we focus in Section~\ref{sec:freud-sobolev} on the case of the quartic potential $V(x) = x^4/4-\kappa x^2/2$, and prove embedding estimates for $H^1(\nu)\hookrightarrow L^2(\nu)$, i.e. Theorems~\ref{thm:quant-poinc} and~\ref{thm:quant-compact}, as well as Proposition~\ref{prop:bread-bounds} on which they rely crucially. The case of an arbitrary even polynomial potential is then discussed in Section~\ref{sec:gen-poly-case}, where we prove Theorem~\ref{thm:compact}. In Section~\ref{sec:GP_eq}, we make use of the estimates of Theorems~\ref{thm:quant-poinc} and~\ref{thm:quant-compact} to solve the Gross--Pitaevskii equation with sextic potential and prove Theorems~\ref{thm:GP1} and~\ref{thm:GP2}. Finally, Section~\ref{sec:stoc-res} is dedicated to our study of stochastic resonance. It contains some more background on this problem, and several quantitative statements complementing Theorem~\ref{thm:rho-intro}. Some of the most technical part of the required estimates for the computer-assisted proofs are collected in the Appendix.

\section{Weighted Sobolev spaces}\label{sec:weighted-Sobolev-spaces}

We start by recalling known results about the operator $\cL$ and the associated weighted Sobolev spaces in Section~\ref{sec:FA}. Then, we construct orthogonal polynomials on $H^1(\nu)$ in Section~\ref{sec:Sobolevpoly}, and discuss some generalisations and connections with existing works in Section~\ref{sec:gene}. Finally, in Section~\ref{subsec:diff-op}, we express the differential operator $\partial_x$ and $\cL$ in these orthogonal polynomial bases. 

\subsection{The Functional Analytic point of view}
\label{sec:FA}

The present work is motivated by the study of the overdamped Langevin (or Smoluchowski) dynamics

\begin{equation}\label{eqn:sde}
    \d X_t = -V'(X_t) \d t+\sqrt{2}\d B_t,
\end{equation}
for a potential $V\in \mathcal{C}^2(\R)$ which induce a reversible diffusion process $(X_t)_{t\geq 0}$ with stationary probability (Gibbs) measure
$$\nu(\d x) = \frac{e^{-V(x)}}{\cZ}\d x, \qquad \cZ = \int_{\R}e^{-V}\d x,$$
where $V$ and $\nu$ will always be of this form unless stated otherwise.
The Markov semi-group of $(X_t)_{t\geq 0}$ is generated by $-\mathcal{L}$, where
$$\mathcal{L} = V'(x)\partial_x -\partial_{xx}$$
is a positive self-adjoint operator with respect to the $L^2$-inner product $\langle \cdot, \cdot \rangle$. Indeed, by parts, one has that
\begin{equation}\label{eqn:adjoint-L}
    \langle u, \mathcal{L}v\rangle = \int_{\R} u (\mathcal{L}v)\d \nu = \int_{\R}u'v'\d\nu = \int_{\R}(\mathcal{L}u)v\d \nu = \langle \mathcal{L} u, v\rangle, \qquad \mbox{for all $u, v\in H^1(\nu)$,} 
\end{equation}
which can thus also be written as a Dirichlet form. An accessible introduction to this theory can be found in~\cite[§4]{Pavliotis2014StochasticApplications} (see also~\cite{Bonnefont2022PoincareDimension} for a nice introduction to Poincaré inequalities) but we emphasise that various aspects of this theory can be generalised~\cite{Bakry1994LhypercontractiviteSemigroupes, Bakry2014AnalysisOperators}, to higher dimensions and to various geometries. 

A central tool in this theory is the Poincaré inequality~\eqref{eqn:poincaré-eq} on $H^1(\nu)$.
For instance~\cite[Theorem 4.4]{Pavliotis2014StochasticApplications}, such a Poincaré inequality immediately implies the exponential ergodicity of the process $(X_t)_{t\geq 0}$ evolving under~\eqref{eqn:sde}. A criterion for the Poincaré inequality~\eqref{eqn:poincaré-eq} to hold and which will be sufficient for our purposes is given by the following proposition. See~\cite{Bonnefont2022PoincareDimension} for various functional analytic aspects related to this Poincaré inequality.
\begin{prop}\label{prop:fonct-poinc}
    Suppose that the potential $V$ is such that
    \begin{equation}\label{eq:assumption_V}
        \lim_{x\to\pm\infty} \frac{(V'(x))^2}{2} - V''(x) = +\infty,
    \end{equation}
    then there holds a Poincaré inequality of the form~\eqref{eqn:poincaré-eq} on $H^1(\nu) = H^1(e^{-V}/\cZ)$.
\end{prop}

\begin{proof}
    See~\cite[Theorem A.1]{Villani2009Hypocoercivity}.
\end{proof}

This criterion essentially amounts to $V$ having to grow strictly faster than linearly at infinity. Note that a more sophisticated way to verify the Poincaré inequality is via the Muckenhoupt criterion~\cite[Theorem 4.5.1]{Bakry2014AnalysisOperators} which is actually an equivalent statement and gives (large) bounds on the optimal Poincaré constant $C_P$ (see also~\cite[§A.19]{Villani2009Hypocoercivity}).

\begin{ex}\label{ex:hermite}
    Consider the case of the Gaussian measure, i.e. $V(x) = x^2/2$ such that $\nu(\d x) = e^{-x^2/2}\d x/\sqrt{2\pi}$. Then, we have the following Poincaré inequality, for all $u\in H^1(\nu)$,
    \begin{equation*}
        \int_{\R}u^2\d \nu - \left(\int_{\R}u\d \nu\right)^2 \leq \int_{\R}(u')^2 \d \nu.
    \end{equation*}
    This is easily shown by decomposing $u$ in terms of the orthonormal Hermite polynomials $\mathcal{P} = \{p_n\}_{n\in\mathbb{N}}$~\cite[§22]{Abramowitz1970HandbookSeries}. One can actually improve this further and obtain the compactness estimates
    $$\|u\|_{L^2(\nu)}\leq \frac{1}{\sqrt{n}}\|u\|_{H^1(\nu)},$$
    for all $u \in \overline{\mathrm{Span}\{p_j\}_{j\geq n}}^{H^1(\nu)}$.
\end{ex}
The Hermite case is actually one of the few settings for which such fully explicit constructive compactness estimates had so far been proved, thanks to the diagonal nature of the operator $\partial_x$ when represented with respect to $\mathcal{P}$ (see~\cite{Bakry2003CharacterizationPolynomials,Bakry2022OrthogonalOperators,Mazet1997ClassificationOrthogonaux}). This is a shame, as from the probability theoretic point of view, the Hermite case (associated to the so-called Ornstein--Uhlenbeck process) is a limiting case in the classification of diffusion semi-groups generated by $-\mathcal{L}$ and associated to $\nu$ (see~\cite[§7.7]{Bakry2014AnalysisOperators} for a discussion, and also~\cite{Bakry1994LhypercontractiviteSemigroupes, Kavian1993SomeUltracontractivity}). Nonetheless, the embedding $H^1(\nu)\hookrightarrow L^2(\nu)$ is known to be compact in a much broader setting.
\begin{thm}\label{thm:fonct-compact}
    If $V$ satisfies~\eqref{eq:assumption_V}, $H^1(\nu) = H^1(e^{-V}/\cZ)$ is compactly embedded in $L^2(\nu)$ and $\mathcal{L}$ has a discrete spectrum.
\end{thm}
\begin{proof}
    This is a standard result, although typically not stated this way. Assumption~\eqref{eq:assumption_V} implies that the Witten Laplacian 
    \begin{align*}
        \Delta_V := -\partial_{xx} + \frac{(V'(x))^2}{2} - V''(x)
    \end{align*}
    has compact resolvent on $L^2(\R)$ (see, e.g.,~\cite[Theorem 3.1]{Berezin1991TheEquation}), and thus so does $\cL = e^{V/2}\Delta_V e^{-V/2}$ on $L^2(\nu)$. The compactness of the embedding $H^1(\nu)\hookrightarrow L^2(\nu)$ then follows from ~\eqref{eqn:adjoint-L}.
\end{proof}
As shown in Theorem~\ref{thm:compact} in the polynomial case, we expect that if $V(x) \sim |x|^p$ as $x\to\infty$ for $p\in(1,\infty)$, then the embeddings should get quantitatively better as $p$ increases. We note that in the limiting case $p=1$, while the Poincaré inequality~\eqref{eqn:poincaré-eq} holds, the compactness of the embedding $H^1(\nu)\hookrightarrow L^2(\nu)$ does not (see~\cite[§4.4.1]{Bakry2014AnalysisOperators}).

\begin{rmk}\label{rmk:poinc0}
    The Poincaré inequality~\eqref{eqn:poincaré-eq} is in fact equivalent to the standard Poincaré inequality on the space of $\nu$-mean-zero functions: Denoting
    $$H^1_0(\nu) := \left\{u \in H^1(\nu)\Big|\int_{\R}u\d\nu = 0\right\},$$
    we have 
        \begin{equation}\label{eqn:rmk-poinc-ineq}
        \int_{\R}u^2\d \nu \leq C_P\int_{\R}(u')^2 \d \nu \quad \text{for all }u\in H^1_0(\nu)
    \end{equation}
    if and only if~\eqref{eqn:poincaré-eq} holds.
\end{rmk}

In summary, once the Poincaré inequality~\eqref{eqn:poincaré-eq} is available, all the Functional Analysis on $H^1(\nu)$ is reduced to textbook analysis of elliptic PDEs on Hilbert spaces~\cite{Brezis2011FunctionalEquations, Evans2010PartialEquations}. 
In particular, the problem $\cL u = g$ is then well-posed for all $g\in L^2(\nu)$ with $\nu$-zero mean by the Lax--Milgram Theorem.

\subsection{Sobolev orthogonal polynomials}
\label{sec:Sobolevpoly}
A problem naturally arising in Numerical Analysis is the construction of \emph{Sobolev orthonormal polynomials} on the Sobolev space $H^1(\nu)$ and their use in the resolution of partial differential equations~\cite{Marcellan2015OnPolynomials}. 

While orthogonal polynomials with respect to ``standard'' inner products of the form~\eqref{eqn:usual-ip} (and to variants including discrete measures) have been heavily studied, actually using them for the Numerical Analysis of PDEs seems very challenging (except in the classical case $V(x) = x^2/2$). We propose to consider a different inner product, which leads to a very natural construction of orthogonal polynomials on $H^1(\nu)$.

\begin{thm}\label{thm:inner-product}
    Suppose that $\nu$ is a probability measure such that the Poincaré inequality~\eqref{eqn:poincaré-eq} holds.
    Then the inner product $(\cdot, \cdot)$ defined by
    $$(u,v) = \int_{\R}u'v'\d \nu+\int_{\R}u\d \nu\int_{\R}v\d \nu\qquad u, v\in H^1(\nu)$$
    defines a Hilbert structure on $H^1(\nu)$ inducing its topology. In particular, $H^1(\nu)$ is complete under the induced norm given by $\|\cdot\|_{H^1(\nu)}:= \sqrt{(\cdot, \cdot)}$.
\end{thm}
\begin{proof}
It is straightforward to check that $(\cdot, \cdot)$ is indeed an inner product on $H^1(\nu)$. Moreover, for all $u\in H^1(\nu)$,
    \begin{align*}
        \left( \int_\R u \d \nu\right)^2 \leq \int_\R u^2 \d \nu \leq \max(1, C_P) \left\{\int_{\R}(u')^2\d\nu+\left(\int_{\R}u\d\nu\right)^2\right\},
    \end{align*}
    where the second inequality follows from the Poincaré inequality~\eqref{eqn:poincaré-eq}, and thus $\|\cdot\|_{H^1(\nu)}= \sqrt{(\cdot, \cdot)}$ is Lipschitz-equivalent to the standard Sobolev $H^1(\nu)$-norm $\sqrt{\langle \partial_x\cdot,\partial_x \cdot\rangle + \langle \cdot, \cdot\rangle}$.
\end{proof}

We recall that $\cQ = \{q_n\}_{n\in\N}$ denotes the basis of orthonormal polynomials with respect to $(\cdot, \cdot)$, which are characterised by Proposition~\ref{prop:Q-construct}. From now on, we refer to them as Sobolev orthogonal polynomials, even though they are not obtained with a scalar product of the form~\eqref{eqn:usual-ip}. We show in Section~\ref{subsec:diff-op} and Section~\ref{sec:Vpoly} that these Sobolev orthogonal polynomials have valuable properties that make them helpful for obtaining compactness estimates and solving differential equations.

\begin{rmk}
Based on Remark~\ref{rmk:poinc0}, one can think of the construction of the mean-zero Sobolev orthonormal polynomials $\{q_n\}_{n=1}^{\infty}$ as the orthogonalisation of the basis $\{p_n\}_{n=1}^{\infty}$ of $H^1_0(\nu)$ with respect to the inner product
$$\int_{\R}u'v'\d \nu,\qquad \mbox{for }u, v\in H^1_0(\nu),$$
which is another way to show Proposition~\ref{prop:Q-construct}.
\end{rmk}

\begin{rmk}\label{rmk:density}
    Note that the density of the Sobolev orthogonal polynomials $\{q_n\}_{n\in \N}$ in $H^1(\nu)$ follows directly from the density of the $L^2$-orthogonal polynomials $\{p_n\}_{n\in \N}$ in $L^2(\nu)$.
\end{rmk}

\begin{rmk}
    Our Sobolev orthogonal polynomials $\mathcal{Q} = \{q_n\}_{n\in \N}$ can be obtained as the limit as $\Lambda\to 0$ of the orthonormal polynomials with respect to the standard inner product~\eqref{eqn:usual-ip} (since these have $\nu$-zero mean if they are not constant).
\end{rmk}

\begin{rmk}
One could add weights $\Lambda_1,\Lambda_2>0$ and consider seemingly more general inner products than $(\cdot,\cdot)$, of the form
$$\Lambda_1\int_{\R}u'v'\d \nu+\Lambda_2\int_{\R}u\d \nu\int_{\R}v\d \nu\qquad u, v\in H^1(\nu).$$
However, up to normalisation, the obtained $q_n$'s do not depend on such weights, which is not the case for inner products of the form~\eqref{eqn:usual-ip}.
\end{rmk}

\begin{ex}
Consider the case of the Gaussian measure, i.e.~$V(x) = x^2/2$ such that $\nu(\d x) = e^{-x^2/2}\d x/\sqrt{2\pi}$. Then $\partial_x$ acts very simply on the basis $\{p_n\}$, namely $p_n' = \sqrt{n}\, p_{n-1}$ for all $n\geq 1$ and therefore $q_n = p_n/\sqrt{n}$ for $n\geq 1$: In that case, the Sobolev orthonormal polynomials are simply renormalised classical orthogonal polynomials. 
\end{ex}

We emphasize that the previous example is the exception and not the rule: in general, $q_n$ need not be proportional to $p_n$, and indeed this will not be the case in Section~\ref{sec:Vpoly} where we work with non-classical weights.


\subsection{Some generalisations of our Sobolev orthogonal polynomials}
\label{sec:gene}

The main focus of this paper is on measures of the form $\nu(\d x) = e^{-V}\d x/\cZ$ supported on the whole real line, with $V$ satisfying~\eqref{eq:assumption_V}. However, even for other measures, our construction of Sobolev orthogonal polynomials can be used as soon as a form of Poincaré inequality is at hand, and in some specific cases it produces back classical orthogonal polynomials. 

\begin{ex}[The Legendre case]\label{ex:Legendre}
    Consider $\d\nu = \d x/2$ on $[-1,1]$, where the orthonormal polynomials are the Legendre polynomials~\cite[§22]{Abramowitz1970HandbookSeries}. While this example does not sit within the general setting of this paper, it satisfies the assumptions of Theorem~\ref{thm:inner-product}. Indeed on $H^1(\nu)$, we have the Poincaré--Wirtinger inequality~\cite[§5.8.1]{Evans2010PartialEquations} (see also~\cite[§4.5.2]{Bakry2014AnalysisOperators})
    \begin{equation*}
        \int_{-1}^1u^2\d \nu - \left(\int_{-1}^1u\d \nu\right)^2 \leq C_P\int_{-1}^1(u')^2 \d \nu, \qquad C_P = \frac{4}{\pi^2}.
    \end{equation*}
    The Sobolev orthogonal polynomials $\cQ = \{q_n\}_{n\in \N}$ associated to the scalar product~\eqref{eqn:good-inner-product} are then given by
    $$q_0 = 1, \qquad q_1 = \frac{p_1}{\sqrt{3}},\qquad q_2 =\frac{p_2}{\sqrt{15}},\qquad q_n = \frac{1}{\sqrt{2n-1}}\left(\frac{p_n}{\sqrt{2n+1}}-\frac{p_{n-2}}{\sqrt{2n-3}}\right)\quad\mbox{for $n\geq 3$.}$$
    Note that $q_n(-1) = q_n(1) = 0$ for all $n\geq 3$, 
    which furthermore shows that the $q_n$'s cannot simply be some (even rescaled) Jacobi polynomials, as those have interlacing zeros in $(-1,1)$. For $n\geq 3$, the $q_n$'s 
    in fact coincide with some so-called Legendre associated polynomials~\cite[§8]{Abramowitz1970HandbookSeries}, which are the natural extension of the Jacobi polynomials at the parameters $\alpha = \beta = -1$.
\end{ex}

Note that it might sometimes be more natural to construct $H^1$-Sobolev orthogonal polynomials $\{q_n\}_{n\in\N}$ with respect to an inner product with two different measures $\nu$ and $\tilde{\nu}$
\begin{equation}\label{eqn:2-meas-ip}
    \int_{\R}u'v'\d \nu +\int_{\R}u\d\tilde{\nu}\int_{\R}v\d\tilde{\nu},
\end{equation}
if there holds a Poincaré inequality of the form
\begin{equation}\label{eqn:diff-poinc-ineq}\mathrm{Var}_{\tilde{\nu}}(u) = \int_{\R}u^2\d \tilde{\nu} - \left(\int_{\R}u\d \tilde{\nu}\right)^2\leq C_P \int_{\R}(u')^2\d\nu,\qquad \mbox{for all $u\in L^2(\tilde{\nu})\cap H^1(\nu)$}.\end{equation}
Such an inequality can be verified using Muckenhoupt's criterion (see~\cite[Chapter 4.5.1]{Bakry2014AnalysisOperators} and~\cite{Muckenhoupt1972HardysWeights}). This naturally extends our construction of Sobolev orthogonal polynomials, which coincide with known orthogonal polynomials if the measures $\nu$ and $\tilde{\nu}$ belong to classical families. For instance, in the case of beta distributions we have the following. 

\begin{ex}[The Jacobi family]
    Let $\alpha, \beta>0$ and denote $\nu_{\alpha, \beta}(\d x) =(1-x)^{\alpha}(1+x)^{\beta}\d x/Z_{\alpha, \beta}$ on $[-1,1]$, for which the $p_n$'s are (orthonormal) Jacobi polynomials of parameter $(\alpha, \beta)$~\cite[§22]{Abramowitz1970HandbookSeries}. Then, on $H^1(\nu_{\alpha, \beta})\cap L^2(\nu_{\alpha-1, \beta-1})=H^1(\nu_{\alpha, \beta})$, there holds a Poincaré inequality
    $$\int_{-1}^1 u^2 \d\nu_{\alpha-1, \beta-1} - \left(\int_{-1}^1 u \d\nu_{\alpha-1, \beta-1}\right)^2 \leq C_P \int_{-1}^1 (u')^2 \d\nu_{\alpha, \beta},\qquad  C_P = \alpha +\beta.$$
    See~\cite[§2.7.4]{Bakry2014AnalysisOperators} for more details about the associated semi-group. Then, the $q_n$'s associated to the inner product~\eqref{eqn:2-meas-ip} (with $\nu = \nu_{\alpha, \beta}$ and $\tilde{\nu} = \nu_{\alpha-1,\beta-1}$) are nothing but the (rescaled) Jacobi polynomials associated to $\nu_{\alpha-1,\beta-1}$. Observe that the Legendre-Sobolev polynomials on $H^1(\nu_{0,0})$ (from Example~\ref{ex:Legendre}), cannot be constructed via this inner product as we require $\alpha,\beta>0$. 
\end{ex}

Similarly, concerning the gamma distributions, we have the following.

\begin{ex}[The Laguerre family]
    For $\alpha>-1$, denote $\nu_{\alpha} = x^{\alpha}e^{-x}\d x/Z_{\alpha}$ on $\R_+= [0,\infty)$ such that the $p_n$'s are orthonormal Laguerre polynomials of parameter $\alpha$~\cite[§22]{Abramowitz1970HandbookSeries}.
    \begin{itemize}
        \item If $\alpha>0$, on $H^1(\nu_{\alpha})\cap L^2(\nu_{\alpha -1})=H^1(\nu_{\alpha})$, we have the Poincaré inequality~\cite[§4.4.1]{Bakry2014AnalysisOperators}
        $$\int_{0}^{\infty} u^2 \d\nu_{\alpha-1} - \left(\int_{0}^{\infty} u \d\nu_{\alpha-1}\right)^2 \leq \int_{0}^{\infty} (u')^2 \d\nu_{\alpha},$$
        such that the $q_n$'s associated to the inner product~\eqref{eqn:2-meas-ip} (with $\nu = \nu_{\alpha}$ and $\tilde{\nu} = \nu_{\alpha-1}$) are (rescaled) Laguerre polynomials of parameter $(\alpha-1)$.
        \item In the case $\alpha = 0$, we have the Poincaré inequality~\cite[§4.4.1]{Bakry2014AnalysisOperators}
        $$\int_{0}^{\infty} u^2 \d\nu_{0} - \left(\int_{0}^{\infty} u \d\nu_{0}\right)^2 \leq 4\int_{0}^{\infty} (u')^2 \d\nu_{0}.$$
        In a similar way as in the Legendre case, for $n\geq 2$, the $q_n$'s correspond to a natural extension of the Laguerre polynomials at the parameter $\alpha = -1$ (which are not orthogonal with respect to a probability measure). They are for instance obtained by generalising the closed form of the Laguerre polynomials. 
    \end{itemize}
\end{ex}


\begin{rmk}[On coherent pairs] Let $\{\tilde{p}_n\}_{n\in \N}$ denote the $L^2(\tilde{\nu})$-orthonormal polynomials. If for all $n\in \N^*$, $p_n$ is a linear combination of only $\tilde{p}'_{n+1}$ and $\tilde{p}'_{n}$ (resp. $\tilde{p}'_{n-1}$), the measures $\tilde{\nu}$ and $\nu$ are said to be \emph{coherent} (resp. symmetrically coherent). This notion was introduced in~\cite{Iserles1991OnProducts} and has since been thoroughly studied (see~\cite[Section 5.1]{Marcellan2015OnPolynomials} for a survey) to give algebraic relations between the usual Sobolev orthogonal polynomials and the $\tilde{p}_n$'s. However, it can be directly seen that if $p_n$ is a linear combination of only $\tilde{p}'_{n+1}$ and $\tilde{p}'_{n}$ (resp. $\tilde{p}'_{n-1}$), then $q_{n+1}$ associated to~\eqref{eqn:2-meas-ip} is a linear combination of $\tilde{p}_{n+1}$ and $\tilde{p}_n$ (resp. $\tilde{p}_{n-1}$).

For~\eqref{eqn:2-meas-ip} to be inner product on $H^1(\nu)$ which induces its usual topology, it then suffices for the coherent pair $(\tilde{\nu}, \nu)$ to satisfy the Poincaré inequality~\eqref{eqn:diff-poinc-ineq} or equivalently the Muckenhoupt criterion~\cite[Chapter 4.5.1]{Bakry2014AnalysisOperators}. Remarkably, it turns out that this property is satisfied by \emph{all} pairs $(\tilde{\nu},\nu)$ which are classified to be coherent (see~\cite[Section 5.2]{Marcellan2015OnPolynomials} and~\cite{Meijer1997DeterminationPairs} for this classification).

In the context of coherent pairs, the $q_n$'s associated to~\eqref{eqn:2-meas-ip} readily appear in the literature in order to study the usual Sobolev orthogonal polynomials~\cite{Marcellan2015OnPolynomials, Meijer1993CoherentPolynomials}, but were not identified as being themselves orthogonal with respect to a Sobolev inner product.
\end{rmk}

\subsection{The differentiation operator}
\label{subsec:diff-op}

In preparation for the resolution of equations of the form~\eqref{eqn:Lu=f(u)}, we now write out explicitly the ``infinite matrix'' representation of $\partial_x$ and $\cL$ in the bases $\cP$ of $L^2(\nu)$ and $\cQ$ of $H^1(\nu)$. In particular, this section contains the proof of Proposition~\ref{prop:decompo-intro}, which summarizes Proposition~\ref{prop:changeofbasis} and Proposition~\ref{prop:L-decomp} below.


\begin{notation}
    In the rest of this paper, so that all our statements are precise, we denote the $\ell^2$- representation of an element $u$ of a weighted Sobolev space
    \begin{itemize}
    \item by $\left[u\right]_{L^2}$ if we consider its representation with respect to the orthonormal basis $\mathcal{P} = \{p_n\}_{n=0}^{\infty}$ of $L^2(\nu)$,
    \item by $\left[u\right]_{H^1}$ if we consider its representation with respect to the orthonormal basis $\mathcal{Q} = \{q_n\}_{n=0}^{\infty}$ of $H^1(\nu)$.
    \end{itemize}
    That is, $\left[u\right]_{L^2} \in \ell^2$ is the infinite vector of coefficients $\left(\langle p_n, u \rangle\right)_{n\in\N} $, and $\left[u\right]_{H^1}\in \ell^2$ is the infinite vector of coefficients $\left((q_n, u)\right)_{n\in\N}$.
    Furthermore, for the sake of clarity, we may shorten some notations by omitting the $\nu$. Henceforth, $\|\cdot\|_{H^1}$ (resp. $\|\cdot\|_{L^2}$) will always denote $\|\cdot\|_{H^1(\nu)}$ (resp. $\|\cdot\|_{L^2(\nu)}$).

    Given an operator $\mathcal{A}$ acting between two weighted Sobolev spaces, for instance $\mathcal{A}:H^1 \to L^2 $, we denote by $\left[\mathcal{A}\right]_{H^1\to L^2} : \ell^2\to \ell^2$  the operator such that $\left[\mathcal{A}u\right]_{L^2} = \left[\mathcal{A}\right]_{H^1\to L^2} \left[u\right]_{H^1}$. In other words, $\left[\mathcal{A}\right]_{H^1\to L^2}$ is the infinite matrix representing the operator $\mathcal{A}$ with basis $\mathcal{Q}$ for the domain and basis $\mathcal{P}$ for the range.  We use similar arrow notations for operator norms, i.e.,
    $$\|\mathcal{A}\|_{H^1\to L^2} = \sup_{\substack{u\in H^1 \\ u\neq 0}}\frac{\|\mathcal{A}u\|_{L^2(\nu)}}{\|u\|_{H^1(\nu)}},$$
    and note that
    $$\|\mathcal{A}\|_{H^1\to L^2}  =  \left\|\left[\mathcal{A}\right]_{H^1\to L^2}\right\|_{\ell^2\to\ell^2}.$$
    When the two spaces are the same, we sometimes write $\left[\mathcal{A}\right]_{L^2}$ instead of $\left[\mathcal{A}\right]_{L^2\to L^2}$.
\end{notation}

In order to eventually deal with $\cL$, we first consider the (unbounded) differentiation operator $\partial_x$, which we write with respect to $\mathcal{P}$, i.e.,
$$\left[\partial_x\right]_{L^2\to L^2} = D = \begin{pNiceMatrix}[margin]
  0 &\Block[borders={top, left}]{4-4}{}& & & & \\
  \vdots & & D_{>0} & & & \\
  & & & & & \\
  & & & & & \\
\end{pNiceMatrix}.$$
As a differentiation operator acting on a polynomial basis, $D$ is strictly upper triangular. The submatrix $D_{>0}$, which corresponds to $\partial_x$ from the subspace of mean zero functions of $L^2(\nu)$ to $L^2(\nu)$, will play a central role in the analysis to come. In particular, the operator $\partial_x$ (or equivalently $D$) is closely related to the change of basis matrix $P = \left[\id\right]_{L^2\to H^1}$ from $\mathcal{P}$ to $\mathcal{Q}$, i.e. the matrix such that 
$\left[u\right]_{H^1} = P \left[u\right]_{L^2} $.

\begin{prop}
\label{prop:changeofbasis}
The change of basis matrix $P$ from $\mathcal{P}$ to $\mathcal{Q}$ is given by

\begin{align}
\label{eq:defP}
    P  = \left[\mathrm{id}\right]_{L^2\to H^1}  = \begin{pNiceMatrix}[margin]
  1 & 0 & \cdots& & & \\
  0 &\Block[borders={top, left}]{4-4}{}& & & & \\
  \vdots & & D_{>0} & & & \\
  & & & & & \\
  & & & & & \\
\end{pNiceMatrix}.
\end{align}
\end{prop}

\begin{proof}
We normalised the $p_n$ and $q_m$ so that $p_0 = q_0 =1$, which explains the first column of $P$. The first row of zeros is an immediate consequence of the fact that the $p_n$ all have mean $0$ for $n\geq 1$. Finally, using that the $q_m$ also all have mean $0$ for $m\geq 1$, we get

\begin{align*}
    (q_m, p_n) &= \langle q_m', p'_n\rangle + \langle q_m, 1\rangle \langle p_n, 1\rangle\\
    &= \langle p_{m-1}, p'_n\rangle\\
    &=D_{m-1,n}. \qedhere
\end{align*}
\end{proof}


The following observations are useful for the resolution of differential equations involving the operator $\mathcal{L}$ such as~\eqref{eqn:GP-intro}.

\begin{prop}\label{prop:L-decomp}
    We have that 
    $$\left[\mathcal{L}\right]_{L^2} = D^TD \qquad \text{and}\qquad \left[\mathcal{L}\right]_{H^1_0} = D_{>0}D_{>0}^T.$$
\end{prop}
\begin{proof}
    According to~\eqref{eqn:adjoint-L}, for all $u, v\in H^1$
    \begin{equation*}
        \langle u, \mathcal{L}v\rangle = \langle u', v'\rangle = [\partial_x u]_{L^2}^T[\partial_x v]_{L^2} = (D[u]_{L^2})^T(D[v]_{L^2}) = [u]_{L^2}^TD^TD[v]_{L^2},
    \end{equation*}
    which yields the first identity. In particular, the restriction of $\cL$ to the subspace of $\nu$-mean zero of $L^2$ writes $D_{>0}^T D_{>0}$, and, using Proposition~\ref{prop:changeofbasis},
\begin{equation*}
\left[\mathcal{L}\right]_{H^1_0} = D_{>0}\left(D_{>0}^T D_{>0}\right) D_{>0}^{-1} = D_{>0}D_{>0}^T. \qedhere
\end{equation*}
\end{proof}
Since $D$ is upper triangular, $\left[\mathcal{L}\right]_{L^2} = D^TD$ gives a Cholesky factorisation (LU) of $\left[\mathcal{L}\right]_{L^2}$, whereas $\left[\mathcal{L}\right]_{H^1_0} = D_{>0}D_{>0}^T$ gives a \emph{reverse Cholesky} factorisation (UL) of $\left[\mathcal{L}\right]_{H^1_0} $. UL decompositions are rather uncommon as they are typically non-unique in infinite dimension, but it seems to be canonical in this case.

Proposition~\ref{prop:L-decomp} allows for a simple inversion of the operator $\mathcal{L}$. This is particularly the case if the potential $V$ is polynomial. Indeed, we have

\begin{lem}\label{lem:bandedness}
    Let $V(x) = x^{2k}/(2k) + o(x^{2k})$ be a polynomial of degree $2k$, and $\nu(\d x) = e^{-V}\d x/\cZ$.
    Then, for all $m\leq n-2k$
    $$D_{m,n} = \langle p_m, p'_n\rangle = 0,$$
    such that $D$ (and $P$) are upper triangular operators on $\ell^2$ of upper bandwidth $(2k-1)$.
\end{lem}

\begin{proof}
    Let $m<n$, then
    \begin{align*}
        \langle p_m, p_n' \rangle &= \frac{1}{\cZ}\int_{-\infty}^{+\infty}p_m(x)p_n'(x)e^{-V(x)}\d x\\
        &=-\frac{1}{\cZ}\int_{-\infty}^{+\infty}p_n(x)\left(p_m(x)e^{-V(x)}\right)'\d x\qquad\mbox{By parts}\\
        &=\frac{1}{\cZ}\int_{-\infty}^{+\infty} p_n(x)p_m(x)V'(x)e^{-V(x)}\d x -\frac{1}{\cZ}\int_{-\infty}^{+\infty}p_n(x)p_m'(x)e^{-V(x)}\d x\\
        &=\int_{-\infty}^{+\infty} p_n(x)p_m(x)V'(x)\nu(\d x)\qquad\mbox{By orthogonality since $\deg(p_m')= m-1<n$}.
    \end{align*}
    Since $V'p_m$ is a polynomial of degree $m+2k-1$, the above term is zero by orthogonality if $m+2k-1<n$, i.e. $m\leq n-2k$.
\end{proof}

\begin{rmk}
     In the case of a polynomial potential $V$, one can think of the basis $\{p_n\}_{n\in \N}$ as semi-classical in the sense that $[\partial_x]_{L^2}$ is banded with upper bandwidth $(2k-1)$, in contrast with the classical case of Hermite polynomial ($\nu(\d x) = e^{-x^2/2}\d x/\cZ$) where $[\partial_x]_{L^2}$ is diagonal. This, in turn, implies that $[\mathcal{L}]_{L^2}$ is banded with bandwidth $(4k-3)$, extending the result of~\cite{Mazet1997ClassificationOrthogonaux} which classified diffusion operators with a polynomial eigenbasis on $\R$.
\end{rmk}

\begin{rmk}\label{rmk:parity}
    Of course, if $V$ is an even polynomial then for all $m = n\mod 2$, $D_{m,n} = \langle p_m, p_n'\rangle = 0 $.
\end{rmk}

\begin{rmk}\label{rmk:regularity}
    Observe that
    $$\mathcal{L} = -\mathcal{J}\circ\partial_x \quad \text{where } \quad \mathcal{J} = -
    V' +\partial_x$$
    is the so-called probability flux (or current) operator~\cite{Pavliotis2014StochasticApplications}. Now, for $u,v\in H^1(\nu)$,
    \begin{align*}
        \langle u, \partial_{x}v\rangle &= \frac{1}{\cZ}\int_{\R} u (\partial_x v)e^{-V}\d x\\
        &=- \frac{1}{\cZ}\int_{\R}v\partial_x(ue^{-V})\d x\\
        &=-\frac{1}{\cZ}\int_{\R}v(-V'u +\partial_x u)e^{-V}\d x\\
        &=\langle -\mathcal{J}u, v\rangle,
    \end{align*}
    which shows that $-\mathcal{J}$ is the $L^2(\nu)$-adjoint of $\partial_x$, and provides another proof of the first identity of Proposition~\ref{prop:L-decomp}. Another consequence of this observation is that 
    \begin{equation}\label{eqn:D-V-relation}
        D^T =  \left[-\mathcal{J}\right]_{L^2} = \left[V'\cdot\right]_{L^2} - D,\qquad \text{and thus} \qquad [V'\cdot]_{L^2} = D + D^T,
    \end{equation}
    where $\left[V'\cdot\right]_{L^2}$ is the multiplication operator by $V'$.
    This means that the coefficients of $[V'\cdot]_{L^2}$ grow at the same speed as the ones of $D$, which is consistent with the expectation that multiplying by $V'$ loses as much regularity as differentiating once on a $\nu$-weighted Sobolev space. Furthermore, when $V$ is a polynomial, note that one can therefore deduce the differentiation matrix $D$ directly from the multiplication by $V'$ operator, and thus from the multiplication by $x$ (Jacobi) operator.
\end{rmk}




\section{The semiclassical case of a polynomial potential $V$}
\label{sec:Vpoly}

In this section, we now focus on the case of $V$ being an even polynomial potential of the form
\begin{equation}
\label{eq:generalV}
    V(x) = \sum_{j=0}^kc_j \frac{x^{2j}}{2j}, \qquad c_k = 1,
\end{equation}
and study the compactness of the embedding $H^1(\nu)\hookrightarrow L^2(\nu)$.

We start in Section~\ref{sec:freud-sobolev} with the quartic case $k=2$, for which $D_{>0}$ is merely bidiagonal. We first explicitly write down the operator $D$ (and $P$) in terms of coefficients $a_n$ from~\eqref{eqn:jacobi-rec-rel_intro} in Section~\ref{sec:Freuddef}. Proposition~\ref{prop:bread-bounds}, which we prove in Section~\ref{sec:bread}, provides us with quantitative and very precise estimates on the coefficients $a_n$. In Section~\ref{subsec:compact-est}, we then use these estimates to obtain a fine control on $D_{>0}^{-1}$, which allows us to quantify the compactness of the embedding $H^1(\nu)\hookrightarrow L^2(\nu)$ (i.e., to prove Theorem~\ref{thm:quant-compact}), and to obtain tight bounds on the Poincaré constant $C_P$ (i.e., to prove Theorem~\ref{thm:quant-poinc}).
Finally, we put all these estimates in slightly broader perspective in Section~\ref{sec:gen-poly-case} , which allows us to study the compactness of the embedding $H^1(\nu)\hookrightarrow L^2(\nu)$ for all $k\in \N$, and to prove Theorem~\ref{thm:compact}.

\subsection{The quartic case}\label{sec:freud-sobolev}
Let us first consider the quartic or double-well potential
\begin{equation}\label{eqn:V-form}
    V(x) = \frac{x^4}{4}-\kappa\frac{x^2}{2},
\end{equation}
so that
$$\d\nu = \frac{e^{-V(x)}}{\cZ}\d x, \qquad \text{with }\cZ = \int_{-\infty}^{+\infty}\exp\left(-\frac{x^4}{4}+\kappa\frac{x^2}{2}\right)\d x,$$
where $\cZ$ can be expressed and computed via the use of Bessel functions.

\subsubsection{Freud-type orthogonal polynomials}
\label{sec:Freuddef}


The orthonormal polynomials with respect to $\nu$ are the so-called \emph{Freud-type} orthonormal polynomials $\{p_n\}_{n\in\N}$ and constitute a family of \emph{semi-classical} orthonormal polynomials.
These $p_n$'s can be constructed recursively via a three-term recurrence formula of the form~\eqref{eqn:jacobi-rec-rel_intro}, where each $a_n>0$, and where there is no $p_n$ term in the right-hand side since $V$ is even.

Furthermore, in the case of Freud-type polynomials, one can give an explicit recurrence formula for the positive sequence $(a_n)_{n\in \N}$~\cite{Bonan1984OrthogonalI}. Indeed, the sequence $(b_n)_{n\in \N} := (a^2_n)_{n\in\N}$ satisfies the so-called discrete Painlevé~I equation~\eqref{eqn:intro-bread}
%
%
\begin{equation}\label{eqn:intro-bread-initial}
   b_0=0 \qquad \text{and} \qquad b_1 = \int_{\R}x^2\nu(\d x) = \frac{\displaystyle{\int_{-\infty}^{+\infty}x^2\exp\left(-\frac{x^4}{4}+\kappa\frac{x^2}{2}\right)\d x}}{\displaystyle{\int_{-\infty}^{+\infty}\exp\left(-\frac{x^4}{4}+\kappa\frac{x^2}{2}\right)\d x}},
\end{equation}
which can also be computed via Bessel functions (see~\eqref{eq:b1}). 

In view of Lemma~\ref{lem:bandedness} and Remark~\ref{rmk:parity}, $p_n'$ should be a linear combination of $p_{n-1}$ and $p_{n-3}$ and indeed have the following:
\begin{prop}
\label{prop:derpn}
Let $V$ be of the form~\eqref{eqn:V-form} and $\{p_n\}_{n\in\N}$ be the basis of orthonormal polynomials of $L^2(\nu) = L^2(e^{-V}/\cZ)$ generated by~\eqref{eqn:jacobi-rec-rel_intro} and~\eqref{eqn:intro-bread}, then
    \begin{equation}
    \label{eq:derpn}
        p_n' = \frac{n}{a_n}p_{n-1}+a_n a_{n-1}a_{n-2}p_{n-3}.
    \end{equation}
\end{prop}

\begin{proof}
This can be obtained rather easily from Remark~\ref{rmk:regularity}. Indeed, applying~\eqref{eqn:jacobi-rec-rel_intro} three times to express $x^3p_n$ yields
\begin{align*}
    V'(x) p_n &= (x^3 - \kappa x)p_n \\
    &= a_{n+1}a_{n+2}a_{n+3} p_{n+3} + a_{n+1}\left(a_{n+2}^2+a_{n+1}^2+a_{n}^2-\kappa\right) p_{n+1} \\
    &\quad + a_{n}a_{n-1}a_{n-2} p_{n-3} + a_{n}\left(a_{n+1}^2+a_{n}^2+a_{n-1}^2-\kappa\right) p_{n-1}.
\end{align*}
Using that $[V']_{L^2} = D + D^T$ and that $D$ is upper triangular allows us to extract the formula for $p_n'$, namely
\begin{align*}
    p_n' = a_{n}a_{n-1}a_{n-2} p_{n-3} + a_{n}\left(a_{n+1}^2+a_{n}^2+a_{n-1}^2-\kappa\right) p_{n-1},
\end{align*}
and using~\eqref{eqn:intro-bread} (with $a_n^2 = b_n$) gives~\eqref{eq:derpn}.
\end{proof}
%


Proposition~\ref{prop:derpn} means that the differential operator $D$ and change of basis operator $P$ are bidiagonal and can be written as
\begin{equation}
    \label{eq:DandP}
D = \begin{pmatrix}
     0& \alpha_{1} & 0 & \beta_{1} & &\\
    & 0&\alpha_{2} & 0 & \beta_2 & &\\
    & & \ddots &\ddots & \ddots &\ddots &\\
    & & & & &\\
    & & & & &
\end{pmatrix}\qquad\text{and}\qquad P = \begin{pmatrix}
    1 & 0 & 0& & & \\
     & \alpha_{1} & 0 & \beta_{1} & & &\\
    &   &\alpha_{2} &0 &\beta_2 &\\
    & &   &\ddots & \ddots& \ddots& \\
    & & & & & 
\end{pmatrix},
\end{equation}
where 
\begin{equation}\label{eqn:alpha-beta-def}
\alpha_n = \frac{n}{a_{n}}, \qquad \beta_n = a_{n}a_{n+1}a_{n+2}.
\end{equation}

\begin{rmk}
    Bonan and Nevai~\cite{Bonan1984OrthogonalI} actually prove that $[\partial_x]_{L^2}$ is bidiagonal if and only if $V$ is of the form~\eqref{eqn:V-form}, up to a translation and scaling (in $x$).
\end{rmk}

In order to justify the invertibility of $P$ (or equivalently, of $D_{>0}$) and to get explicit compactness estimates for its inverse, we need a precise control on the growths of the $\alpha_n$'s and the $\beta_n$'s, and thus on the growth of the $b_n$'s. Indeed, using the triangular structure of $P$ one can easily derive a formal inverse $P^{-1}$, but to prove that this yields a well-defined and bounded operator, and to then get quantitative compactness estimates, we will need an estimate of the form
\begin{align}\label{eqn:cond_theta}
    \frac{\beta_n}{\alpha_n} \leq \theta < 1 \qquad \text{for all }n\geq N_0,
\end{align}
with explicit threshold $N_0$ and constant $\theta$. This condition, which naturally appears in our proof of Theorem~\ref{thm:quant-compact} in Section~\ref{subsubsec:compactness}, was already used in~\cite{Breden2015RigorousPart} in the slightly more general context of tridiagonal operators.

Looking at~\eqref{eqn:intro-bread}, a formal calculation suggests that $b_n\sim \sqrt{n}/\sqrt{3}$ as $n\to\infty$, and this can in fact be established rigorously~\cite[§6]{Alsulami2015ASolutions}. From the formulae~\eqref{eqn:alpha-beta-def} and the fact that $a_n = \sqrt{b_n}$, we get $\beta_n/\alpha_n \to 1/3$, which shows the existence of a $\theta<1$ and of an $N_0$ large enough such that~\eqref{eqn:cond_theta} holds. However, since we require explicit $N_0$ and $\theta$, we need more quantitative estimates on the $a_n$'s (or equivalently on the $b_n$'s), like the ones stated in Proposition~\ref{prop:bread-bounds}, with explicit constants $c^-$ and $c^+$. In order to establish these bounds, we have to overcome the following difficulties. While the recurrence relation~\eqref{eqn:intro-bread} might seem in appearance very innocent, it turns out that it is extremely unstable and sensitive to initial conditions.  More is in fact true~\cite{Alsulami2015ASolutions,Bonan1984OrthogonalI}: the initial condition given by~\eqref{eqn:intro-bread-initial} is the unique one such that $b_n\geq 0$ for all $n\in \N$. In particular, this suggests that one cannot get estimates solely from~\eqref{eqn:intro-bread}, the positivity condition being crucial. This furthermore constitutes a computational challenge as even calculating a finite number of terms in double precision quickly diverges from the positive solution.


We point out the analysis of the recurrence relation~\eqref{eqn:intro-bread} has been the topic of a vast literature~\cite{Alsulami2015ASolutions, Bleher1999SemiclassicalModel, Clarkson2018PropertiesPolynomials,Fokas1991DiscreteGravity,Hajela1987OnRecurrences, Lew1983NonnegativeRecurrence, Lubinsky1986FreudsWeights,Magnus1999FreudsEquations, Magnus1995Painleve-typePolynomials, Nevai1983OrthogonalX4} including the celebrated Freud's conjecture and its proof in the case $\kappa=0$~\cite{Freud1976OnPolynomials} (see~\cite[§2]{Alsulami2015ASolutions} for a nice historical review of the problem). However, quantitative bounds like~\eqref{eqn:bn_bounds} with explicit constants do not seem to be already available, and we obtain those via a computer-assisted proof, which, although based on a fixed-point argument, has some unorthodox features.

Finally, let us emphasise that we do not have access to an explicit form of the $p_n$'s and $q_n$'s, and nor of the $a_n$'s, but we show in this work that these explicit descriptions are in fact not necessary in order to successfully use these bases in practice. This is demonstrated in an unequivocal manner via the establishment of Theorems~\ref{thm:GP1} and~\ref{thm:GP2}, which use these bases both for the obtention of the approximate solutions, and for deriving the compactness estimates required for the proof.

\subsubsection{A precise analysis of the discrete Painlevé I equation}\label{sec:bread}



\begin{proof}[Proof of Proposition~\ref{prop:bread-bounds}]

The qualitative statement, namely that for all $0<c^-<1<c^+$, there exists $N_0$ such that~\eqref{eqn:intro-bn_bounds} holds for all $n\geq N_0$, follows directly from the fact that $b_n\sim \sqrt{n}/\sqrt{3}$, which is for instance proved nicely in~\cite[§6]{Alsulami2015ASolutions}. However, such asymptotic analysis does not provide quantitative bounds on the $b_n$'s. The quantitative statement follows from a computer-assisted proof and adapts a construction first proposed in~\cite{Lew1983NonnegativeRecurrence} (which gives quantitative bounds in the case $\kappa = 0$) and greatly simplified in~\cite{Hajela1987OnRecurrences} in the case of this particular recurrence. We first describe below a procedure where, for a given $\kappa$ and fixed $c^-$ and $c^+$ close enough to $1$, we find the smallest possible $N_1$ such that, for $(b_n)_{n\in \N}$ satisfying~\eqref{eqn:intro-bread} and~\eqref{eqn:intro-bread-initial}, we get
\begin{equation}\label{eqn:bn_bounds}
    c^-\frac{\sqrt{n}}{\sqrt{3}}\leq b_n\leq c^+\frac{\sqrt{n}}{\sqrt{3}},
\end{equation}
for all $n\geq N_1$.

\step{1}[Set-up of a fixed-point problem]\normalfont Adapting the strategy in~\cite{Hajela1987OnRecurrences}, we introduce 
$$f(x) = - x +\sqrt{1+x^2} \qquad \text{and}\qquad g_n(x,y) = \frac{x+y-\kappa}{2\sqrt{n}},\ n\geq1,$$
and the map $S : \R^{\N^*} \to \R^{\N^*}$ such that, for $b\in \R^{\N^*}$,
\begin{align*}
    (Sb)_n &= \sqrt{n}f(g(b_{n-1}, b_{n+1})), \qquad \text{for }n>1,\\
    (Sb)_1 &= f(g(0, b_{2})),
\end{align*}
where $\R^{\N^*}$ is endowed with the product topology.
It is then straightforward to check that, if $Sb = b$, then, defining $b_0=0$, $b$ satisfies the recurrence relation \eqref{eqn:intro-bread}, and we thus seek a fixed point of $S$ in $\R_{+}^{\N^*}$. Note that, by the aforementioned uniqueness of a positive solution of~\eqref{eqn:intro-bread} (see~\cite{Alsulami2015ASolutions, Bonan1984OrthogonalI}), the solution to~\eqref{eqn:intro-bread}-\eqref{eqn:intro-bread-initial} is the unique fixed point of $S$ in $\R_{+}^{\N^*}$.

\step{2}[Upper and lower bounds of a fixed point]\normalfont We keep following the ideas from~\cite[Theorem 2.1.(a)]{Hajela1987OnRecurrences}, which relies on a Banach lattice theory argument, and state the following lemma. 
\begin{lem}\label{lem:bread-proof}
If there are sequences $b^-, b^+\in \mathbb{R}_+^{\N^*}$ such that for all $n\in \N^*$
\begin{equation}
    0\leq b^-_n \leq (Sb^+)_n\leq (Sb^-)_n\leq b^+_n,\label{eqn:bn_ineq}
\end{equation}
then the unique positive solution $b$ of~\eqref{eqn:intro-bread} satisfies $b^-_n \leq b_n \leq b^+_n$ for all $n \in \N^*$.
\end{lem}

\begin{proof}The inequality $(Sb^+)_n\leq (Sb^-)_n$ is implied by $b^-_n \leq b^+_n$, as from its definition $S$ is monotone decreasing with respect to the canonical partial order $\leq$ on $\R_+^{\N^*}$. Now, for the same reason, we have that for all $n \in \N^{*}$, 

$$(S(K))_n \subset [b^-_n, b^+_n],\quad\text{where}\quad K = \prod_{n = 1}^{\infty} [b^-_n, b^+_n]\subset \R^{\N^*}.$$
Thus $S(K) \subset K$, which is convex and compact in $\R^{\N^*}$ by Tychonoff's theorem. Since $S : \R^{\N^*} \to \R^{\N^*}$ is continuous and $\R^{\N^*}$ is locally convex, by the Schauder--Tychonoff theorem we have that there exists $b\in K$  such that $Sb = b$. Since $K\subset \R^{\N^*}$, this fixed point is the positive solution of~\eqref{eqn:intro-bread}, and the fact that $b$ belongs to $K$ finishes the proof.
\end{proof}
We are now going to get explicit enclosures of $b$ by constructing explicit sequences $b^-$ and $b^+$ satisfying the assumptions of Lemma~\ref{lem:bread-proof}.

\step{3}[Criterion for very large $n$]\normalfont

We henceforth assume that $\kappa>0$ (the same procedure, with slightly different conditions, can be used if $\kappa<0$ or $\kappa = 0$).
Let $c^+>1$ and $c^-<1$ such that there exists $N_2>1$ satisfying
\begin{equation}\label{eqn:bound1}
\sqrt{12+\left(c^+\sqrt{1-\frac{1}{N_2}}+c^+\sqrt{1+\frac{1}{N_2}}-\frac{\sqrt{3}\kappa}{\sqrt{N_2}}\right)^2}-2c^+ - 2c^-\geq 0,
\end{equation}
and
\begin{equation}\label{eqn:bound2}
\sqrt{12+\left(2c^-\right)^2}+\frac{\sqrt{3}\kappa}{\sqrt{N_2}}-c^-\left(\sqrt{1-\frac{1}{N_2}}+\sqrt{1+\frac{1}{N_2}}\right) - 2c^+\leq 0.
\end{equation}
Note that, if $c^-$ and $c^+$ are both close enough to $1$, such an $N_2$ always exists.
For all $n\geq N_2$, let
\begin{equation}\label{eqn:bn_def}
    b^+_n = c^+\frac{\sqrt{n}}{\sqrt{3}}, \qquad b^-_n = c^-\frac{\sqrt{n}}{\sqrt{3}}.
\end{equation}
Then, for $n>N_2$,
\begin{align*}
    (Sb^+)_n &= \sqrt{n} \left(\sqrt{\frac{\left(\frac{c^+ \sqrt{n-1}}{\sqrt{3}}+\frac{c^+ \sqrt{n+1}}{\sqrt{3}}-\kappa\right)^2}{4 n}+1}-\frac{\frac{c^+
   \sqrt{n-1}}{\sqrt{3}}+\frac{c^+ \sqrt{n+1}}{\sqrt{3}}-\kappa}{2 \sqrt{n}}\right)\\
   &=\frac{\sqrt{n}}{2\sqrt{3}}\left(\sqrt{12\left(c^+\sqrt{1-\frac{1}{n}}+c^+\sqrt{1+\frac{1}{n}}-\frac{\sqrt{3}\kappa}{\sqrt{n}}\right)^2}+\frac{\sqrt{3}\kappa}{\sqrt{n}}-c^+\left(\sqrt{1-\frac{1}{n}}+\sqrt{1+\frac{1}{n}}\right)\right)\\
   &\geq\frac{n}{2\sqrt{3}}\left(\sqrt{12+\left(c^+\sqrt{1-\frac{1}{N_2}}+c^+\sqrt{1+\frac{1}{N_2}}-\frac{\sqrt{3}\kappa}{\sqrt{N_2}}\right)^2}-2c^+\right),
\end{align*}
and similarly, we get that 
\begin{align*}
    (Sb^-)_n &=\frac{\sqrt{n}}{2\sqrt{3}}\left(\sqrt{12+\left(c^-\sqrt{1-\frac{1}{n}}+c^-\sqrt{1+\frac{1}{n}}-\frac{\sqrt{3}\kappa}{\sqrt{n}}\right)^2}+\frac{\sqrt{3}\kappa}{\sqrt{n}}-c^-\left(\sqrt{1-\frac{1}{n}}+\sqrt{1+\frac{1}{n}}\right)\right)\\
    &\leq \frac{\sqrt{n}}{2\sqrt{3}}\left(\sqrt{12+\left(2c^-\right)^2}+\frac{\sqrt{3}\kappa}{\sqrt{N_2}}-c^-\left(\sqrt{1-\frac{1}{N_2}}+\sqrt{1+\frac{1}{N_2}}\right)\right).
\end{align*}
Thanks to~\eqref{eqn:bound1}--\eqref{eqn:bound2}, we do get that
\begin{equation*}
    0\leq b^-_n \leq (Sb^+)_n\leq (Sb^-)_n\leq b^+_n,
\end{equation*}
for all $n>N_2$.

\step{4}[Construction of $K$]\normalfont

We start by defining $b^+$ and $b^-$ as in~\eqref{eqn:bn_def}, for all $n\geq 1$.
\begin{itemize}
    \item Picking $N_2$ as in Step 3, i.e. satisfying~\eqref{eqn:bound1}--\eqref{eqn:bound2}, we already have proven that \eqref{eqn:bn_ineq} holds for all $n>N_2$. 
    \item However, we do not expect the threshold given by~\eqref{eqn:bound1}--\eqref{eqn:bound2} to be sharp. We therefore look for the smallest $N_3\leq N_2$ such that \eqref{eqn:bn_ineq} holds for all $n>N_3$. In principle it could be that $N_3=N_2$, but in practice we typically get $N_3 \ll N_2$.
    \item For $n \leq N_3$, we gradually decrease $b^-_n$ and increase $b^+_n$ by $\varepsilon$-inflation~\cite{Mayer1995Epsilon-inflationAlgorithms}, until~\eqref{eqn:bn_ineq} holds for all $n \geq 1$. 
\end{itemize}
All the calculations are guaranteed via interval arithmetic. If this step succeeds, we have constructed $b^- = (b^-_n)_{n\geq 1}$ and $b^+ = (b^+_n)_{n\geq 1}$ such that \eqref{eqn:bn_ineq} holds for all $n\in \N^*$. Thus, we can apply Lemma~\ref{lem:bread-proof}, which yields that $b_n^-\leq b_n\leq b_n^+$ for all $n\in \N^*$. In particular the estimate \eqref{eqn:bn_bounds} holds for all $n> N_3$. 

\step{5}[Estimates for small $n$]\normalfont
Finally, we compute $b_1$ with extended precision using the formula

\begin{equation}
\label{eq:b1}
b_1 = \int_{\R}x^2\nu(\d x) = \frac{\kappa ^2 I_{-\frac{1}{4}}\left(\frac{\kappa ^2}{8}\right)+\kappa ^2
   \left(I_{\frac{3}{4}}\left(\frac{\kappa
   ^2}{8}\right)+I_{\frac{5}{4}}\left(\frac{\kappa ^2}{8}\right)\right)+\left(\kappa
   ^2+4\right) I_{\frac{1}{4}}\left(\frac{\kappa ^2}{8}\right)}{2 \kappa 
   \left(I_{-\frac{1}{4}}\left(\frac{\kappa
   ^2}{8}\right)+I_{\frac{1}{4}}\left(\frac{\kappa ^2}{8}\right)\right)},
\end{equation}
where $I_{\gamma}$ denotes the modified Bessel function of the first kind with parameter $\gamma$ (and can be evaluated rigorously with \texttt{Arb}~\cite{Johansson2017ArbArithmetic}). Recall that we assumed $\kappa>0$, otherwise a slightly different expression in terms of Bessel functions is obtained for $b_1$. We can then compute the $b_n$'s by recursion, to find the smallest $N_1\in \N^*$ such that the estimate \eqref{eqn:bn_bounds} also holds for all $n \in \{N_1, N_1+1, \ldots, N_3\}$.

With $\kappa = 4$, $c^-=0.987$ and $c^+=1.025$, the entire procedure is successful with $N_2 = 9,000,000$, $N_3 = 9,215$ and $N_1 = 2,186$. If we aim for a value of $N_0<N_1$ (as in Proposition~\ref{prop:bread-bounds}), $c^+$ and $c^-$ are respectively increased and decreased until~\eqref{eqn:bn_bounds} is satisfied for all $n\geq N_0$. The proof of Proposition~\ref{prop:bread-bounds} is performed at~\cite{Chu2025Huggzz/Freud} in the file \texttt{GP\_eq/Painleve\_bounds.ipynb}.
\end{proof}

\begin{figure}[h]
\begin{subfigure}{0.5\textwidth}
\includegraphics[width=0.9\linewidth]{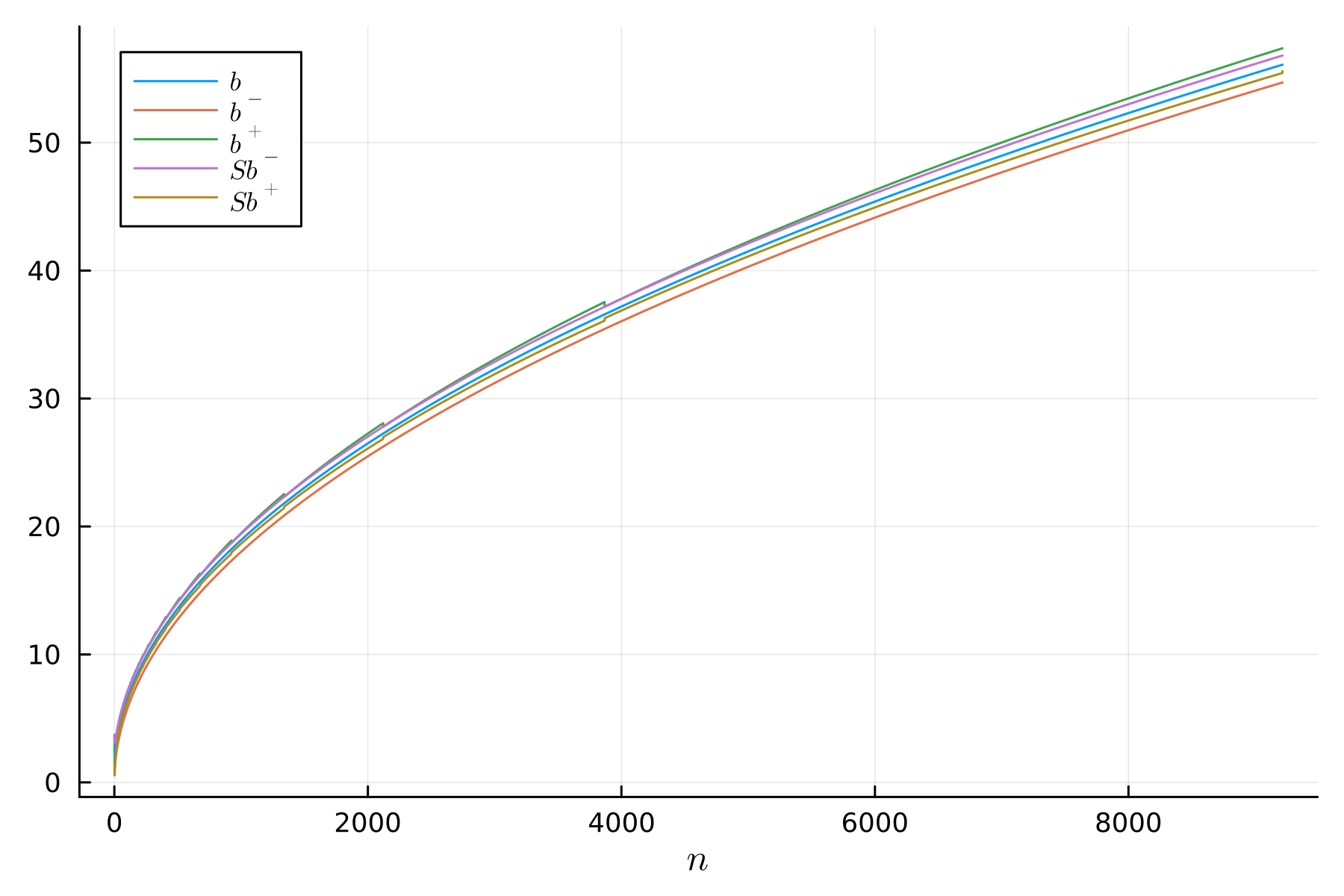} 
\caption{Construction of $(b^-_n)_{n\geq 1}$ and $(b^+_n)_{n\geq 1}$ by\\$\varepsilon$-inflation for $n\leq N_3$}
\end{subfigure}
\begin{subfigure}{0.5\textwidth}
\includegraphics[width=0.9\linewidth]{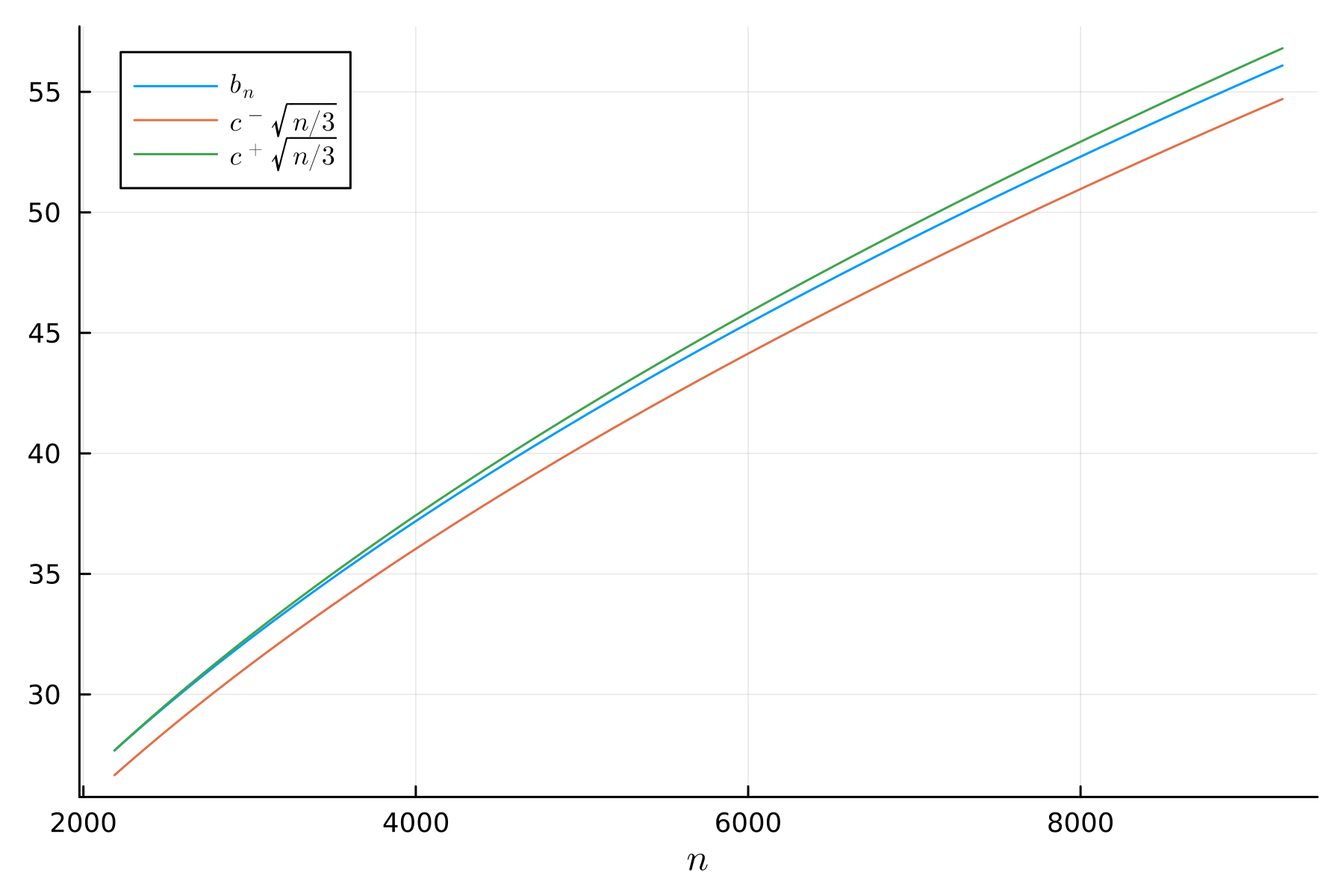} 
\caption{Verification of the bounds~\eqref{eqn:bn_bounds} for\\$n\in \{N, \ldots, N_3\}$}
\end{subfigure}

\caption{Bounds on $b_n$ for $k = 4$, $c^-=0.987$ and $c^+=1.025$.}
\end{figure}


\subsubsection{Compactness estimates}
\label{subsec:compact-est}
We now show how estimates on the growth of the positive solution of the discrete Painlevé I equation can be leveraged to quantify the compactness of the Sobolev embedding $H^1(\nu)\hookrightarrow L^2(\nu)$ and the Poincaré inequality on $H^1(\nu)$, i.e.~we give quantitative versions of Proposition~\ref{prop:fonct-poinc} and Theorem~\ref{thm:fonct-compact} via Theorems~\ref{thm:quant-poinc} and~\ref{thm:quant-compact} respectively.
First, using the upper bound on $b_n = a_n^2$ from Proposition~\ref{prop:bread-bounds}, we now obtain explicit bounds on the coefficients
$$\alpha_n = \frac{n}{a_{n}} \quad \text{and}\quad \beta_n = a_{n}a_{n+1}a_{n+2}$$
of the (unbounded) $\ell^2$-operator $P$ introduced in Section~\ref{subsec:diff-op}, which in particular allow us to satisfy~\eqref{eqn:cond_theta}.

\begin{cor}\label{cor:theta-bound}
For all $c^+>1$, there exists $N_0\in\N$ such that
\begin{equation}
\label{eq:bnc+}
 b_n\leq c^+\frac{\sqrt{n}}{\sqrt{3}},
\end{equation}
for all $n\geq N_0$.
Defining
\begin{align*}
C_{\alpha} =\frac{\sqrt[4]{3}}{\sqrt{c^+}},\qquad  \theta=\frac{(c^+)^2}{3}\left(\left(1+\frac{1}{N_0}\right)\left(1+\frac{2}{N_0}\right)\right)^{1/4}, 
\end{align*}
we get, for all $n\geq N_0$,
    \begin{equation}
        \alpha_n\geq C_{\alpha}n^{3/4} \qquad \text{and}\qquad \frac{\beta_{n}}{\alpha_n}\leq \theta. \label{eqn:alpha-bound}
    \end{equation}
In the case $\kappa = 4$, one can take $c^+ = 1.177$ and $N_0 = 50$, which yields
    $$C_{\alpha} \approx 1.213\qquad \text{and}\qquad\theta \approx 0.469.$$
\end{cor}

\begin{rmk}
\label{rmk:theta}
    Recall that $b_n\sim \sqrt{n/3}$ yields $\alpha_n \sim \sqrt[4]{3}n^{3/4}$ and $\beta_n \sim (n/3)^{3/4}$, and therefore $\beta_n/\alpha_n \sim 1/3$. In particular, up to taking $N_0$ large enough, one can always get $\theta<1$, which is needed for some subsequent estimates. If need be, one could take $N_0$ larger in Corollary~\ref{cor:theta-bound}, which would allow $c^+$ to be closer to $1$, and thus $\theta$ to be closer to $1/3$.

    Note that we only needed the upper-bounds on $b_n$ from Proposition~\ref{prop:bread-bounds} here, but the lower-bounds will also be used, in Appendix~\ref{app:annoying-bounds}.
\end{rmk}

\begin{rmk}
    Of course, these estimates become worse and harder to obtain as $\kappa$ gets larger. Furthermore, if $c^-$ and $c^+$ were chosen to be further from $1$, then $N$ would be smaller, though the quality of $C_{\alpha}$ and $\theta$ would be worse. However, the crucial feature is for $\theta$ be strictly less than $1$.
\end{rmk}


\begin{notation}
    Note that the unbounded operator $\id:H^1(\nu)\subset L^2(\nu)\to H^1(\nu)$ obviously maps even functions to even functions and odd functions to odd functions, which is reflected in the fact that the action of $P$ on even and odd modes is decoupled. In the sequel, we focus on the subspace of even functions (the case of odd functions is treated analogously), which amounts to considering
$$P = \begin{pmatrix}
    1 & 0 & & &\\
     & \alpha_{2} & \beta_{2} & &\\
    &   &\alpha_{4} &\ddots &\\
    & &   &\ddots & &\\
    & & & &
\end{pmatrix}.$$
Furthermore, for a matrix or operator $A$, we use the computational notations for selection of indices, e.g. $A_{1:N, -1}$ denotes the vector made up of the first $N$ indices (i.e. the first index is $1$ and notations are inclusive) of the last column of $A$. 
\end{notation}

\paragraph{Proof of Theorem~\ref{thm:quant-compact}.}
\label{subsubsec:compactness}

Observe that $P^{-1}$ is the $\ell^2$-representation of the embedding $\iota :H^1(\nu)\hookrightarrow L^2(\nu)$, that is

$$P^{-1} = \left[\iota\right]_{H^1\to L^2}.$$

We thus aim to obtain compactness estimates via the analysis of the $\ell^2$-operator $P^{-1}$. We first pick an even integer $N$, $N\geq N_0$, with respect to which we carve up our problem. Let us carve $P$ as follows

$$P := \begin{pNiceMatrix}[margin]
  \Block[borders={bottom, right}]{4-4}{P_{\leq N}} & & & & & & & &\\
  & & & & & & & &\\
  & & & & & & & &\\
  & & & &\beta_N  & & & &\\
  & & & & \Block[borders={top, left}]{5-5}{D_{>N}}& & & &\\
  & & & & & & & &\\
  & & & & & & & &\\
  & & & & & & & &\\
  & & & & & & & &
\end{pNiceMatrix} := \begin{pNiceMatrix}[margin]
  \Block[borders={bottom, right}]{1-1}{1} & & & & & & & &\\
  & \Block[borders={top, bottom, left, right}]{3-3}{D_{\leq M}}& & & & & & &\\
  & & & & & & & &\\
  & & & &\beta_N  & & & &\\
  & & & & \Block[borders={top, left}]{5-5}{D_{>N}}& & & &\\
  & & & & & & & &\\
  & & & & & & & &\\
  & & & & & & & &\\
  & & & & & & & &
\end{pNiceMatrix},$$
%
%
where
$$P_{\leq N} = \begin{pNiceMatrix}
    1 & 0 & \cdots &{} &{} &{}\\
    0 & \Block[borders={top, left}]{5-5}{D_{\leq N}} & & & &\\
    \vdots & & & & & \\
    & & & & & \\
    & & & & & \\
    & & & & &
\end{pNiceMatrix} = \begin{pmatrix}
    1& 0& & & & \\
    & \alpha_2 &\beta_2 & & & \\
    & & \alpha_4& \ddots& & \\
    & & & \ddots& \ddots& \\
    & & & & \ddots&\beta_{N-2} \\
    & & & & &\alpha_N
\end{pmatrix},$$
and
\begin{equation}\label{eqn:D_{>n}-def}
D_{>N} = \begin{pmatrix}
    \alpha_{N+2} & \beta_{N+2} & & &\\
     & \alpha_{N+4} & \beta_{N+4} & &\\
    &   &\ddots &\ddots &\\
    & &   & & &\\
\end{pmatrix}.
\end{equation}
Computing the inverse of $P$ seen as a two-by-two block matrix and using the Schur complement, one thus infer that
\begin{align}
\label{eq:carvedPinv}
P^{-1} = \begin{pNiceMatrix}[margin]
  \Block[borders={bottom, right}]{3-4}{P_{\leq N}^{-1}} & & & &\Block[]{3-1}{-\beta_N(P_{\leq N}^{-1})_{:,-1}(D_{>N}^{-1})_{1,:}}\\
  & & & &\\
  & & & &\\
  & & & & \Block[borders={top, left}]{4-1}{D_{>N}^{-1}}\\
  & & & &\\
  & & & &\\
  & & & &
\end{pNiceMatrix} = \begin{pNiceMatrix}[margin]
  \Block[borders={bottom, right}]{1-1}{1} & & & &\\
  & \Block[borders={top, bottom, left, right}]{2-3}{D_{\leq N}^{-1}}& & & \Block[borders={top}]{2-1}{-\beta_N(D_{\leq N}^{-1})_{:,-1}(D_{>N}^{-1})_{1,:}}\\
  & & & &\\
  & & & & \Block[borders={top, left}]{4-1}{D_{>N}^{-1}}\\
  & & & &\\
  & & & &\\
  & & & &
\end{pNiceMatrix},
\end{align}
where we will show that $P^{-1}:\ell^2(\N^*)\rightarrow\ell^2(\N^*)$ is a bounded operator. Since
\begin{align*}
    \left\Vert u \right\Vert_{L^2} = \left\Vert \iota u \right\Vert_{L^2} = \left\Vert P^{-1} [u]_{H^1} \right\Vert_{\ell^2},
\end{align*}
and
\begin{align*}
    \left\Vert u \right\Vert_{H^1} = \left\Vert [u]_{H^1} \right\Vert_{\ell^2},
\end{align*}
in order to prove Theorem~\ref{thm:quant-compact}, we need to find a constant $C$ such that, for all $u\in \overline{\mathrm{Span}\{q_j\}_{j>n}}^{H^1(\nu)}$, $n\geq N$,
\begin{align*}
    \left\Vert P^{-1} [u]_{H^1} \right\Vert_{\ell^2} \leq \frac{C}{n^{3/4}} \left\Vert [u]_{H^1} \right\Vert_{\ell^2}.
\end{align*}
In the remainder of this proof, we take $N=N_0$, but some of the calculations will be reused in the proof of Theorem~\ref{thm:quant-poinc} with a different (larger) $N$.

Denoting $n=N+2k$, $k\geq 0$, and recalling that we only work on even subspaces here, we therefore need to find $C$ such that
$$
\left\|\begin{pNiceMatrix}[margin]
  \Block[]{2-1}{-\beta_N(P_{\leq N}^{-1})_{:,-1}(D_{>N}^{-1})_{1,k+1:}}\\
  \\
  \Block[borders={top}]{3-1}{(D_{>N}^{-1})_{:,k+1:}}\\
  \\
  \,
\end{pNiceMatrix}\right\|_{\ell^2}\leq\frac{C}{(N+2k)^{3/4}}.$$
We are going to obtain two separate constants $C_{12}$ and $C_{22}$ such that
\begin{align}\label{eqn:C12-C22-def}
    \|\beta_N(P_{\leq N}^{-1})_{:,-1}(D_{>N}^{-1})_{1,k+1:}\|_{\ell^2} \leq C_{12} (N+2k)^{-3/4} \qquad\text{and}\qquad \|(D_{>N}^{-1})_{:,k+1:}\|_{\ell^2}\leq C_{22}(N+2k)^{-3/4},
\end{align}
and then take $C$ defined by
\begin{equation}
    \label{eqn:C-def}
    C := \sqrt{C_{12}^2+C_{22}^2} =\left\|\begin{pmatrix}
    C_{12}\\
    C_{22}
\end{pmatrix}\right\|_{\ell^2}.
\end{equation}

Starting with $C_{12}$, note that by backward substitution, we get from~\eqref{eqn:D_{>n}-def} that
\begin{align}
\label{eq:Dninv}
    (D_{>N}^{-1})_{ij} = \begin{cases}
    \displaystyle\frac{(-1)^{i-j}}{\alpha_{N+2j}}\prod_{l=i}^{j-1}\frac{\beta_{N+2l}}{\alpha_{N+2l}}, \qquad &\mbox{for $i\leq j$,}\\
    0 &\mbox{otherwise.}
\end{cases}
\end{align}
Using Corollary~\ref{cor:theta-bound}, assuming $N$ is taken large enough so that $\theta<1$ (see Remark~\ref{rmk:theta}), we can therefore estimate
\begin{align*}
    \|(D_{>N}^{-1})_{1,k+1:}\|_{\ell^2} &= \left(\sum_{j = k+1}^{\infty}\frac{1}{\alpha_{N+2j}^2}\prod_{l=1}^{j-1}\frac{\beta_{N+2l}^2}{\alpha_{N+2l}^2}\right)^{1/2}\\
    &\leq \frac{1}{C_{\alpha}(N+2k)^{3/4}}\left(\sum_{j = 1}^{\infty}\theta^{2(j-1)}\right)^{1/2},
\end{align*}
and thus,

\begin{equation}\label{eqn:Dn-estimate}
    \|(D_{>N}^{-1})_{1,k+1:}\|_{\ell^2}\leq\frac{1}{C_{\alpha}\sqrt{1-\theta^2}(N+2k)^{3/4}}.
\end{equation}
%
Since, by the extremal case of the Cauchy--Schwarz inequality,
$$\|-\beta_N(P_{\leq N}^{-1})_{:,-1}(D_{>N}^{-1})_{1,k+1:}\|_{\ell^2} = \beta_n\|(P_{\leq N}^{-1})_{:,-1}\|_{\ell^2}\|(D_{>N}^{-1})_{1,k+1:}\|_{\ell^2},$$
we choose
\begin{equation}\label{eqn:C12-def}
    C_{12} = \frac{\beta_N\|(P_{\leq N}^{-1})_{:,-1}\|_{\ell^2}}{C_{\alpha}\sqrt{1-\theta^2}}.
\end{equation}

Turning our attention to $C_{22}$, let us first recall that for an operator $A : \ell^2(\N^*)\rightarrow\ell^2(\N^*)$,
$$\|A\|_{\ell^2} \leq \|A\|_{\ell^1}^{1/2}\|A\|_{\ell^\infty}^{1/2},$$
where $\|\cdot\|_{\ell^1},\|\cdot\|_{\ell^\infty}$ can be evaluated easily via the formulae
$$\|A\|_{\ell^1} = \sup_{j \in \N^*}\|A_{:,j}\|_{\ell^1}\qquad \text{and}\qquad \|A\|_{\ell^\infty} = \sup_{i \in \N^*}\|A_{i,:}\|_{\ell^1}.$$
Using again~\eqref{eq:Dninv} together with~\eqref{eqn:alpha-bound}, we estimate, for all $i\geq 1$,
$$\|(D_{>N}^{-1})_{i,k+1:}\|_{\ell^1}=\sum_{j = i\vee (k+1)}^{\infty}\left|\frac{(-1)^{i-j}}{\alpha_{N+2j}}\prod_{l=i}^{j-1}\frac{\beta_{N+2l}}{\alpha_{N+2l}}\right|\leq \frac{1}{C_{\alpha}(N+2k)^{3/4}}\sum_{j = i}^{\infty}\theta^{j-i}\leq\frac{1}{C_{\alpha}(1-\theta)(N+2k)^{3/4}},$$ 
which yields 
\begin{align*}
    \|(D_{>N}^{-1})_{:,k+1:}\|_{\ell^\infty} \leq \frac{1}{C_{\alpha}(1-\theta)(N+2k)^{3/4}}.
\end{align*}
Similarly, we estimate, for all $j\geq k+1$,
$$\|(D_{>N}^{-1})_{:,j}\|_{\ell^1}=\sum_{i = 1}^{j}\left|\frac{(-1)^{i-j}}{\alpha_{N+2j}}\prod_{l=i}^{j-1}\frac{\beta_{N+2l}}{\alpha_{N+2l}}\right|\leq \frac{1}{C_{\alpha}(N+2k)^{3/4}}\sum_{i = 1}^{j}\theta^{j-i}\leq\frac{1}{C_{\alpha}(1-\theta)(N+2k)^{3/4}},$$
which yields 
\begin{align*}
    \|(D_{>N}^{-1})_{:,k+1:}\|_{\ell^1} \leq \frac{1}{C_{\alpha}(1-\theta)(N+2k)^{3/4}}.
\end{align*}

Combining the two above estimates, we get
$$\|(D_{>N}^{-1})_{:,k+1:}\|_{\ell^2} \leq \frac{1}{C_{\alpha}(1-\theta)}\frac{1}{(N+2k)^{3/4}},$$
such that we can choose
\begin{equation}\label{eqn:C22-def}
    C_{22} = \frac{1}{C_{\alpha}(1-\theta)}.
\end{equation}
When $\kappa=4$, according to Proposition~\ref{prop:bread-bounds} we can pick $c^+ = 1.177$ and $N = 50$ in Corollary~\ref{cor:theta-bound}. Repeating the same analysis on the odd subspaces and taking the maximum of the two constants yields the announced value of the constant $C$, and finishes the proof of Theorem~\ref{thm:quant-compact}.
The computational parts of the proof can be reproduced using the notebook \texttt{GP\_eq/compactness.ipynb} available at~\cite{Chu2025Huggzz/Freud}.
\hfill \qed

\paragraph{Proof of Theorem~\ref{thm:quant-poinc}}\label{sec:quant-poinc}

Recalling Remark~\ref{rmk:poinc0}, the Poincaré constant is the smallest number $C_P$ such that for all $u\in H^1_0(\nu)$,
\begin{equation}
\label{eq:defCp}
\int_{\R}u^2\d \nu \leq C_P\int_{\R}(u')^2 \d \nu,
\end{equation}
i.e. such that
$$\|u\|_{L^2}^2 \leq C_P \|u'\|_{L^2}^2 = C_P\|u\|^2_{H^1},$$
which means 
$$ 
C_P = \sup_{\substack{ u\in H^1_0 \\ u\neq 0}}\frac{\|u\|_{L^2}^2}{\|u\|_{H^1}^2}.$$
Rephrasing this in terms of our operators $P$ and $D_{>0}$, we have that the Poincaré constant is the smallest number $C_P$ such that for all $u\in H^1_0(\nu)$,
$$\|[u]_{L^2}\|_{\ell^2}^2 = \|P^{-1}[u]_{H^1}\|_{\ell^2}^2\leq C_P\|[u]_{H^1}\|_{\ell^2}^2.$$
Since $u\in H^1_0(\nu)$, i.e. $(u, q_0) = \langle u, p_0\rangle = 0$, and recalling~\eqref{eq:defP}, we get
\begin{equation}
\label{eq:CP}
C_P = \|D_{>0}^{-1}\|_{\ell^2}^2 = \sup_{\substack{ v\in \ell^2(\N^*) \\ v\neq 0}}\frac{\|D_{>0}^{-1}v\|_{\ell^2}^2}{\|v\|_{\ell^2}^2},
\end{equation}
and we therefore prove Theorem~\ref{thm:quant-poinc} by rigorously enclosing the $\ell^2$ operator norm of $D_{>0}^{-1}$.

First, since $D_{>0}^{-1}$ is upper triangular, using the carving and notations from Section~\ref{subsubsec:compactness} we have that $\|D_{\leq N}^{-1}\|_{\ell^2} \leq \|D_{>0}^{-1}\|_{\ell^2}$. A lower bound for the Poincaré constant is therefore given by
$$\bar{C}_P := \|D_{\leq N}^{-1}\|_{\ell^2}^2.$$
We emphasise that $D_{\leq N}$ is a finite matrix, and that rigorous and sharp enclosures of the $\ell^2$ operator norm of a finite matrix can easily be obtained. We give below a simple approach based on the Gershgorin circle theorem and interval arithmetic, and also refer to~\cite[Section 12.4.5]{Nakao2019NumericalEquations} for further discussion about this question.
\begin{itemize}
    \item First, compute rigorously (i.e. in interval arithmetic) $ M = D_{\leq N}^{-1}(D_{\leq N}^{-1})^T$, so that $\|D_{\leq N}^{-1}\|_{\ell^2} = \|M\|_{\ell^2}^{1/2}$. Since $M$ is symmetric, $\|M\|_{\ell^2}$ is simply the maximum of $\vert \lambda(M)\vert$ over all the eigenvalues $\lambda(M)$ of $M$.
    \item Secondly, diagonalise $M$ numerically, i.e. find $Q$ such that $Q^{-1}MQ$ is approximately diagonal.
    \item Then, rigorously compute $Q^{-1}$; this can be achieved via the function \texttt{inv} in the Julia library\\ \texttt{IntervalArithmetic.jl}~\cite{david_p_sanders_2024_10459547}.
    \item Rigorously compute $\bar{\Lambda} := Q^{-1}MQ$. Note that $\bar{\Lambda}$ and $M$ have same spectrum.
    \item Enclose the spectrum of $\bar{\Lambda}$ (and thus $M$, which is a subset of $\R_+$) via the Gershgorin circle theorem. This gives an upper bound on $\|M\|_{\ell^2}^{1/2} = \|D_{\leq N}^{-1}\|_{\ell^2}$.
    \item Finally, take $v$ the column of $Q$ corresponding to the largest numerical eigenvalue of $M$ in modulus, and rigorously compute $\|D_{\leq N}^{-1}v\|_{\ell^2}/\|v\|_{\ell^2}$, which gives a lower bound on $\|D_{\leq N}^{-1}\|_{\ell^2}$.
\end{itemize}

Using the above procedure, we obtain a rigorous enclose of $\bar{C}_p$, i.e., a lower bound for $C_P$.
It thus remains to find an upper bound for $C_P$. We pick some $N\geq N_0$, extract $D_{>0}^{-1}$ from~\eqref{eq:carvedPinv}, and proceed as in the proof of Theorem~\ref{thm:quant-compact} to get
$$C_P = \|D_{>0}^{-1}\|^2_{\ell^2} \leq \left\|\begin{pmatrix}
    \|D_{\leq N}^{-1}\|_{\ell^2} & \beta_N \|(D_{\leq N}^{-1})_{:,-1}(D_{>N}^{-1})_{1,:}\|_{\ell^2}\\
    0 & \|D_{>N}^{-1}\|_{\ell^2}
\end{pmatrix}
\right\|_{\ell^2}^2\leq \left\|\begin{pmatrix}
    \bar{C}_{P}^{1/2} & C_{12}N^{-3/4}\\
    0 & C_{22} N^{-3/4}
\end{pmatrix}
\right\|_{\ell^2}^2,$$
where $C_{12}$ and $C_{22}$ are as in~\eqref{eqn:C12-def} and~\eqref{eqn:C22-def}.
This time we pick $c^+ = 1.025$, for which assumption~\eqref{eq:bnc+} in Corollary~\ref{cor:theta-bound} holds with $N=3,500$. The corresponding values of $C_\alpha$ and $\theta$ (and thus of $C_{12}$ and $C_{22}$) yield the announced enclosure for $C_P$. The computational parts of the proof can be reproduced using the notebook \texttt{GP\_eq/Poincare.ipynb} available at~\cite{Chu2025Huggzz/Freud}. \hfill \qed

\begin{rmk}
    Note that 
    $$\bar{C}_P\leq C_P\leq\left\|\begin{pmatrix}
    \bar{C}_{P}^{1/2} & C_{12}N^{-3/4}\\
    0 & C_{22} N^{-3/4}
\end{pmatrix}
\right\|_{\ell^2}^2 = \bar{C}_{P}+\frac{C_{12}^2}{N^{3/2}}+ O\left(\frac{1}{N^3}\right),$$
which explains the sharpness of the estimate~\eqref{eqn:cp-val} in Theorem~\ref{thm:quant-poinc}.
\end{rmk}

\paragraph{Other embedding estimates}\label{sec:Sobolev-embeddings}

In the context of (nonlinear) PDEs, it is often useful to make use of Sobolev inequalities~\cite[Corollary IX.13]{Brezis2011FunctionalEquations} to deal with nonlinearities. In one dimension, this can be achieved via the derivation of $L^{\infty}$-bounds which are also useful in the proof of the positivity of solutions to elliptic equations in the fashion of~\cite[§4.3]{Breden2025ConstructiveRd}. Such bounds can be obtained via the standard inequality~\cite[Theorem VIII.7]{Brezis2011FunctionalEquations}
$$\|\varphi\|_{\infty} \leq \sqrt{2}\|\varphi\|_{L^2(\R)}^{1/2}\|\varphi'\|_{L^2(\R)}^{1/2}, \qquad \mbox{for all $\varphi\in H^1(\R)$.}$$
By setting $\varphi = e^{-V/2}u$, we then obtain
$$\|e^{-V/2}u\|_{\infty} \leq \sqrt{2}\|e^{-V/2}u\|_{L^2(\R)}^{1/2}\|(e^{-V/2}u)'\|_{L^2(\R)}^{1/2}.$$
Now, observe that
\begin{align*}
    e^{V/2}(ue^{-V/2})' = -\frac{1}{2}V'u +u' = \frac{1}{2}(-V'u+u')+\frac{1}{2}u' = \frac{1}{2}\left(\mathcal{J}u+\partial_x u\right),
\end{align*}
where $\mathcal{J} = -V'(x) + \partial_x$, such that
$$\|(e^{-V/2}u)'\|_{L^2(\R)} = \sqrt{\cZ}\|e^{V/2}(e^{-V/2}u)'\|_{L^2(\nu)}\leq \frac{\sqrt{\cZ}}{2}\left(\|\mathcal{J}u\|_{L^2}+\|\partial_x u\|_{L^2}\right).$$

Note that this calculation does not depend on the choice of $V$. Since $-\mathcal{J}$ is the $L^2(\nu)$-adjoint of $\partial_x$ (cf. proof of Proposition~\ref{prop:L-decomp}, also Remark~\ref{rmk:regularity}), it is expected that $\|\mathcal{J}\|_{H^1\to L^2}$ is finite as

$$\|\mathcal{J}\|_{H^1(\nu)\to L^2(\nu)} = \|[\mathcal{J}]_{H^1\to L^2}\|_{\ell^2}=\|[\mathcal{J}]_{L^2\to L^2}P^{-1}\|_{\ell^2} = \|-D^{T}P^{-1}\|_{\ell^2} = \|D^{T}P^{-1}\|_{\ell^2}.$$
We furthermore need to quantify this bound and have the following lemma, which can be proved using similar methods as in Section~\ref{subsubsec:compactness}.

\begin{lem}\label{lem:annoying-bounds}
    For $V = x^4/4-\kappa x^2/2$, if $\mathcal{J} = -V'(x) + \partial_x$, then there exist an explicit $C_{\mathcal{J}}>0$
    $$\|\mathcal{J}u \|_{L^2(\nu)} \leq C_{\mathcal{J}} \|u\|_{H^1(\nu)}.$$
    Furthermore, for all $u\in H^1(\nu)$,
    \begin{equation}\label{eqn:annoying-bounds1}
    \|e^{-V/2}u\|_{\infty} \leq\sqrt{\cZ(C_{\mathcal{J}}+1)}\|u\|_{L^2(\nu)}^{1/2}\|u\|_{H^1(\nu)}^{1/2}\leq\sqrt{\cZ(C_{\mathcal{J}}+1)}\max(1, C_P)^{1/4}\|u\|_{H^1(\nu)},
    \end{equation}
    and
    \begin{equation}\label{eqn:annoying-bounds2}
    \|e^{-V/2}u\|_{H^1(\R)}\leq \sqrt{\cZ}\left(\max(1, C_P)+\frac{(C_{\mathcal{J}}+1)^2}{4}\right)^{1/2}\|u\|_{H^1(\nu)}.
    \end{equation}
    Finally, if $\kappa = 4$, we can choose $C_{\mathcal{J}} = 24.6332$.
\end{lem}
\begin{proof}
    See Appendix~\ref{app:annoying-bounds}.
\end{proof}

\subsection{The general polynomial case}\label{sec:gen-poly-case}

In this section, we explain how the ideas which allowed us to derive Theorem~\ref{thm:quant-compact} in the case of $V$ being quartic can be generalised to $V$ being an arbitrary even polynomial of degree $2k$. We first recall~\cite{Clarkson2023GeneralizedWeights, Lubinsky1986FreudsWeights} that in this case, the coefficients $a_n$ of the three-term recurrence relation~\eqref{eqn:jacobi-rec-rel_intro}
can still be obtained as the positive solution of a nonlinear recurrence relation or so-called ``string equation''.

\begin{lem}\label{lem:gen-bread}
    There exists a \emph{computable} set of homogeneous polynomials $\left\{R_j : \R^{2j-1}\rightarrow\R\right\}_{j= 1}^\infty$ having positive integer coefficients and with $\deg R_j = j-1$ such that, for every $k\in\N^*$ and every even polynomial potential $V(x) = \sum_{j = 1}^k c_j\frac{x^{2j}}{2j}$, the positive sequence $(b_n)_{n\in\N} = (a^2_n)_{n\in\N}$ satisfies the recurrence relation
    \begin{equation}\label{eqn:gen-rec-rel}
        \frac{n}{b_n} = \sum_{j=1}^{k}c_jR_j(b_{n-j},\ldots, b_n,\ldots,b_{n+j}),
    \end{equation}
    where $(a_n)_{n\in\N}$ is the sequence associated to $V$ via~\eqref{eqn:jacobi-rec-rel_intro}. Furthermore,  the sum of the monomial coefficients of $R_j$ is $R_j(1,\ldots, 1) = \binom{2j-1}{j-1}$.
 \end{lem}
\begin{proof}
Using~\eqref{eqn:jacobi-rec-rel_intro} and the orthogonality of the $p_n$'s, we first observe that, for all $n\geq 1$,
\begin{align*}
    a_n\langle  p_{n-1},p_n'\rangle = \langle x p_n,p_n'\rangle = \langle p_n,x p_n'\rangle = \langle p_n,n p_n\rangle = n.
\end{align*}
    We also have from~\eqref{eqn:D-V-relation} that
    $$\langle p_{n-1}, V'p_n\rangle = \langle p_{n-1}, p_n'\rangle,$$
    and thus
    \begin{equation}
    \label{eq:V'pn}
        \langle p_{n-1}, V'p_n\rangle = \frac{n}{a_n}.
    \end{equation}
    On the other hand, by applying $2j-1$ times the recurrence formula~\eqref{eqn:jacobi-rec-rel_intro}, we can identify the polynomial $R_j$ (which does not depend on $V$) such that
    \begin{equation}
        \label{eq:xjpn}
        \langle p_{n-1}, x^{2j-1}p_n \rangle = a_n R_j(a^2_{n-j+1},\ldots, a_n^2,\ldots,a_{n+j-1}^2).
    \end{equation}
    Therefore, since $V'(x) = \sum_{j = 1}^k c_jx^{2j-1}$, combining~\eqref{eq:xjpn} and~\eqref{eq:V'pn} yields~\eqref{eqn:gen-rec-rel}. $R_j(1, \ldots, 1)$ can then be inferred by applying $2j-1$ times the recurrence formula~\eqref{eqn:jacobi-rec-rel_intro} with the sequence $(a_n)_{n\in N}$ being constant equal to $1$. This is equivalent to iterating $2j-1$ times the corresponding Jacobi operator which is now the adjacency matrix of a symmetric random walk on $\Z$. $R_j(1, \ldots, 1)=\langle p_{n-1}, x^{2j-1}p_n \rangle$ is thus the number of paths of length $2j-1$ from $n-1$ to $n$ (or equivalently from $0$ to $1$) which equals $\binom{2j-1}{j-1}$.
\end{proof}

Without loss of generality, assume that $V(x) = \sum_{j = 1}^k c_j\frac{x^{2j}}{2j}$ with $c_k = 1$. Then, analogously as in the case $k = 2$, one can once again expect that for large $n$ (if the $b_n$'s are positive), the recurrence~\eqref{eqn:gen-rec-rel} is dominated by the term of highest degree, that is
\begin{equation}
    \frac{n}{b_n} \sim R_k(b_{n-k},\ldots, b_n,\ldots,b_{n+k}),
\end{equation}
and therefore that
\begin{equation}\label{eqn:freud-b_n-growth}
    b_n\sim n^{\frac{1}{k}}\binom{2k-1}{k-1}^{-\frac{1}{k}}.
\end{equation}
For a general potential $V$ with polynomial growth, the holding of such asymptotics is known as \emph{Freud's conjecture} and has been the topic of a rich literature (see for instance~\cite{Deift1999StrongWeights, Lubinsky1986FreudsWeights} and the references therein). In the case of $V$ being polynomial, the conjecture is solved in~\cite{Deift1999StrongWeights} which however does not rely on the study of the nonlinear recurrence~\eqref{eqn:gen-rec-rel}. Designing a general validated numerical method akin to the one proposed in Section~\ref{sec:bread} to give explicit bounds on the solution of~\eqref{eqn:gen-rec-rel} would therefore be of great interest.

\begin{prop}[The polynomial case of Freud's conjecture~\cite{Deift1999StrongWeights}]\label{prop:Freud}
    Let $V(x) = \sum_{j = 1}^k c_j\frac{x^{2j}}{2j}$ with $c_k = 1$, and $(b_n)_{n\in\N} = (a^2_n)_{n\in\N}$ be the corresponding positive sequence defined by~\eqref{eqn:jacobi-rec-rel_intro}. Then \eqref{eqn:freud-b_n-growth} holds. 
\end{prop}


Proposition~\ref{prop:Freud} is a key ingredient in our proof of Theorem~\ref{thm:compact}. Before turning to the actual proof, let us give a quick heuristic explanation of how the two are related. On $H^1(\nu)$, taking a derivative is essentially equivalent in norm to multiplying by $V'$, in the sense that they induce the same loss of regularity (they both send $H^1(\nu)$ to $L^2(\nu)$). This is shown more precisely in the subsequent proof. Since $V'$ is of degree $2k-1$ and equivalent to $x^{2k-1}$ at infinity, multiplying by $V'$ essentially amounts to multiplying by $x$ $(2k-1)$ times, i.e.~applying $(2k-1)$ times the recurrence relation~\eqref{eqn:jacobi-rec-rel_intro}. Therefore, using~\eqref{eqn:freud-b_n-growth} and ignoring multiplicative constants, we expect to have
$$\|p_n\|_{H^1(\nu)}=\|p_n'\|_{L^2(\nu)}\gtrsim\| V' p_n\|_{L^2(\nu)} \gtrsim\| x^{2k-1} p_n\|_{L^2(\nu)}\gtrsim\| a_n^{2k-1} p_n\|_{L^2(\nu)}\gtrsim n^{\frac{2k-1}{2k}} \|p_n\|_{L^2(\nu)}.$$

\noindent\textit{Proof of Theorem~\ref{thm:compact}.} 
The proof follows essentially the same lines as the one of Theorem~\ref{thm:quant-compact}, in a slightly less quantitative way.
\step{1}[Carving of $D_{>0}$]\normalfont We proceed similarly as in Section~\ref{subsubsec:compactness}, except that we first factor out the expected asymptotic growth in $D_{>0} = [\partial_x]_{\cP\to\cP^*}$, and then carve it as follows
    \begin{equation*}D_{>0} = \Diag((j^{1-1/2k})_{j\in \N})M = \Diag((j^{1-1/2k})_{j\in \N}) \begin{pNiceMatrix}
        \Block[borders={bottom, right}]{1-1}{M_{11}} & M_{12}\\
        0& \Block[borders={top, left}]{1-1}{M_{22}}
    \end{pNiceMatrix},
    \end{equation*}
    where for some $N$ (to be determined later), $M_{11} = \Pi_{\leq N}M\Pi_{\leq N}$, $M_{12} = \Pi_{\leq N}M\Pi_{> N}$ and $M_{22} = \Pi_{> N}M\Pi_{> N}$. It is now sufficient to prove that $M$ has a bounded inverse $M^{-1}$: indeed, we then have 
    $$\|D^{-1}_{>0}\Pi_{>n}\|_{\ell^2\to\ell^2} = \|M^{-1}\Diag((j^{1/2k-1})_{j\in \N}) \Pi_{>n}\|_{\ell^2\to\ell^2}\leq \frac{\|M^{-1}\|_{\ell^2\to\ell^2}}{n^{1-1/2k}}. $$
    \step{2}[Boundedness of $M^{-1}$]\normalfont Again similarly as in Section~\ref{subsubsec:compactness}, we can formally invert $M$ as
    \begin{equation}M^{-1} = \begin{pNiceMatrix}
        \Block[borders={bottom, right}]{1-1}{M_{11}^{-1}} & -M_{11}^{-1}M_{12}M_{22}^{-1}\\
        0& \Block[borders={top, left}]{1-1}{M_{22}^{-1}}\end{pNiceMatrix}.\end{equation}
    First, note that $M_{11}$ is indeed invertible, as a finite-dimensional triangular matrix with non-zero coefficients on the diagonal (these are in fact $j^{1/2k}/a_j>0$, see the proof of Lemma~\ref{lem:gen-bread}). Moreover, $M_{12}$ is bounded since it has finitely many non-zero entries. The proof that $M$ has bounded inverse therefore boils down to showing that the infinite-dimensional banded operator $M_{22}$ is boundedly invertible.
    \step{3}[Boundedness of $M_{22}^{-1}$]\normalfont Now recall that $D_{>0}$ is the upper triangular part of the multiplication operator $[V'\cdot]_{L^2\to L^2}$ (Remark~\ref{rmk:regularity}).
    Then, proceeding similarly as in Lemma~\ref{lem:gen-bread} by iterating $2i-1$ times the recurrence relation~\eqref{eqn:jacobi-rec-rel_intro}, one finds that
    \begin{equation*}\langle p_{m-2j-1}, x^{2i-1}p_m\rangle = a_m a_{m-1},\ldots a_{m-2j} R_{i,j}(a_{m-i-j+1}^2, \ldots, a_{m+i-j-1}^2),\end{equation*}
    where $R_{i,j}$ is a homogeneous polynomial of degree $i-j-1$ and $R_{i,j}(1, \ldots, 1) =  \binom{2i-1}{i-j-1}$ (by counting the number of paths of length $2i-1$ between $m$ and $m-2j-1$). Using~\eqref{eqn:freud-b_n-growth}, we have that  $a_{m+l}\sim a_m$ for each fixed $l$,
    and therefore
    \begin{equation*}\langle p_{m-2j-1}, x^{2i-1}p_m\rangle\sim \binom{2i-1}{i-1-j}\left(m^{\frac{1}{2k}}\binom{2k-1}{k-1}^{-\frac{1}{k}}\right)^{2i-1} = \binom{2i-1}{i-1-j}m^{\frac{2i-1}{2k}}\binom{2k-1}{k-1}^{-\frac{2i-1}{k}}.\end{equation*}
    Therefore, 
    $$D_{2m-2j-1, m} = \langle p_{m-2j-1}, V'(x)p_m\rangle= \sum_{i=1}^kc_{i}\langle p_{m-2j-1}, x^{2i-1}p_m\rangle\sim \binom{2k-1}{k-1-j}m^{\frac{2k-1}{2k}}\binom{2k-1}{k-1}^{-\frac{2k-1}{k}},$$
    since $c_k = 1$.
    This show that the matrix $M$ introduced in Step 1 is asymptotically Toeplitz. More precisely, letting 
    $$\quad m_{2j} = \left\{\begin{tabular}{c c}$\binom{2k-1}{k-j-1},$ &for $0\leq j\leq k-1,$\\$0,$ &otherwise\end{tabular}\right. \qquad \text{and}\qquad m_{2j+1}=0,$$
    and $\bar{M}_{22}$ the upper triangular Toeplitz matrix given by
    $$\bar{M}_{22} = \binom{2k-1}{k-1}^{-\frac{2k-1}{k}}
    \begin{pmatrix}
        m_0 & m_1 & m_2 & \ldots & \\
        0 & m_0 & m_1 & m_2 & \\
         & \ddots & \ddots & \ddots & \ddots
    \end{pmatrix},$$
    we have that for all $\varepsilon>0$, if $N$ is large enough, $\|M_{22} - \bar{M}_{22}\|_{\ell^2\to\ell^2}<\varepsilon.$
    Thus, provided $\bar{M}_{22}$ is boundedly invertible, $M_{22}$ will also be boundedly invertible for all $N$ large enough.

    \step{4}[Boundedness of $\bar{M}_{22}^{-1}$]\normalfont
    By Wiener's $1/f$ Theorem, $\bar{M}_{22}$ is boundedly invertible on $\ell^1(\N)$ if and only if the function 
    \begin{align*}
        P:z\mapsto \sum_{n\geq 0} m_n z^n = \sum_{n\geq 0} m_{2n} (z^2)^n
    \end{align*}
    does not vanish on the unit disk $\{|z|\leq 1\}$. However, by the Enestr\"{o}m--Kakeya theorem~\cite{Kakeya1912OnCoefficients}, if $P(z) = 0$, then we must have
        $$|z^2|\geq \min_{j = 0, \ldots, k-2} \frac{m_{2j}}{m_{2(j+1)}} = \min_{j = 0, \ldots, k-2}\frac{\binom{2k-1}{k-1-j}}{\binom{2k-1}{k-2-j}}=\min_{j = 0, \ldots, k-2}\frac{k+j+1}{k-1-j}=\frac{k+1}{k-1}>1.$$
    Therefore, $\bar{M}_{22}^{-1}$ is indeed bounded on $\ell^1$, and because $\bar{M}_{22}^{-1}$ is Toeplitz  $\|\bar M^{-1}_{22}\|_{\ell^2\to \ell^2} \leq \|\bar M^{-1}_{22}\|_{\ell^1\to \ell^1} <\infty$. \qed

\begin{rmk}
    Note that as $k\to\infty$ then $\exp(-x^{2k}/(2k)) \to \Ind{[-1,1]}$, such that the asymptotic of~\eqref{eqn:conj-compact} as $k\to\infty$ agrees with the compactness on the embedding $H^1([-1,1])\hookrightarrow L^2([-1,1])$ which decays like $1/n$. Accordingly, $\binom{2k-1}{k-1}^{-1}P_k$ converges pointwise to $\frac{1}{1-z}$ which is in agreement with the fact that, on the interval $[-1,1]$ with appropriate boundary conditions, the integration operator is bidiagonal.
\end{rmk}


\section{The Gross--Pitaevskii equation with sextic potential}\label{sec:GP_eq} 

In this section, we prove Theorem~\ref{thm:GP1} and Theorem~\ref{thm:GP2}. As is usual with such computer-assisted proofs~\cite{vandenBerg2015RigorousDynamics,Nakao2019NumericalEquations}, the main idea is to study a well-chosen Newton-like operator and prove that it is contracting in a small neighbourhood of the approximate solution. The quantitative compactness estimate of Theorem~\ref{thm:quant-compact} plays a crucial role in this procedure, and the computable upper-bound of the Poincaré constant $C_P$ provided by Theorem~\ref{thm:quant-poinc} will also be used. 

\begin{notation}
    For the rest of this section, we fix some $n\in \mathbb{N}$ greater than $N$ in Theorem~\ref{thm:quant-compact} and 
    we denote by $\fPL$ (resp. $\fPH$) the projection onto $\mathrm{Span}\{p_m\}_{m=0}^n$ (resp. $\mathrm{Span}\{q_m\}_{m=0}^n$) and by $\iPL$ (resp. $\iPH$) the projection onto $\overline{\mathrm{Span}\{p_m\}_{m>n}}^{L^2}$ (resp. $\overline{\mathrm{Span}\{q_m\}_{m>n}}^{H^1}$). Furthermore, we treat the problem of this section only with respect to even subspaces and polynomials of $L^2(\nu)$ and $H^1(\nu)$, and all operators are restricted to these subspaces.
\end{notation}

We consider the Gross--Pitaevskii equation

$$-\partial_{xx}\varphi + W(x)\varphi +\varphi^3 = \omega \varphi,$$
where $W(x) = x^6/4 - \kappa x^4/2 + cx^2 +d$ is a sextic polynomial potential. We look for real even solutions $\varphi$ of the above equation such that
\begin{equation}\label{eqn:GP-cond}
    \int_{\mathbb{R}}\varphi^2 W\d x <\infty.
\end{equation}
Then, still with $V(x) = x^4/4 - \kappa x^2/2$, writing $\varphi = e^{-V/2}u$ gives
$$\mathcal{L} u + e^{-V}u^3- \omega u +r(x) u = 0, $$
where $r(x) = (6 + 4 c - \kappa^2)x^2/4 + (2 d - \kappa )/2$. For simplicity, we choose parameters $c$ and $d$ such that $r(x) = 0$, though handling a non-zero $r$ would not  present any additional difficulty. We are thus left with the equation
\begin{equation}\label{eqn:GP-reformulation}
    \mathcal{L} u + e^{-V}u^3- \omega u = 0,
\end{equation}
which is equivalent to the zero-finding problem
\begin{equation}\label{eqn:GP-F}
    F(u) := u - \tL^{-1}f(u)= u - \tL^{-1}(\omega u + u_0-e^{-V}u^3) = 0, \qquad u_0 := \int_{\mathbb{R}}u \d \nu,
\end{equation}
where we introduced $\tL := \mathcal{L} +\Pi_0^{L^2} = \mathcal{L} +\Pi_0^{H^1}$, where $\Pi_{>0}$ is the projection on the constants (or equivalently returns the mean),  in order to get rid of the one-dimensional kernel of $\cL$ and work with an invertible operator.

\begin{cor}\label{cor:tL}
    We have the following representations of $\tL$ on $L^2(\nu)$ and $H^1(\nu)$:
    \begin{equation}\label{eqn:tL-decomp}
        [\tL]_{L^2} = P^TP, \qquad [\tL]_{H^1} = PP^T \qquad \text{and}\qquad     [\tL]_{H^1\to L^2} = P^T.
    \end{equation}
    Furthermore,
    \begin{equation}\label{eqn:tL-estimates}
        \|\iPH\tL^{-1}\iPL\|_{L^2\to H^1}\leq \frac{C_{22}}{n^{3/4}},\qquad \|\iPH\tL^{-1}\|_{L^2\to H^1}\leq \frac{C}{n^{3/4}},
    \end{equation}
    where $C_{22}$ and $C$ are as in Section~\ref{subsubsec:compactness} and
    \begin{equation}\label{eqn:tL-zero}
        \fPH\tL^{-1}\iPL = 0\qquad \text{such that}\qquad \fPH\tL^{-1} = \fPH\tL^{-1}\fPL.
    \end{equation}
 \end{cor}
\begin{proof}
    The identities in~\eqref{eqn:tL-decomp} follow from Proposition~\ref{prop:L-decomp} by adding the action of $\tL$ on constants. Combining~\eqref{eqn:tL-decomp} with the estimates of Section~\ref{subsubsec:compactness}, gives~\eqref{eqn:tL-estimates}. The identities in~\eqref{eqn:tL-zero} follow from the fact that $P^T$ and thus $(P^{-1})^T$ is lower triangular. Note that we also have that for $u,v\in H^1(\nu)$,
    \begin{equation*}
    \langle u, \tL v\rangle = \int_{\R}u'v'\d\nu +\int_{\R}u\d\nu\int_{\R}v\d\nu = (u,v).\qedhere
    \end{equation*}
\end{proof}
Since $[\tL^{-1}]_{L^2\to H^1} = (P^{-1})^T$ which is a bounded (actually compact from Theorem~\ref{thm:quant-compact}) $\ell^2$-operator, $\tL^{-1}$ is bounded and compact as an operator from $L^2(\nu)$ to $H^1(\nu)$. Moreover, by the Sobolev embeddings recalled in Section~\ref{sec:Sobolev-embeddings}, $u\in H^1(\nu)$ implies $e^{-V/2}u\in L^{\infty}(\R)$ and thus $e^{-V}u^3\in L^2(\nu)$. Therefore, $F:H^1(\nu)\rightarrow H^1(\nu)$ is bounded and a compact perturbation of the identity. Given an approximate solution $\bar{u}$ to Eq.~\eqref{eqn:GP-F}, this motivates the introduction of an operator $A:H^1(\nu) \rightarrow H^1(\nu)$ defined as
\begin{equation}\label{eqn:A-form}A = A_n +\iPH =     \begin{pmatrix}
       A_n & 0 \\
       0 & \id
    \end{pmatrix},\end{equation}
where $A_n \approx (\fPH DF(\bar{u})\fPH)^{-1}$ (note that $\fPH DF(\bar{u})\fPH$ is finite-dimensional, hence this approximate inverse can be computed numerically), so that $A\approx DF(\bar{u})^{-1}$. This approximate inverse $A$ can then be used to build a quasi-Newton (fixed-point) operator
\begin{equation}\label{eqn:T-def}
\begin{tabular}{c c c c}
    $T : $&$\mathcal{H}$& $\longrightarrow$ & $\mathcal{H}\qquad\qquad$\\
    &$\;u$ &$\longmapsto$ &$u- AF(u)$,
\end{tabular}
\end{equation}
where
$$\mathcal{H} = \left\{u\in H^1(\nu)\mid u\text{ is even}\right\}.$$

This is a common reformulation for computer-assisted proofs, and we will show that $T:\mathcal{H}\rightarrow\mathcal{H}$ is a contraction around $\bar{u}$, such that there exists a true solution $u$ to Eq.~\eqref{eqn:GP-reformulation} in a tight and explicit neighbourhood around $\bar{u}$. Note that by Lemma~\ref{lem:annoying-bounds}, if we find a solution $u\in \mathcal{H}\subset H^1(\nu)$ to Eq.~\eqref{eqn:GP-reformulation}, then $\varphi = e^{-V/2}u$ satisfies the integrability condition~\eqref{eqn:GP-cond}. Finally, the positivity of a solution $u$ may be checked using the criterion of Appendix~\ref{app:positivity}.




\begin{rmk}
    Any Schr\"odinger operator with sextic potential
    $$-\Delta + \frac{1}{4}x^6 -\frac{\kappa}{2}x^4 +ax^2 +b, $$
    is a tridiagonal symmetric (i.e.~Jacobi) operator with respect to the weighted basis $\{e^{-V/2}p_n\}_{n = 0}^{\infty}$, where $V(x) = x^4/4-\kappa x^2/2$.
\end{rmk}


\subsection{The Newton--Kantorovich theorem}

In order to show that $T$ is actually a contraction on a neighbourhood of $\bar{u}$, we make use of the following statement, which is used in many computer-assisted proofs, and can be interpreted as a kind of Newton--Kantorovich theorem when its assumptions are expressed in terms of $F$ rather than $T$.

\begin{thm}\label{thm:N-K}
    Let $T$ be as in~\eqref{eqn:T-def} and suppose that for some $Y, Z_1, Z_2, Z_3$
    \begin{align}
    \|T(\bar{u})-\bar{u}\|_{H^1} &\leq Y,\label{eqn:Ycond}\\
    \|D^kT(\bar{u})h^k\|_{H^1} &\leq Z_k\|h\|^k, \qquad \text{for all $h\in H^1(\nu)$, with  $k = 1,2, 3$.}\label{eqn:Zcond}
    \end{align}
    If there exists $\delta>0$ such that
    \begin{align}
        Y +Z_1 \delta + \frac{1}{2}Z_2\delta^2 + \frac{1}{6}Z_3\delta^3 \leq \delta, \label{eq:condNK1}\\
        Z_1 + Z_2\delta + \frac{1}{2}Z_3 \delta^2  < 1,\label{eq:condNK2}
    \end{align}
       then $T$ has a unique fixed point $u^{\star} \in \bar{B}(\bar{u},\delta) \subset \mathcal{H}$.
\end{thm}
\begin{proof}
    See for instance~\cite[Theorem 2.1 and Corollary 4.4]{Breden2019RigorousPaths}, where this version of the theorem is proved in a more general form, and where further historical references are given.
\end{proof}

\subsection{The bounds}\label{subsec:bounds}
We first compute an approximate solution $\bar{u} \in\fPH(\mathcal{H})\subset \mathcal{C}^{\infty}(\R)$, for $n=2500$ (such that $n>N = 2,187$ in the case $\kappa = 4$ of Theorem~\ref{thm:quant-compact}).
Note that this implies that $\fPL\bar{u} = \fPH \bar{u} = \bar{u}$ and $\iPL\bar{u} = \iPH \bar{u} = 0$. We then derive below appropriate bounds $Y, Z_1, Z_2, Z_3$ to prove the existence of an actual solution $u$ to Eq.~\eqref{eqn:GP-F} (and thus to Eq.~\eqref{eqn:GP-reformulation}) close to $\bar{u}$ via Theorem~\ref{thm:N-K}. Our estimates rely on a carving of the problem in the fashion of Section~\ref{subsubsec:compactness}. Indeed, we choose $Z_1$ by treating finite and infinite dimensional parts separately, and therefore take
    \begin{equation}\label{eqn:Z1-def}Z_1 := \left\|\left(\begin{matrix}
        Z^{11} & Z^{12}\\
        Z^{21} & Z^{22}
    \end{matrix}\right)\right\|_{\ell^2}\end{equation}
    where the $Z^{ij}$ are real numbers satisfying
    \begin{align}
    \label{eq:Z1ij}
        \begin{cases}
        \|\fPH DT(\bar{u})\fPH\|_{H^1} &\leq Z^{11},\\
        \|\iPH DT(\bar{u})\fPH\|_{H^1} &\leq Z^{21},\\
        \|\fPH DT(\bar{u})\iPH\|_{H^1} &\leq Z^{12},\\
        \|\iPH DT(\bar{u})\iPH\|_{H^1} &\leq Z^{22}.\\
        \end{cases}
    \end{align}
    By Lemma~\ref{lem:seq-norm}, that implies $\|DT(\bar{u})h\|_{H^1}\leq Z_1$. To apply Theorem~\ref{thm:N-K}, we can use the following bounds which rely on the crucial estimates derived in Section~\ref{subsec:compact-est}.
\begin{prop}\label{prop:GP-bounds}
    Consider $T = \id - AF$, with $F$ as in~\eqref{eqn:GP-F} and $A$ as in~\eqref{eqn:A-form}. Let
    \begin{align*}
        Y&:=\Bigg\{\| A_n(\bar{u} -\fPH\tL^{-1}\fPL f(\bar{u}))\|_{H^1}^2\\
        &\qquad+\frac{1}{n^{3/2}}\left\{\frac{\beta_n}{C_{\alpha}\sqrt{1-\theta^2}}\left|(P_{\leq n}^{-1})_{:,-1}^{T}\left[\fPL f(\bar{u})\right]_{L^2}\right|+C_{22}\left(\| e^{-V}\bar{u}^3\|_{L^2}^2-\|\fPL (e^{-V}\bar{u}^3)\|_{L^2}^2\right)^{1/2}\right\}^2\Bigg\}^{1/2},\\
        Z^{11} &:=\|[\fPH]_{H^1}- [A_n]_{H^1} [\fPH DF(\bar{u})\fPH]_{H^1}\|_{\ell^2},\\
        Z^{21} &:=\frac{1}{n^{3/4}}\Bigg\{\frac{\beta_n}{C_{\alpha}\sqrt{1-\theta^2}} \left\|((P_{\leq n}^{-1})_{:,n})^T\left[\fPL Df(\bar{u})\fPH\right]_{H^1\to L^2}\right\|_{\ell^2}\\
        &\qquad+3C_{22}\left(\sum_{m=0}^n\|e^{-V}\bar{u}^2 q_m\|_{L^2}^2-\|\fPL(e^{-V}\bar{u}^2 q_m)\|_{L^2}^2\right)^{1/2}\Bigg\},\\
        Z^{12} &:=\frac{1}{n^{3/4}}\Bigg\{\frac{\beta_n}{C_{\alpha}\sqrt{1-\theta^2}}\left\|\left[A_n\right]_{H^1}(P_{\leq n}^{-1})^{T}\left[\fPL Df(\bar{u})\fPL\right]_{L^2}(P_{\leq n}^{-1})_{:,-1}\right\|_{\ell^2}+3C_{22}\left\|\left[A_n\right]_{H^1}(P_{\leq n}^{-1})^T w\right\|_{\ell^2}\Bigg\},\\
        Z^{22} &:= \frac{C^2}{n^{3/2}}\left\{3\|e^{-V/2}\bar{u}\|_{\infty}^2 +|\omega|\right\},\\
        Z_2 &:=6Z(C_{\cJ}+1)\max(1,C_P)^{3/2}\|A\|_{H^1}\|\bar{u}\|_{L^2},\\
        Z_3 &:=6Z(C_{\cJ}+1)\max(1,C_P)^{5/2}\|A\|_{H^1},
    \end{align*}
    where $\beta_n$ is given in~\eqref{eqn:alpha-beta-def}, $\theta$ and $C_\alpha$ in Corollary~\ref{cor:theta-bound}, $C_{22}$ in~\eqref{eqn:C22-def}, $C$ in~\eqref{eqn:C-def}, $C_{\cJ}$ in Lemma~\ref{lem:annoying-bounds}, and $C_P$ is the Poincaré constant (see Section~\ref{sec:quant-poinc}).
    If $Z_1$ be as in~\eqref{eqn:Z1-def}, then the bounds $Y, Z_1, Z_2, Z_3$ satisfy the conditions~\eqref{eqn:Ycond} and~\eqref{eqn:Zcond}.
\end{prop}
\begin{proof}
    See Appendix~\ref{app:GP-bounds}.
\end{proof}

Note that the actual calculation of these bounds relies on rigorously computing integrals of polynomials against Freud-weights. This is achieved using a quadrature rule given in Appendix~\ref{app:quad} and implemented in the file \texttt{GP\_eq/quadrature.jl} at~\cite{Chu2025Huggzz/Freud}.
We now have all the required ingredients for proving Theorem~\ref{thm:GP1} and Theorem~\ref{thm:GP2}.

\medskip

\noindent\textit{Proof of Theorem~\ref{thm:GP1}.} Using the file \texttt{GP\_eq/proof1.ipynb}  and the approximate solution $\bu_1$ stored in \texttt{GP\_eq/ubar1}, both available at~\cite{Chu2025Huggzz/Freud}, we compute the bounds $Y, Z_1, Z_2, Z_3$ given in Proposition~\ref{prop:GP-bounds}. Still using \texttt{GP\_eq/proof1.ipynb}, we then check that~\eqref{eq:condNK1} and~\eqref{eq:condNK2} hold, with $\delta = 2\times 10^{-102}$. Theorem~\ref{thm:N-K} then yields the existence of a unique zero $u_1^\star$ of $F$ (defined in~\eqref{eqn:GP-F}) such that $\Vert u_1^\star - \bu_1\Vert_{H^1(\nu)}\leq \delta$. Defining $\varphi_1^\star = e^{-V/2}u^\star_1$ and $\bar{\varphi}_1 = e^{-V/2}\bu_1$, we get that $\varphi^\star_1$ solves~\eqref{eqn:GP-intro}, and, according to Lemma~\ref{lem:annoying-bounds},
\begin{equation*}
    \Vert \varphi_1^\star - \bar{\varphi}_1\Vert_{H^1(\R)} \leq \sqrt{\cZ}\left(\max(1, C_P)+\frac{(C_{\cJ}+1)^2}{4}\right)^{1/2} \delta \leq 3.89\times 10^{-101}.
\end{equation*}
Finally, the strict positivity of $\varphi_1^\star$ is obtained using Lemma~\ref{lem:pos-GP} \hfill\qed

\medskip

\noindent\textit{Proof of Theorem~\ref{thm:GP2}.} The proof is the same as the one of Theorem~\ref{thm:GP1}, except we use the approximate solution $\bu_2$ stored in \texttt{GP\_eq/ubar2}, and then the file \texttt{GP\_eq/proof2.ipynb} to compute the corresponding bounds. \hfill\qed





%

\section{A computer-assisted proof of stochastic resonance}
\label{sec:stoc-res}
\subsection{Description of the problem}

In their early 1980s seminal papers~\cite{Benzi1981TheResonance, Benzi1983AChange}, Benzi, Parisi, Sutera and Vulpiani introduced the notion of \emph{stochastic resonance} and gave heuristics for its occurrence in the noisy Duffing oscillator
\begin{equation}\label{eqn:resonance}\tag{SR} \d x_t =(-x^3_t +\kappa x_t +\eta\cos(\omega t))\d t + \sigma \d B_t,\end{equation}
i.e., a Brownian particle evolving in a double-well potential tilted by a periodic forcing. They consider the situation in which this periodic forcing is weak, meaning that, in the absence of noise, $\eta$ is too small for the solution to switch periodically between the two wells (that is, $\eta < 2\kappa^{3/2}/3\sqrt{3})$. However, the addition of noise can counterintuitively enable the appearance of periodic-like behaviour with large amplitude. Indeed,
\begin{itemize}
    \item if $\sigma$ is large enough, the dynamics are mostly Brownian and the particle explores both wells at all times;
    \item if $\sigma$ is small enough, the particle remains in the same well for long time intervals, a behaviour close to the case $\sigma = 0$, and switches well irregularly; 
    \item there exists a ``sweet spot'' for $\sigma$, as illustrated in Figure~\ref{fig:resonance-paths} and Figure~\ref{fig:resonance-density}, at which the noise enhances the periodic forcing, making the particle switch well at very regular time intervals, with a ``period'' corresponding to that of the weak forcing.
\end{itemize}
\begin{figure}[h]
    \centering
    \includegraphics[width=0.6\textwidth]{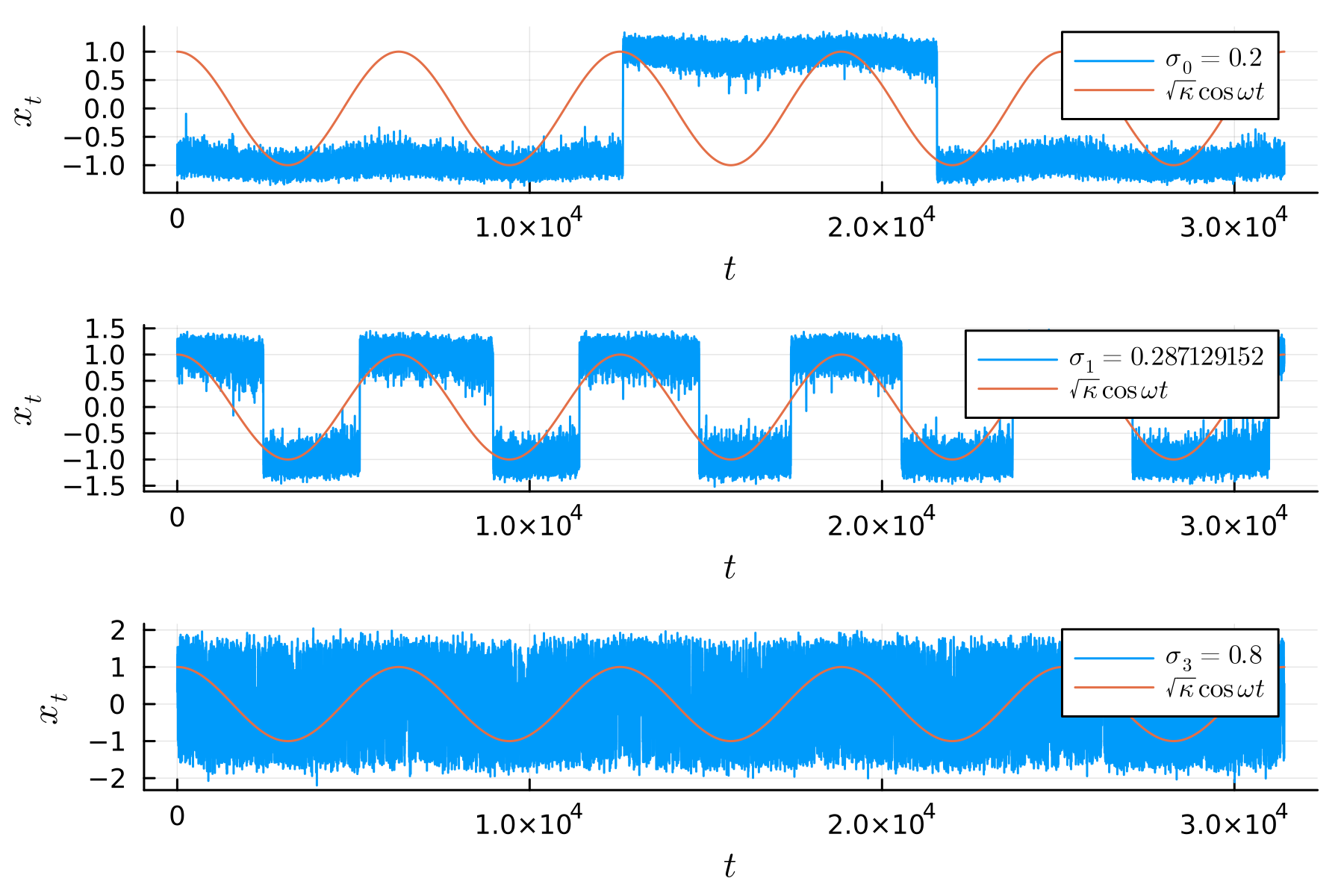}
    \caption{\centering Sample paths of $(x_t)_{t\geq 0}$ for $\eta = 0.12$, $\kappa = 1$ and $\omega = 0.001$ and various noise levels $\sigma$.\label{fig:resonance-paths}}
\end{figure}
Stochastic resonance was initially introduced in climate modelling; it has since been relevant in many fields of application such as ecology, engineering, finance, neuroscience, physics etc. (see for instance~\cite{McNamara1989TheoryResonance, Wiesenfeld1995StochasticSQUIDs} for reviews of applications). Although stochastic resonance is well understood heuristically (e.g. see~\cite[§7.4]{Pavliotis2014StochasticApplications} or~\cite[§1]{Herrmann2013StochasticResonance}) and has been extensively studied in numerical simulations~\cite{Benzi1983AChange, Cherubini2017AResonance, Gammaitoni1998StochasticResonance}, there has been no rigorous validation of this phenomenon for the SDE~\eqref{eqn:resonance}. To the best of our knowledge, the only rigorous mathematical analysis of stochastic resonance for~\eqref{eqn:resonance} has been conducted by Berglund and Gentz~\cite{Berglund2002APotential} when the system is close to forming a single-well potential at times in $\pi\N$ (that is $\eta\approx 2\kappa^{3/2}/3\sqrt{3}$) and $\sigma$ is small so that trajectories of the random system can be compared with those of the deterministic one. Let us also mention the work~\cite{Herrmann2013StochasticResonance} and references therein which show how systems of the type~\eqref{eqn:resonance} can be approximated by a two-state continuous-time Markov chain in the small noise limit $\sigma \to 0$.


Since stochastic resonance is characterised by the probability that a particle is located in a particular well at a given time, it can be captured by the probability density function $p(t, x)$ of $x_t$, which solves the Fokker--Planck equation
\begin{equation}
\label{eqn:FP-p}\pdiff{p}{t} = -\pdiff{}{x}\Big((-x^3+\kappa x+\eta\cos\omega t)p\Big)+\frac{\sigma^2}{2}\pdiff{^2p}{x^2}.
\end{equation}
Asymptotically, the statistics of $(x_t)_{t\geq 0}$ follow a so-called \emph{stationary periodic density} function $\varrho_{\sigma}$ solving Eq.~\eqref{eqn:FP-p} on the time-periodic domain $(\T/\omega)\times \R$. Indeed~\cite{Feng2023ExistenceEquations}, there exists $\lambda>0$ such that
$$\lim_{t\to\infty}\|p(t,\cdot)-\varrho_{\sigma}(t, \cdot)\|_{L^1(\R)}\lesssim e^{-\lambda t}.$$
Figure~\ref{fig:resonance-density} illustrates that some kind of local optimum with respect to the noise intensity $\sigma$ can also be observed at the level of the stationary periodic densities.


In the literature, several quantities have been introduced with the aim of characterising stochastic resonance~\cite{Benzi1983AChange,Gammaitoni1998StochasticResonance}, sometimes in a much broader context than the one of~\eqref{eqn:resonance}. In this work, we study the indicator proposed in~\cite{Gammaitoni1998StochasticResonance}, 
\begin{equation}
\label{eq:defR}
    \mathcal{R}(\varrho_{\sigma}) = \max_{t\in \T/\omega}\left|\int_{\R}x\varrho_{\sigma}(t, x)\d x\right|.
\end{equation}
If stochastic resonance occurs, this indicator should be maximised when the periodicity of the dynamics of $(x_t)_{t\geq 0}$ is the most enhanced, and closer to zero when $\sigma$ is too large or too small. Using some of the tools introduced in Section~\ref{sec:weighted-Sobolev-spaces} and~\ref{sec:Vpoly}, we prove the existence of $\sigma_{SR}>0$ such that $\sigma\mapsto \cR(\varrho_\sigma)$ has a local maximum at $\sigma=\sigma_{SR}$, thereby proving that~\eqref{eqn:resonance} exhibits stochastic resonance. We also give a rigorous enclosure for $\sigma_{SR}$, as well as a rigorous lower bound for $\cR(\varrho_{\sigma_{SR}})$. 
\begin{figure}[h]
    \centering
    \includegraphics[width=0.6\textwidth]{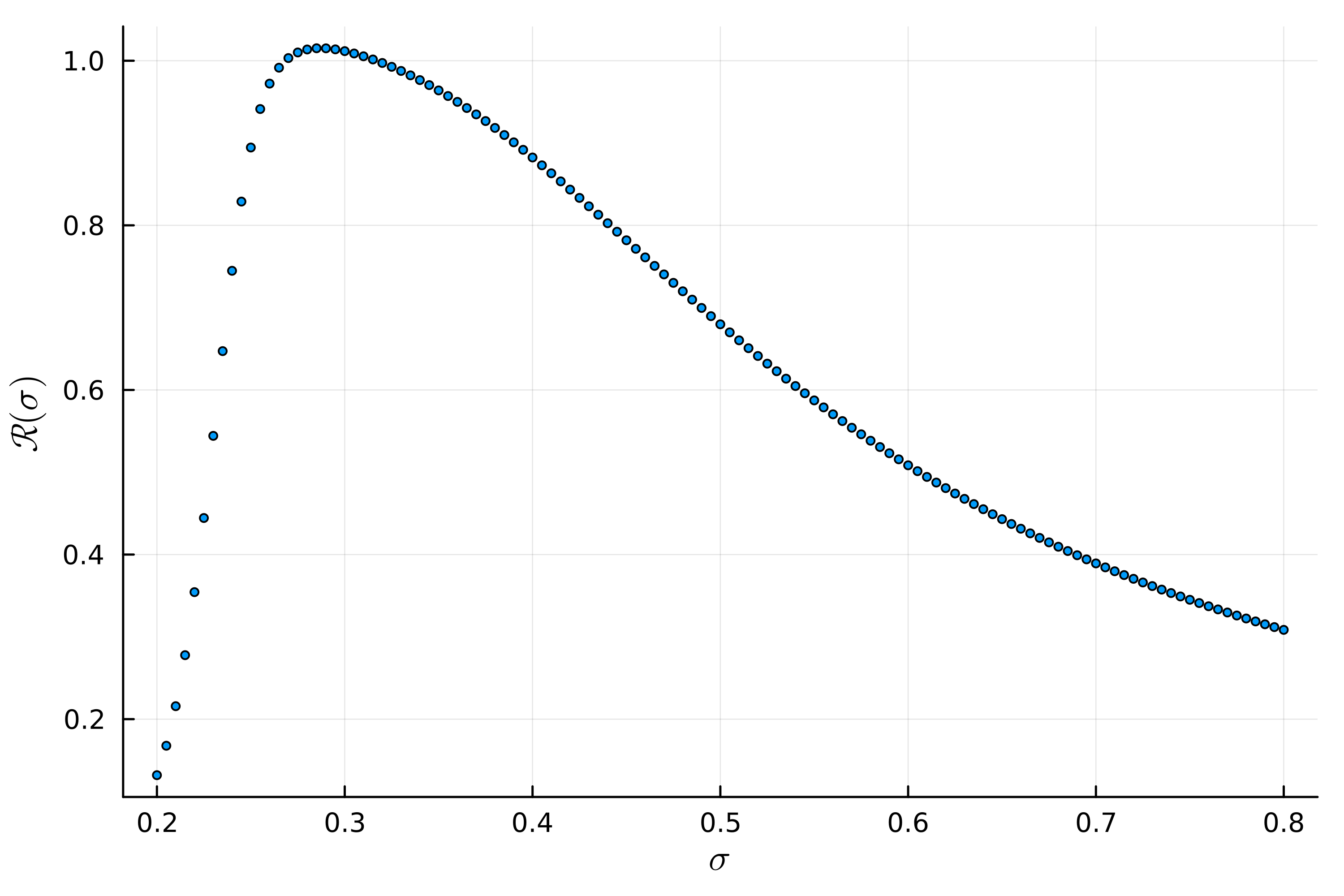}
    \caption{\centering The stochastic resonance indicator $\mathcal{R}_{\sigma}$, rigorously enclosed for $\eta = 0.12$, $\kappa = 1$, $\omega = 0.001$ and at $121$ equally spaced noise levels $\sigma\in [0.2, 0.8]$.\label{fig:indicator}}
\end{figure}
\begin{thm}\label{thm:stoc-res}
    Let $\eta = 0.12$, $\kappa = 1$ and $\omega = 0.001$, then there exists a local maximum 
    $$\sigma_{SR}\in(0.28712915\mathbf{1}, 0.28712915\mathbf{3})$$ for the function $\sigma \mapsto \cR(\varrho_{\sigma})$. Furthermore, $\cR(\varrho_{\sigma_{SR}})\geq 1.0152811464$.
\end{thm}
In fact, we are able to obtain much more quantitative information on $\cR$, summarised in the following statement, which provides a quantitative but rigorous description of the stochastic resonance in~\eqref{eqn:resonance}.
\begin{thm}\label{thm:stoc-res-pic}
    All the values of $\cR(\varrho_\sigma)$ depicted in Figure~\ref{fig:indicator} are correct up to $1.2\times 10^{-16}$. The precise values $(\sigma, \cR(\varrho_\sigma))$ used to produce this plot can be found in the file~\texttt{stoc\_res/Rsigmas} at~\cite{Chu2025Huggzz/Freud}.
\end{thm}
In order to prove Theorem~\ref{thm:stoc-res} and Theorem~\ref{thm:stoc-res-pic}, we rigorously enclose the stationary periodic density $\varrho_\sigma$ for various values of $\sigma$. That is, we obtain an approximate stationary periodic density $\bar\varrho_\sigma$ and an explicit error bound, as exemplified in Theorem~\ref{thm:rho-intro} for a specific value of $\sigma$. Since $\bar\varrho_\sigma$ is known explicitly, we can then rigorously compute $\cR(\bar\varrho_\sigma)$
and the error bound between $\bar\varrho_\sigma$ and $\varrho_\sigma$ then yields an error bound between $\bar\cR(\varrho_\sigma)$ and $\cR(\bar\varrho_\sigma)$. We emphasise that our approach allows for a much finer analysis than the mere computation of $\cR$, as we obtain a rigorous enclosure of the stationary periodic density, and therefore we could also get rigorous enclosures for different indicators of stochastic resonance, such as the spectral power amplification~\cite[§1.3]{Herrmann2013StochasticResonance}, as long as they can be explicitly computed from $\varrho_\sigma$.
\begin{rmk}
\label{rem:Rnoabs}
    Note that the absolute value in the definition of $\cR$ can be removed, i.e.,
    \begin{equation*}
        \mathcal{R}(\varrho) = \max_{t\in \T/\omega}\int_{\R}x\varrho(t, x)\d x,    \end{equation*}
        see Remark~\ref{rem:Rabsolutevalue}, which makes it easy to rigorously compute the maximum in time for a given $\bar\varrho$.
\end{rmk}

To the best of our knowledge, the stationary periodic density $\varrho$ has only been approximated numerically before using Monte Carlo simulations, which is by far the most widely used method to compute time-dependent averages of observables for SDEs. However, it is clear that our method provides bounds on $\varrho$ orders of magnitude smaller than what one can expect with a Monte-Carlo approximation. Furthermore, our method is much simpler computationally both in terms of memory and computation time, as using a Monte Carlo method for this problem would require to sample many realisations of $(x_{2\pi n/\omega +\tau })_{n\in\N}$ for each phase $\tau>0$.

In Section~\ref{sec:stoc-res}, we explain how the orthogonal polynomials introduced in Section~\ref{sec:weighted-Sobolev-spaces} and the quantitative study of their properties conducted in Section~\ref{sec:freud-sobolev} enable us to prove Theorem~\ref{thm:rho-intro}, Theorem~\ref{thm:stoc-res} and Theorem~\ref{thm:stoc-res-pic}. More precisely, in Section~\ref{sec:setting-resonance} we first rescale the problem and introduce suitable spaces for conducting a computer-assisted proof for $\varrho$. A high-level outline of the proof of Theorem~\ref{thm:rho-intro}, which highlights the main difficulties, is then given in Section~\ref{sec:outlineproof}, whereas the main quantitative estimates are presented in Section~\ref{sec:estimates-resonance}. All these ingredients are then combined in Section~\ref{sec:resonance-proofs} where the computer-assisted proofs of Theorem~\ref{thm:rho-intro}, Theorem~\ref{thm:stoc-res} and Theorem~\ref{thm:stoc-res-pic} are presented, whereas the most technical steps are postponed to Appendix~\ref{sec:appendix_resonance}.

\subsection{Setting for the computer-assisted proof}
\label{sec:setting-resonance}

Our main task is to find an explicit and rigorous enclosure for a time-periodic solution $\varrho=\varrho(t,x)$, with frequency $\omega$,
of the Fokker-Planck equation~\eqref{eqn:FP-p}.
In order to reuse the framework introduced in Section~\ref{sec:weighted-Sobolev-spaces} and Section~\ref{sec:freud-sobolev}, we first rescale the problem, using the change of variables
\begin{equation}\label{eqn:change-var}
\tilde{\eta} =\frac{2^{3/4}}{\sigma^{3/2}}\eta, \qquad \tilde{\kappa} =\frac{2^{1/2}}{\sigma}\kappa, \qquad \tilde{\omega} = \frac{2^{1/2}}{\sigma}\omega,\qquad s=\omega t,\qquad y= \frac{2^{1/4}}{\sigma^{1/2}}x,\end{equation}
and the change of unknown
\begin{equation}\label{eqn:resonance-u-def}
    \varrho(t, x) = (1+u(s, y))\frac{\sqrt[4]{2}}{\sqrt{\sigma}}\frac{e^{-y^4/4+\tilde{\kappa} y^2/2}}{\cZ}, \qquad \cZ = \int_\R e^{-y^4/4+\tilde{\kappa} y^2/2} \d y,
\end{equation}
which yields the following equation for $u$:
\begin{equation}
\label{eq:utilde}
    -\tilde{\omega} \pdiff{u}{s} + (-y^3+ \tilde{\kappa} y)\pdiff{u}{y} +\pdiff{^2u}{y^2}+\tilde{\eta}\cos s\left[(y^3-\tilde{\kappa} y) u -\pdiff{u}{y}\right] = \tilde{\eta}(-y^3+\tilde{\kappa} y)\cos s.
\end{equation}

\begin{rmk}
    The normalisation factor $\cZ$, which appears in several estimates to come, can be expressed as
\begin{equation*}
    \cZ = \frac{\pi}{2}e^{\tilde{\kappa}^2/8}\sqrt{\tilde{\kappa}}\left(\mathcal{I}_{-1/4}(\tilde{\kappa}^2/8)+\mathcal{I}_{1/4}(\tilde{\kappa}^2/8)\right),
\end{equation*}
where $\mathcal{I}_{\gamma}$ denotes the modified Bessel function of the first kind with parameter $\gamma$, and thus can be rigorously computed using \texttt{Arb}~\cite{Johansson2017ArbArithmetic}. 
\end{rmk}


With the operator
\begin{equation*}
    \cL = V'(y)\pdiff{ }{y}  - \pdiff{^2}{y^2},\qquad V(y) = \frac{y^4}{4}-\tilde{\kappa}\frac{y^2}{2},
\end{equation*}
and the so-called probability flux (or current) operator
$$\mathcal{J} = -V'(y) +\pdiff{}{y} = - \mathcal{D}^T,$$
both already introduced in Section~\ref{sec:weighted-Sobolev-spaces}, the equation satisfied by $u$ writes 
\begin{equation}\label{eqn:resonance-u-eq}
    \tilde{\omega} \pdiff{u}{s} + \mathcal{L} u+\tilde{\eta}\cos s\mathcal{J}u = 
    \tilde{\eta} V'(y)\cos s.
\end{equation}
In order to enforce that $\rho$ has mass one, we also impose that
\begin{equation}\int_{-\infty}^{+\infty} u(s, y) \nu(y) \d y = 0 \qquad \mbox{for all $s\in \T$},\end{equation}
where $\nu(\d y) = (e^{-V(y)}/\cZ)\d y$.
With this zero-mean condition, we instead choose to solve for $v = (\tilde{\omega}\partial_s +\mathcal{L})u$ such that Eq.~\eqref{eqn:resonance-u-eq} is equivalent to the linear problem
%
%
%
\begin{equation}
\label{eq:Fresonance}
    \cF v = g,
\end{equation}
where $\mathcal{F} = \id + \tilde{\eta} \cos{s}\mathcal{J}(\tilde{\omega}\partial_s+\mathcal{L})^{-1}$, and $g(s,y) = \tilde{\eta} V'(y)\cos s$. It is for this problem~\eqref{eq:Fresonance} that we are going to obtain computable error bounds between the exact solution and an approximate one. 

\begin{rmk}
    Solving for $v$ instead of for $u$ can be interpreted as changing the topology in which we solve the problem. This is not mandatory, but leads to slightly less cumbersome calculations in practice, essentially because some of the finer estimates for $\mathcal{J}(\tilde{\omega}\partial_s+\mathcal{L})^{-1}$ turned out to be somewhat simpler than their counterparts for $(\tilde{\omega}\partial_s+\mathcal{L})^{-1}\mathcal{J}$.
\end{rmk}


We now introduce the Banach space in which we are going to solve~\eqref{eq:Fresonance}. We consider the sequence space
\begin{equation*}
    X := \ell^1\left(\ell^2(\N^*)\right)\subset \ell^2(\N^*)^{\Z},
\end{equation*}
that is, 
\begin{equation*}
    X = \left\{ a = \left( a_{m,n} \right)_{m\in\Z,n\in\N^*},\ \left\Vert a\right\Vert_X := \sum_{m\in\Z}\left(\sum_{n\in\N^*}a_{m,n}^2\right)^{1/2}<\infty \right\},
\end{equation*}
and the corresponding function space, incorporating some symmetries of the problem:
\begin{align*}
\mathcal{X} := \left\{u:(s,y)\mapsto \sum_{m\in2\Z}\sum_{n\in\N^*}a_{m, n}e^{i m s}p_{2n}(y)+\sum_{m\in(2\Z+1)}\sum_{n\in\N^*}a_{m,n}e^{i m s}p_{2n-1}(y) \mid a\in X\right\} \subset \cC^0(\T,L^2(\nu)).
\end{align*}
The norm on $\cX$ is the one inherited from $X$, that is, $\left\Vert u \right\Vert_{\cX} := \left\Vert a \right\Vert_X$. Note that for all $u\in \mathcal{X}$
$$\|u\|_{\mathcal{C}^0(\mathbb{T}, L^2(\nu))} := \sup_{s\in \T}\|u(s, \cdot)\|_{L^2(\nu)} \leq \|u\|_{\mathcal{X}}.$$
The basis associated to $\cX$ is given by
\begin{equation*}
    \mathcal{E} = \bigcup_{m \in \Z}\{e^{ims}\}\otimes \mathcal{P}_{[m]}, \qquad \mathcal{P}_{[m]} := \{p_{[m]+2n-2}\}_{n\in\N^*},
\end{equation*}
where
\begin{equation*}
    [m] := 
    \begin{cases}
         1\quad \text{if }m\text{ is odd},\\
         2\quad \text{if }m\text{ is even}.
    \end{cases}
\end{equation*}
Furthermore, for an element of $\ell^2(\N^*)$ we denote by $\Pi_{\leq N}$ the projection on its first $N$ components. In particular, if $a_{m, \cdot} = (a_{m, n})_{n\in \N^*}\in \ell^2(\N^*)$ corresponds to an element of $\overline{\mathrm{Span}\, \mathcal{P}_{[m]}}$, then $\Pi_{\leq n} a_{m, \cdot} = (a_{m,n})_{1\leq n\leq N}$ corresponds to an element of $\mathrm{Span}\{p_{[m]+2n-2}\}_{n=1}^N$, and similarly $\Pi_{>N}a_{m, \cdot}= (a_{m,n})_{n> N}$ corresponds to an element of $\overline{\mathrm{Span}\{p_{[m]+2n-2}\}_{n>N}}$.



We consider an approximate solution $\bar{v}$ of~\eqref{eq:Fresonance}, obtained numerically, and then derive a computable error bound between $\bar{v}$ and the exact solution $v$. Note that, in contrast to Section~\ref{sec:GP_eq}, the problem considered here is linear, which means we only need computable bounds on the residual error and on $F^{-1}$, where $F = [\mathcal{F}]_{\mathcal{E}}$ is the representation of $F$ with respect to the basis $\mathcal{E}$. In practice, it turns out to be slightly sharper to first introduce an approximate inverse $A:X\to X$ of $F^{-1}$, as was already done in Section~\ref{sec:GP_eq}. Our main task is then to prove that
\begin{equation}
\label{eq:Tresonance}
    T = \id - AF :X\to X
\end{equation}
has norm strictly less than $1$. This then yields an error bound for $v$, and thus error bounds for the density $\varrho$ and for the stochastic resonance indicator $\cR(\varrho)$, as summarised in the following statement.
\begin{thm}
\label{thm:linearNK}
Let $\varrho$ the solution of~\eqref{eqn:FP-p} of mass one, and $v=v(s,y)$ given by 
$$ 
v(s,y)= (\tilde{\omega} \partial_s  + \mathcal{L})\left(\frac{\cZ \sqrt{\sigma}}{\sqrt[4]{2}}e^{V(y)}\varrho\left(\frac{s}{\omega},\frac{\sqrt{\sigma}y}{\sqrt[4]{2}}\right) -1\right).
$$
    Let $\bar{v}\in \mathcal{X}$ and $Y, Z>0$ be such that $\|A(F[\bar{v}]_{\mathcal{E}}-[g]_{\mathcal{E}})\|_{X} \leq  Y$ and $\|T\|_{X}\leq Z$. 
    If $Z<1$, then
    \begin{equation}\label{eqn:res-v-enclosure}\|v - \bar{v}\|_{\mathcal{X}}\leq \frac{Y}{1-Z}.\end{equation}
    Furthermore, denoting $\bar{\varrho} = \bar{\varrho}(t,x)$ given by
    $$ 
    \bar{\varrho}\left(\frac{s}{\omega},\frac{\sqrt{\sigma}y}{\sqrt[4]{2}}\right) =\left(1+(\tilde{\omega} \partial_s  + \mathcal{L})^{-1}\bar{v}(s,y)\right)\dfrac{\sqrt[4]{2}}{\sqrt{\sigma}} \dfrac{e^{-V(y)}}{\mathcal{Z}},
    $$
    we have the estimate
\begin{equation}\label{eqn:res-final-ineq}
    \|e^{V/2}(\varrho - \bar{\varrho})\|_{\mathcal{C}^0((\T/\omega)\times \R)}\leq  \sqrt[4]{2}\sqrt{\frac{C_{\cJ}+1}{\sigma\cZ}} C_P^{3/4}\|v-\bar{v}\|_{\mathcal{X}},
\end{equation}
where $C_{\cJ}$ is as in Lemma~\ref{lem:annoying-bounds} and $C_P$ is the Poincaré constant in~\eqref{eq:defCp}.
Finally,
\begin{align}
\label{eqn:est_Rsigma}
    |\cR(\varrho)-\cR(\bar\varrho)|\leq \frac{\sqrt{\sigma b_1}}{\sqrt[4]{2}} C_P \|v-\bar{v}\|_{\mathcal{X}},
\end{align}
with $b_1$ given by~\eqref{eqn:intro-bread-initial}.
\end{thm}
\begin{proof}By a Neumann series argument, we have that if $ \|T\|_{X}\leq Z<1$, then $AF$ is boundedly invertible and
$$\|(AF)^{-1}\|_{X} = \|(\id - (\id - AF))^{-1}\|_{X}\leq \|(\id - T)^{-1}\|_{X} \leq\frac{1}{1- Z}.$$
Therefore, since $\mathcal{F}v - g=0$
$$\|v - \bar{v}\|_{\mathcal{X}} = \|(AF)^{-1}A\left(F[v]_{\mathcal{E}}-[g]_{\mathcal{E}}-(F[\bar{v}]_{\mathcal{E}}-[g]_{\mathcal{E}}\right)\|_{X}\leq \|(AF)^{-1}\|_{X}\|A(F[\bar{v}]_{\mathcal{E}}-[g]_{\mathcal{E}})\|_{X}\leq \frac{Y}{1- Z}.$$
The embedding estimate~\eqref{eqn:res-final-ineq} is then a direct consequence of Lemma~\ref{lem:regularity}. Finally, noting that
\begin{align*}
    \left\vert \cR(\varrho) - \cR(\bar\varrho) \right\vert \leq \cR(\varrho-\bar\varrho),
\end{align*}
(see Remark~\ref{rem:Rabsolutevalue}), Lemma~\ref{lem:Rsigma} yields~\eqref{eqn:est_Rsigma}.
\end{proof}

Both Theorem~\ref{thm:rho-intro}, Theorem~\ref{thm:stoc-res} and Theorem~\ref{thm:stoc-res-pic} follow readily from Theorem~\ref{thm:linearNK}, provided we can obtain quantitative estimates $Y$ and $Z$. The derivation of such estimates is thus the most challenging part of the proofs of Theorem~\ref{thm:rho-intro}, Theorem~\ref{thm:stoc-res} and Theorem~\ref{thm:stoc-res-pic}.

\subsection{Main difficulties and outline of the strategy for obtaining the $Z$ bound}
\label{sec:outlineproof}
As is common in such computer-assisted proofs, our construction of $A$, which will be detailed below, relies on the fact that $\mathcal{F} = \id + \tilde{\eta} \cos{s}(\mathcal{J}(\tilde{\omega}\partial_s+\mathcal{L})^{-1})$ is a compact perturbation of the identity. 
From a theoretical point of view, the main ingredient required to derive the estimate $Z$ is then to quantify the compactness of $\mathcal{J}(\tilde{\omega} \partial_s +\mathcal{L})^{-1}$. Because we are looking for a periodic solution in time and therefore work with Fourier series, this essentially reduces to quantifying the compactness
on each fibre $\{e^{ims}\}\otimes L^2(\nu)$, that is, to study $\mathcal{J}(i\tilde{\omega} m +\mathcal{L})^{-1}$ on $L^2(\nu)$. Denoting 
\begin{align*}
    L_m :=\left[i\tilde{\omega} m +\mathcal{L}\right]_{\cP_{[m]}} \quad\text{and} \quad J_{[m]} := \left[\cJ\right]_{\cP_{[m]}\to\cP_{[m+1]}},
\end{align*}
these compactness estimates for fixed $m$ will be obtained by first computing a factorisation of $L_m$ the form
$$L_m = U_m^TU_m,$$
where $U_m$ is an upper-triangular infinite matrix. Then, we will prove that $J_{[m]} U_m^{-1}: \ell^2\to \ell^2$ is bounded and that $U_m^{-1}:\ell^2\to \ell^2$ is compact. We sometimes refer to $U_m^TU_m$ as a Cholesky decomposition of $L_m$, even though this is a slight abuse of terminology, as $L_m$ is not self-adjoint (because of the imaginary entries on its diagonal).

Compared to the situation of Section~\ref{sec:GP_eq}, where we dealt with the $m=0$ case, one extra technical difficulty that we have to face here is that we no longer have an explicit expression for the Cholesky factorisation. This will be dealt with by leveraging the tridiagonal structure of $L_m$, using ideas introduced in~\cite{Breden2015RigorousPart}. For any fixed $m\in\Z$, we will then be able to obtain quantitative compactness estimates for $U_m^{-1}$ and therefore for $\mathcal{J}(i\tilde{\omega} m +\mathcal{L})^{-1}$, i.e., compactness in space. Obtaining such estimates with an explicit decay in $m$, i.e., also getting quantitative compactness in time, would in principle be possible, but we can get away with only controlling $\mathcal{J}(i\tilde{\omega} m +\mathcal{L})^{-1}$ for finitely many $m$'s, and then showing some monotonicity with respect to $m$.

Let us now discuss a bit more the structure of the zero finding map $F$, and the associated construction of its approximate inverse $A$. Denoting
\begin{align*}
    B_m := \frac{\tilde{\eta}}{2} J_{[m]}L_m^{-1},
\end{align*}
we have
\begin{align*}
    F = \begin{pNiceMatrix}
        \ddots &\ddots & & & & &\\
        \ddots & &\Block[borders={bottom,left,top,right}]{1-1}{B_{-1}}& & & &\\
        & \Block[borders={bottom,left,top,right}]{1-1}{B_{-2}}& \id &  \Block[borders={bottom,left,top,right}]{1-1}{B_{0}}& & &\\
        & & \Block[borders={bottom,left,top,right}]{1-1}{B_{-1}}& \id &\Block[borders={bottom,left,top,right}]{1-1}{B_{1}} & &\\
        & & & \Block[borders={bottom,left,top,right}]{1-1}{B_{0}}&\id &\Block[borders={bottom,left,top,right}]{1-1}{B_{2}} &\\
        & & & &\Block[borders={bottom,left,top,right}]{1-1}{B_{1}} & &\ddots\\
        & & & & &\ddots &\ddots\\
    \end{pNiceMatrix}.
\end{align*}
That is, $F$ is block-tridiagonal, but each $B_m$ is full. Because we expect each $B_m:\ell^2\to \ell^2$ to be compact (compactness in space), and also $\|B_m\|_{\ell^2}$ to become arbitrarily small for large enough $m$ (compactness in time), we consider an approximate inverse $A$ of $F$ of the form
\begin{equation*}A = \begin{pNiceMatrix}[margin]
    \ddots & & & & & & & &\\
    & \Block[borders={top, bottom, left, right}]{1-1}{\id} & & & & & & & &\\
    & & \Block[borders={top, bottom, left, right}]{1-1}{\id} & & & & & & &\\
    & & & \Block[borders={top, bottom, left, right}]{3-5}{\bar{A}}& & & & & &\\
    & & & & & & & & & &\\
    & & & & & & & & & &\\
    & & & & & & & &\Block[borders={top, bottom, left, right}]{1-1}{\id} & &\\
    & & & & & & & & &\Block[borders={top, bottom, left, right}]{1-1}{\id} &\\
    & & & & & & & & & &\ddots
\end{pNiceMatrix}, \qquad \bar{A} = (A_{ij})_{-M\leq i,j\leq M},\qquad A_{ij} = \begin{pNiceMatrix}[margin]
  \Block[borders={bottom, right}]{3-3}{A^{11}_{ij}} & & & & & & & \\
  & & & & & & & \\
  & & & & & & & \\
  & & & \Block[borders={top, left}]{5-5}{\delta_{ij}\id}& & & &\\
  & & & & & & & \\
  & & & & & & &\\
  & & & & & & & \\
  & & & & & & & 
\end{pNiceMatrix}.
\end{equation*}
Here, $\delta_{ij}$ is the usual Kronecker symbol, $A^{11}_{ij} = \Pi_{\leq N}A_{ij}\Pi_{\leq N}$, and the finite-dimensional matrix $A^{11} = (A^{11}_{ij})_{-M\leq i, j\leq M}$ is taken as an approximate inverse of
\begin{equation}\label{eqn:F11}
\bar{F}^{11} :=\begin{pNiceMatrix}
    \id  &\Block[borders={left, bottom,top,right}]{1-1}{B^{11}_{1-M}} & & & & & & &\\
    \Block[borders={bottom,top,right}]{1-1}{B^{11}_{-M}}&\ddots &\ddots & & & & & &\\
    &\ddots &\Block[borders={left,top}]{1-1}{\id} &\Block[borders={bottom,left,top,right}]{1-1}{B^{11}_{-1}}& & & & &\\
    & & \Block[borders={bottom,left,top,right}]{1-1}{B^{11}_{-2}}&\id &  \Block[borders={bottom,left,top,right}]{1-1}{B^{11}_{0}}& & & &\\
    & & & \Block[borders={bottom,left,top,right}]{1-1}{B^{11}_{-1}}&\id  &\Block[borders={bottom,left,top,right}]{1-1}{B^{11}_{1}} & & &\\
    & & & & \Block[borders={bottom,left,top,right}]{1-1}{B^{11}_{0}}&\id  &\Block[borders={bottom,left,top,right}]{1-1}{B^{11}_{2}} & &\\
    & & & & &\Block[borders={bottom,left,top,right}]{1-1}{B^{11}_{1}} &\Block[borders={bottom,left,top,right}]{1-1}{\id }  &\ddots &\\
    & & & & & &\ddots &\ddots &\Block[borders={bottom,left,top}]{1-1}{B^{11}_{M}}\\
    & & & & & & & \Block[borders={bottom,left,top,right}]{1-1}{B^{11}_{M-1}} & \id 
\end{pNiceMatrix},
\end{equation}
where $B^{11}_m := \Pi_{\leq N}B_m \Pi_{\leq N}$. 
We emphasise that the precise way we construct $A^{11}$ is actually very important in practice, as it can have a major impact on the computational cost of the proof. This will be discussed further in Appendix~\ref{app:A-construction}.

Here and in the sequel, $\bar{\cdot}$ denotes a finite dimensional projection in time (typically, keeping the modes $\vert m\vert \leq M$), whereas $\cdot^{11}$ denotes a finite dimension projection in space. In particular, $\bar{F}^{11}$ is nothing but a finite dimensional projection in space of
$$\bar{F} := \begin{pNiceMatrix}
    \id  &\Block[borders={left, bottom,top,right}]{1-1}{B_{1-M}} & & & & & & &\\
    \Block[borders={bottom,top,right}]{1-1}{B_{-M}}&\ddots &\ddots & & & & & &\\
    &\ddots &\Block[borders={left,top}]{1-1}{\id} &\Block[borders={bottom,left,top,right}]{1-1}{B_{-1}}& & & & &\\
    & & \Block[borders={bottom,left,top,right}]{1-1}{B_{-2}}&\id &  \Block[borders={bottom,left,top,right}]{1-1}{B_{0}}& & & &\\
    & & & \Block[borders={bottom,left,top,right}]{1-1}{B_{-1}}&\id  &\Block[borders={bottom,left,top,right}]{1-1}{B_{1}} & & &\\
    & & & & \Block[borders={bottom,left,top,right}]{1-1}{B_{0}}&\id  &\Block[borders={bottom,left,top,right}]{1-1}{B_{2}} & &\\
    & & & & &\Block[borders={bottom,left,top,right}]{1-1}{B_{1}} &\Block[borders={bottom,left,top,right}]{1-1}{\id }  &\ddots &\\
    & & & & & &\ddots &\ddots &\Block[borders={bottom,left,top}]{1-1}{B_{M}}\\
    & & & & & & & \Block[borders={bottom,left,top,right}]{1-1}{B_{M-1}} & \id 
\end{pNiceMatrix},$$
which is itself a projection of $F$ onto the Fourier modes $\vert m\vert\leq M$.

Now that at least the structure of $A$ is fixed, we can look at $T=\id-AF$, for which we get
\begin{align}
\label{eq:Tmatrix}T = -\begin{pNiceMatrix}[margin]
    \ddots&\ddots &\ddots & & & & & & & &\\
    & \Block[borders={top, bottom, left}]{1-1}{B_{\scriptscriptstyle-M-3}} & \Block[borders={top, bottom, left, right}]{1-1}{0} &\Block[borders={top, right}]{1-1}{B_{\scriptscriptstyle-M-1}} & & & & &\\
    & &\Block[borders={bottom, left}]{1-1}{B_{\scriptscriptstyle-M-2}} & \Block[borders={top, bottom, left, right}]{1-1}{0} &\Block[borders={top, right}]{1-1}{B_{\scriptscriptstyle-M}} & & & & &\\
    & & &\Block[borders={bottom, left}]{1-1}{A_{\scriptscriptstyle -M,-M}B_{\scriptscriptstyle -M-1}} & \Block[borders={top, bottom, left, right}]{5-3}{}& & &\Block[borders={top, right}]{1-1}{A_{\scriptscriptstyle -M,M}B_{\scriptscriptstyle M+1}} & & &\\
    & & &\Block[borders={bottom, left}]{1-1}{A_{\scriptscriptstyle 1-M,-M}B_{\scriptscriptstyle -M-1}} & & & &\Block[borders={top, bottom, right}]{1-1}{A_{\scriptscriptstyle 1-M,M}B_{\scriptscriptstyle M+1}} & & &\\
    & & &\vdots & & \bar{A}\bar{F}-\id & &\vdots & & &\\
    & & &\Block[borders={top,bottom, left}]{1-1}{A_{\scriptscriptstyle M-1,-M}B_{\scriptscriptstyle -M-1}} & & & &\Block[borders={top, bottom, right}]{1-1}{A_{\scriptscriptstyle M-1,M}B_{\scriptscriptstyle M+1}} & & &\\
    & & &\Block[borders={bottom, left}]{1-1}{A_{\scriptscriptstyle M,-M}B_{\scriptscriptstyle -M-1}} & & & &\Block[borders={right}]{1-1}{A_{\scriptscriptstyle M,M}B_{\scriptscriptstyle M+1}} & & &\\
    & & & & & &\Block[borders={bottom, left}]{1-1}{B_{\scriptscriptstyle M}} &\Block[borders={top, bottom, left, right}]{1-1}{0} &\Block[borders={top, right}]{1-1}{B_{\scriptscriptstyle M+2}} & &\\
    & & & & & & &\Block[borders={bottom, left}]{1-1}{B_{\scriptscriptstyle M+1}} &\Block[borders={top, bottom, left, right}]{1-1}{0} &\Block[borders={top, bottom, right}]{1-1}{B_{\scriptscriptstyle M+3}} &\\
    & & & & & & & &\ddots &\ddots &\ddots
\end{pNiceMatrix}.
\end{align}
Equivalently, writing $T=\left(T_{ij}\right)_{i,j\in\Z}$ the decomposition of $T$ in terms of temporal modes, we have
\begin{equation*}T_{ij} = \begin{cases}
    -B_j\quad &\mbox{if $i=j\pm 1$ and $|i|>M$},\\    
    \delta_{ij}\id - (A_{i,j-1}B_j +A_{i,j} + A_{i, j+1}B_j)\quad &\mbox{if $|i|,|j|\leq M-1$},\\
    \delta_{ij}\id - (A_{i,j\mp1}B_j +A_{i,j})\quad &\mbox{if $|i|\leq M$ and $j = \pm M$},\\
    -A_{i,j\mp1}B_j \quad &\mbox{if $|i|\leq M$ and $j = \pm (M+1)$},\\
    0 \quad &\mbox{otherwise}. \\
\end{cases}
\end{equation*}
Note that, if $|i|> M$ or $|j|>M$, then $A_{ij} = \delta_{ij}\id$, therefore we have the following more concise formula: 
\begin{align}
     \label{eq:Tij}
     T_{ij} &=  
    \delta_{ij}\id - (A_{i,j-1}B_j +A_{i,j} + A_{i, j+1}B_j) \nonumber \\
    &=  
    \delta_{ij}\id - A_{i,j} -  (A_{i,j-1}+ A_{i, j+1})B_j,
\end{align}
which is valid for all $i,j\in\Z$. We are going to obtain the estimate $Z$ by estimating each block $T_{i,j}$ independently.

\subsection{Estimates needed for the proof}
\label{sec:estimates-resonance}

We give here computable estimates $Y$ and $Z$ allowing to apply Theorem~\ref{thm:linearNK}. Explicit formulae for some of the constant appearing in these estimates, as well as their proof, can be found in Appendix~\ref{sec:appendix_resonance}.
We note that most of our estimates depend on the value of $\kappa$ in~\eqref{eq:utilde}, via the orthogonal polynomials $p_n$, and more precisely via estimates like~\eqref{eqn:intro-bn_bounds} on the behaviour of the recurrence coefficients $b_n = a_n^2$. Thus, any quantitative statement must be stated for a given value of $\kappa$. We do so below with the $\kappa$ used for proving Theorem~\ref{thm:rho-intro}, but analogous statements can of course be obtained for other values of $\kappa$, which is how we obtain Theorem~\ref{thm:stoc-res-pic}.

\begin{rmk}
    While we consider a single value of $\kappa$ in all our examples and only vary $\sigma$, we recall that the rescaled $\tilde{\kappa}$ from~\eqref{eqn:change-var} depends on $\sigma$, which is why we do need to consider different values of $\tilde{\kappa}$ for proving Theorem~\ref{thm:stoc-res-pic}.
\end{rmk}


The $Y$ estimate itself does not present any difficulty, provided the approximate solution $\bv$ is chosen appropriately. We first find numerically an approximate solution
\begin{equation}
\label{eq:space_ubar}
    \bu \in\mathrm{Span}\left\{ \bigcup_{\vert m\vert \leq M-1} \{e^{ims}\}\otimes \Pi_{\leq N-[m]}\cP_{[m]}\right\}
\end{equation} to equation~\eqref{eqn:resonance-u-eq}, and then define $\bar{v} := (\tilde{\omega}\partial_s+\mathcal{L})\bar{u}$. Thus, $\mathcal{F}(\bar{v}) = (\tilde{\omega}\partial_s + \mathcal{L}) \bu + \tilde{\eta}\cos s(\mathcal{J} \bu - V')$ belongs to 
\begin{equation*}
    \mathrm{Span}\left\{ \bigcup_{\vert m\vert \leq M} \{e^{ims}\}\otimes \Pi_{\leq N}\cP_{[m]}\right\},
\end{equation*}
therefore $F[\bv]_{\cE}$ can be computed exactly using interval arithmetic. Since $A$ is a finite-dimensional perturbation of the identity, $AF[\bv]_{\cE}$ can then also be computed exactly, and we simply take
\begin{equation}
    \label{eq:Yresonance}
    Y = \|AF[\bar{v}]_{\mathcal{E}}\|_{X} .
\end{equation}

Getting a computable $Z$ estimate is significantly more involved, and this is where all the quantitative compactness estimates come into play. According to Lemma~\ref{lem:seq-norm},
\begin{align*}
    \Vert T\Vert_X \leq \sup_{j\in \Z}\sum_{i\in\Z} \Vert T_{ij} \Vert_{\ell^2},
\end{align*}
which is our starting point for getting a computable $Z$ bound. The following lemma allows us to bypass quantitative compactness estimates in time, and to get away with estimating only finitely many of the $\Vert T_{ij} \Vert_{\ell^2}$.
\begin{lem}
    \label{lem:Bmdecreasing}
    For all $m\in\N$ and $k\in\N$, $\left\Vert B_{m+2k}\right\Vert_{\ell^2} \leq \left\Vert B_{m}\right\Vert_{\ell^2}$, and $\left\Vert B_{-m-2k}\right\Vert_{\ell^2} \leq \left\Vert B_{-m}\right\Vert_{\ell^2}$.
\end{lem}
\begin{proof}
We have that
    \begin{align*}\|B_{\pm (m+2k)}\|_{\ell^2\to\ell^2}  &= \|J_{[m]}L_{\pm (m+2k)}^{-1}\|_{\ell^2\to\ell^2} \\
    &\leq\|J_{[m]}L_{\pm m}^{-1}\|_{\ell^2\to\ell^2}\|L_{\pm m}L_{\pm (m+2k)}^{-1}\|_{\ell^2\to\ell^2}\\
    &=\|B_{\pm m}\|_{\ell^2\to\ell^2}\|L_{\pm m}L_{\pm (m+2k)}^{-1}\|_{\ell^2\to\ell^2}.
    \end{align*}
    Without loss of generality, assume that $m$ is even and denote by $(\rho_j, v_j)_{j\in \N}$ the eigenpairs of $L_0$, where the $\rho_j$'s are real positive and the $v_j$'s are orthonormal. Since $L_m =\left[i\tilde{\omega} m +\mathcal{L}\right]_{\cP_{[m]}}$, we have that $L_m v_j = (\rho_j+i\tilde{\omega} m)v_j $ for all $j$. Then for any $w = \sum_{j\in \N}w_j v_j \in \ell^2$,
    \begin{equation*}\|L_{\pm m}L^{-1}_{\pm (m +2k)}w\|_{\ell^2}^2
    = \sum_{j\in \N}\left|\frac{\rho_j \pm i\tilde{\omega} m}{\rho_j\pm i\tilde{\omega} (m +2k)}\right|^2w_j^2\leq \sup_{j\in \N}\frac{\rho_j^2 +\tilde{\omega}^2m^2}{\rho_j^2+\tilde{\omega}^2(m+2k)^2}\sum_{j\in \N}w_j^2\leq\|w\|_{\ell^2}^2. \qedhere
    \end{equation*}
\end{proof}
According to~\eqref{eq:Tmatrix} (or equivalently, to~\eqref{eq:Tij}), for all $j \geq M+2$ we have
\begin{align*}
    \sum_{i\in\Z} \Vert T_{ij} \Vert_{\ell^2}  =  \Vert T_{j-1,j} \Vert_{\ell^2} + \Vert T_{j+1,j} \Vert_{\ell^2} 
     = \Vert B_j \Vert_{\ell^2} + \Vert B_{j+2} \Vert_{\ell^2},
\end{align*}
and similarly, for all $j \leq -M-2$,
\begin{align*}
    \sum_{i\in\Z} \Vert T_{ij} \Vert_{\ell^2} & = \Vert B_j \Vert_{\ell^2} + \Vert B_{j-2} \Vert_{\ell^2}.
\end{align*}
Using Lemma~\ref{lem:Bmdecreasing}, and then once more the structure of $T$ (see again~\eqref{eq:Tmatrix}), we get
\begin{align*}
    \Vert T\Vert_X \leq \max_{\vert j\vert \leq M+3}\sum_{i\in\Z} \Vert T_{ij} \Vert_{\ell^2} = \max_{\vert j\vert \leq M+3}\sum_{\vert i\vert \leq M+4} \Vert T_{ij} \Vert_{\ell^2}.
\end{align*}
We thus reduced the problem of getting a computable $Z$ estimate to that of having to derive finitely many computable estimates $Z_{ij}$ such that
\begin{equation*}
    \|T_{ij}\|_{\ell^2} \leq Z_{ij} \quad \text{for }\vert i\vert \leq M+4,\ \vert j\vert \leq M+3.
\end{equation*}
Once such estimates are obtained, we consider
\begin{equation}
\label{eq:defZ1}
    Z := \max_{\vert j\vert \leq M+3}\sum_{\vert i\vert \leq M+4} Z_{ij},
\end{equation}
which then satisfies $\|T\|_{X} \leq Z$.

\begin{rmk}
\label{rem:Z1vsM}
Let us briefly explain why we can expect such a $Z$ to be less than one, i.e., why we can expect the $\ell^1$ norm of each ``block column'' of $T$ in~\eqref{eq:Tmatrix} to be less than one (where each entry of a block column is in fact an operator on $L^2$), at least when $M$ is taken large enough. For most of the block columns this is clear, as $\|B_j\|_{\ell^2}\to 0$ when $|j|\to \infty$, and $A$ was chosen so that $\bar{A}\bar{F} \approx \id$ (at least provided the truncation level $N$ used for defining $A^{11}$ is also taken large enough, this will be discussed more precisely in Remark~\ref{rem:Z1vsN}). For the columns of $T$ of index $\pm(M+1)$, this might be less obvious at first glance: even if $\|B_{\pm(M+1)}\|_{\ell^2}$ decrease with $M$, the number of non-zero terms on these two columns increases with $M$. However, given the block-tridiagonal structure of $\bar{F}^{11}$ (see Eq.~\eqref{eqn:F11}) with the off-block-diagonal elements $\|B_j\|_{\ell^2}\to 0$ as $|j|\to \infty$, we expect $\|A_{ij}\|_{\ell^2}$ to decay exponentially with $|i-j|$, and $\|A_{jj}\|_{\ell^2}$ to be of order $1$ for large $|j|$. Therefore, we expect the sum $\sum_{\vert i\vert \leq M} \|A_{i,\pm M}\|_{\ell^2}$ to be bounded uniformly in $M$. Using once more that $\|B_j\|_{\ell^2}\to 0$ as $|j|\to \infty$, the norm of the whole column, i.e.,
$$\sum_{|i|\leq M+4}\|T_{i,\pm (M+1)}\|_{\ell^2} = \|B_{\pm (M+1)}\|_{\ell^2} + \sum_{|i|\leq M}\|A_{i, \pm M}B_{\pm (M+1)}\|_{\ell^2},$$
should indeed become less than $1$ for $M$ large enough.

In practice, we may need to take $M$ relatively large for this to hold, as $\|B_{j}\|_{\ell^2}$ is of order $|\tilde{\omega} j|^{-1/2}$, where, for our examples, $\tilde{\omega} = \frac{0.001\sqrt{2}}{\sigma} \ll 1$. 
\end{rmk}
%

Even if we only have finitely many estimates $\|T_{ij}\|_{\ell^2} \leq Z_{ij}$ to obtain, each $T_{ij}$ itself is still an infinite dimensional operator (acting on $L^2(\nu) \simeq \ell^2$). Therefore, we further carve the $T_{ij}$'s in four parts

\begin{equation}\label{eqn:Tcarving}T_{ij} = \begin{pNiceMatrix}[margin]
  \Block[borders={bottom, right}]{2-3}{T^{11}_{ij}} & & &\Block[borders={bottom}]{2-5}{T^{12}_{ij}} & &  & & \\
  & & & & & & & \\
  \Block[borders={}]{3-3}{T^{21}_{ij}}& & & \Block[borders={top, left}]{3-5}{T^{22}_{ij}}& & & &\\
  & & & & & & & \\
  & & & & & & &
\end{pNiceMatrix},\end{equation}
where $T_{ij}^{11} = \Pi_{\leq N} T_{ij}\Pi_{\leq N-[j]}$, $T_{ij}^{12} = \Pi_{\leq N} T_{ij}\Pi_{> N-[j]}$, $T_{ij}^{21} = \Pi_{> N} T_{ij}\Pi_{\leq N-[j]}$ and $T_{ij}^{22} = \Pi_{> N} T_{ij}\Pi_{> N-[j]}$. 

We are going to estimate each of the four blocks separately, i.e., find constants $Z_{ij}^{kl}$ such that $\|T^{kl}_{ij}\|_{\ell^2}\leq Z^{kl}_{ij}$, and then take
\begin{equation}
\label{eq:defZij}
    Z_{ij} := \left\|\begin{pmatrix}
    Z^{11}_{ij} & Z^{12}_{ij}\\
    Z^{21}_{ij} & Z^{22}_{ij}
\end{pmatrix}\right\|_{\ell^2}.
\end{equation}
We emphasise that we did not use the same truncation level on the left (namely $N$) and on the right (namely $N-[j]$) in order to define the four blocks within $T_{i,j}$. This is not mandatory, but will allow us to make some of the bounds $Z_{ij}^{kl}$ simpler and sharper. In particular, the fact that the projection on the right depends on the parity of $j$ is related to the fact that $J_2$ starts with a row of zeros but $J_1$ does not (see~\eqref{eqn:J-def}, and also Remark~\ref{rmk:weird-carving}).

In the next statement, we give computable estimates $Z_{ij}^{kl}$. This constitutes the most technical step of our proof of Theorem~\ref{thm:rho-intro}, and we refer to Appendix~\ref{sec:appendix_resonance} for the explicit formula of the many constants appearing here, for the precise definition of the operator $U_j$, as well as for the proof of the proposition.
\begin{prop}
    \label{prop:Zijkl}
    Let $\kappa = 1$, $\eta = 0.12$, $\omega = 0.001$, $\sigma = 0.287129152$, and $\tilde{\kappa}$, $\tilde{\eta}$ and $\tilde{\omega}$ the corresponding rescaled parameters given by~\eqref{eqn:change-var} and let $C_{\alpha,d}$, $C_{\beta,d}, C_d$ and $\vartheta$ be the constants defined in Lemma~\ref{lem:tridiag_rec} for $m=j$. For all $i,j \in \Z$, with $\vert j\vert \leq M +3= 1004$, define
        \begin{align*}
        \epsilon^{\leq N-[j]}_{ij} := \|T_{ij}^{11}\|_{\ell^2} = \|\delta_{ij}\Pi_{\leq N}  - (A^{11}_{i,j-1}B^{11}_{j-1} +A^{11}_{i,j} + A^{11}_{i, j+1}B^{11}_{j+1}) )\Pi_{\leq N-[j]}\|_{\ell^2},
    \end{align*}
    and
    \begin{align*}
        \epsilon^{> N-[j]}_{ij} := \|(\delta_{ij}\Pi_{\leq N}  - (A^{11}_{i,j-1}B^{11}_{j} +A^{11}_{i,j} + A^{11}_{i, j+1}B^{11}_{j}))\Pi_{> N-[j]}\|_{\ell^2}.
    \end{align*}
Furthermore, let $U_j$ be the upper triangular operator introduced in~\eqref{def:Um} such that $L_j = U_j^TU_j$, and let
\begin{align*}
    Z^{11}_{ij} &:= \epsilon^{\leq N-[j]}_{ij}, \\
    Z^{12}_{ij} &:= \epsilon^{> N-[j]}_{ij} + \frac{\tilde{\eta}}{2}|(U_j)_{N, N+1}|\frac{\|(A^{11}_{i,j-1}+A^{11}_{i,j+1}) J_{[j]}^{11} (U_{j}^{-1})_{:N,N}\|_{\ell^2}}{C_d^2\left(2N(1-\vartheta^2)\right)^{3/2}},\\
    Z^{21}_{ij} &:= (\delta_{i,j+1}+\delta_{i,j-1})\frac{\tilde{\eta}(C_{\alpha, d}+C_{\beta,d})}{2(1-\vartheta)}|(U_j)_{N-[j],N-[j]+1}|\|(U^{-1}_j)_{N-[j]+1:, N-[j]+1}\|_{\ell^2}\|(U^{-1}_{j})_{:N-[j],N-[j]}\|_{\ell^2},\\
    Z^{22}_{ij} &:=(\delta_{i,j+1}+\delta_{i,j-1})\frac{\tilde{\eta}(C_{\alpha, d}+C_{\beta,d})}{2(1-\vartheta)^2C_d(2N-2)^{3/4}}.
\end{align*}
Then $Z_{ij}$ given by~\eqref{eq:defZij} satisfies $\|T_{ij}\|_{\ell^2} \leq Z_{ij}$.
\end{prop}
The proof of this proposition is contained in Appendix~\ref{sec:proofZijkl}, and relies on the following steps:
\begin{enumerate}
    \item Since $L_j$ is a tridiagonal operator, $U_j$ is an upper bidiagonal operator analogous to $D_{> 0}$ in Section~\ref{subsec:compact-est} (recall from Proposition~\ref{prop:L-decomp} that $L_0 = D_{> 0}^TD_{> 0}$). However, the entries of $U_j$ are far from explicit and need to be defined by a recurrence relation (which depends itself on the $a_n$'s defined by~\eqref{eqn:jacobi-rec-rel_intro} which are not explicit). We study the asymptotic behaviour of these entries using ideas from~\cite{Breden2015RigorousPart} to obtain estimates analogous to~\eqref{eqn:alpha-bound}.
    \item Then the derivation of the $Z_{ij}^{kl}$ estimates is reduced to calculations analogous to the proofs of Theorem~\ref{thm:quant-poinc} and Lemma~\ref{lem:annoying-bounds}.
\end{enumerate}

\begin{rmk}\label{rem:Z1vsN}
    By construction of $\bar{A}$, we expect $\epsilon^{\leq N-[j]}_{ij}$ and $\epsilon^{> N-[j]}_{ij}$ to be very small when $|i|,|j|\leq M$, and still small otherwise by virtue of $B_j$ being small for large $j$. Furthermore, since $U$ behaves similarly to the operator $D_{>0}$ from Section~\ref{subsec:compact-est}, we expect
\begin{itemize}
    \item $|(U_j)_{N, N+1}|$ to grow like $N^{3/4}$ and $\|(A^{11}_{i,j-1}+A^{11}_{i,j+1}) J_{[j]}^{11}(U_{j}^{-1})_{:N,N}\|_{\ell^2}$ to be of order $1$ such that $Z^{21}_{ij}$ should decay like $N^{-3/4}$,
    \item $|(U_j)_{N-[j],N-[j]+1}|$ to grow like $N^{3/4}$ while $\|(U^{-1}_j)_{N-[j]+1:, N-[j]+1}\|_{\ell^2}$ and $\|(U^{-1}_{j})_{:N-[j],N-[j]}\|_{\ell^2}$ decay like $N^{-3/4}$ such that $Z^{21}_{ij}$ should decay like $N^{-3/4}$.
\end{itemize}
Overall, for $N$ and $M$ large enough, we can indeed expect each $Z_{ij}$ to be small, but also each $\sum_{i} Z_{ij}$ to be small, because the $Z_{ij}^{kl}$ that are not arbitrarily small, like $Z_{ij}^{22}$, are only non-zero twice for each fixed $j$.
\end{rmk}
The key aspect of Proposition~\ref{prop:Zijkl} is that the given estimates only involve computable constants, as well as norms of finite vectors and finite matrices. Therefore, each $Z_{ij}$ can be rigorously computed, or at least explicitly upper-bounded. Taken together, Proposition~\ref{prop:Zijkl} and the formula~\eqref{eq:defZ1}--\eqref{eq:defZij} give us a computable $Z$ estimate for Theorem~\ref{thm:linearNK}.

\subsection{Proofs of Theorem~\ref{thm:rho-intro}, Theorem~\ref{thm:stoc-res} and Theorem~\ref{thm:stoc-res-pic}}
\label{sec:resonance-proofs}

Let $\kappa = 1$, $\eta = 0.12$, $\omega = 0.001$, $\sigma = \sigma_1 = 0.287129152$, and the corresponding rescaled parameters given by~\eqref{eqn:change-var}. Let $N = 100$ and $M = 1001$. Let $\bu$ be given in~\texttt{stoc\_res/ubar\_plots} at~\cite{Chu2025Huggzz/Freud}, which satisfies~\eqref{eq:space_ubar}, and consider the approximate solution $\bv = (\tilde{\omega}\partial_s+\cL)^{-1} \bu$ of~\eqref{eq:Fresonance}  (corresponding to the density $\bar\varrho$ represented in Figure~\ref{fig:resonance-density-middle}).
We rigorously compute the $Y$ estimate given by~\eqref{eq:Yresonance}, which satisfies $\|AF[\bar{v}]_{\mathcal{E}}\|_{X} \leq  Y$. Then, we rigorously compute the $Z_{i,j}^{k,l}$ estimates given in Proposition~\ref{prop:Zijkl}. Using~\eqref{eq:defZij} and~\eqref{eq:defZ1}, we get a $Z$ estimate that satisfies $\|T\|_{X\to X}\leq Z$. We then check that $Z<1$, and apply Theorem~\ref{thm:linearNK}. In particular, we obtain that the exact solution $v$ of~\eqref{eq:Fresonance} satisfies the error bound $\Vert v -\bv \Vert_\cX \leq \frac{Y}{1-Z}$, and estimate~\eqref{eqn:res-final-ineq} then yields Theorem~\ref{thm:rho-intro}. These estimates are computed rigorously in the file~\texttt{stoc\_res/plots\_proof.jl} for $\sigma = \sigma_1$ as well as for $\sigma = \sigma_0, \sigma_2$.

The same procedure is repeated in~\texttt{stoc\_res/resonance\_proof.jl} for the values of $\sigma$ in Figure~\ref{fig:indicator}. In each case, we also rigorously compute $\cR(\bar\varrho)$ using Remark~\ref{rem:Rabsolutevalue}, and use estimate~\eqref{eqn:est_Rsigma} to obtain a rigorous enclosure of $\cR(\varrho)$, which yields Theorem~\ref{thm:stoc-res-pic}.
Finally, this analysis is again performed in~\texttt{stoc\_res/prec\_proof.jl} for
\begin{equation*}
    \sigma_1^- = 0.287129151,\qquad \sigma_1 = 0.287129152,\qquad \sigma_1^+ = 0.287129153,
\end{equation*}
and the obtained enclosures show that $\cR(\varrho_{\sigma_1}) > \cR(\varrho_{\sigma^-_1})$ and $\cR(\varrho_{\sigma_1}) > \cR(\varrho_{\sigma_1^+})$. Thus, $\sigma \mapsto \cR(\varrho_\sigma)$ admits a local maximum for $\sigma_{SR}\in[\sigma_1^-,\sigma_1^+]$, and $\cR(\varrho_{\sigma_{SR}}) \geq \cR(\varrho_{\sigma_{1}})$, which yields Theorem~\ref{thm:stoc-res}.

\section*{Acknowledgments}

HC thanks Pierre Faugère, Grigorios Pavliotis and Luca Ziviani for valuable discussions. HC is particularly grateful to Sheehan Olver for discussions at the early stage of this research and for his continued encouragements.
HC thanks Olivier Hénot for discussions about the library \texttt{IntervalArithmetic.jl}~\cite{david_p_sanders_2024_10459547}.
MB and HC thank the \textit{Centre de recherches mathématiques}, Montréal for its hospitality during the thematic program \textit{Computational Dynamics: Analysis, Topology and Data}. HC is grateful to G-Research and the Doris Chen mobility fund for making this visit possible. The authors also thank Yanbo Tang for his hospitality during the write-up of the present paper. HC was supported by an Hadamard Lecturer fellowship awarded by the \textit{Fondation Mathématique Jacques Hadamard}. This work was partially completed as part of HC's PhD thesis at Imperial College supported by a scholarship from the Imperial College London EPSRC DTP in Mathematical Sciences (EP/W523872/1) and by the EPSRC Centre for Doctoral Training in Mathematics of Random Systems: Analysis, Modelling and Simulation (EP/S023925/1).
MB and HC were supported by the ANR project CAPPS: ANR-23-CE40-0004-01.
MB and HC thank the CNRS--Imperial \textit{Abraham de Moivre} IRL for supporting this research.

\printbibliography[heading=bibintoc]

\begin{appendix}
\section{Proof of Lemma~\ref{lem:annoying-bounds}}
\label{app:annoying-bounds}

Again, we first consider the action of $\mathcal{J}$ and bound its norm on the subspace of even functions (we then repeat the same procedure for odd functions and take the maximum of the two estimates). Recall that we want to find a constant $C_{\cJ}$ such that
$$\|\mathcal{J}\|_{H^1\to L^2}=\left\|\left[\mathcal{J}\right]_{H^1\to L^2}\right\|_{\ell^2} =  \|D^{T}P^{-1}\|_{\ell^2}\leq C_{\cJ}.$$
Note that from $\{p_n\}_{n\in 2\N}$ to $\{p_n\}_{n\in 2\N+1}$, $\mathcal{J}$ has the representation

$$-D^T = -\begin{pmatrix}
    \alpha_{1} & & & &\\
     \beta_1& \alpha_{3} &  & &\\
    & \beta_3&\ddots &\\
    & & \ddots&& &\\
    & & & &
\end{pmatrix}.$$

Thus,
$$\|D^TP^{-1}\|_{\ell^2} \leq\|\mathrm{Diag}(\alpha_{1:2:})P^{-1}\|_{\ell^2}+\|\mathrm{Diag}(\beta_{1:2:})P^{-1}\|_{\ell^2}.$$

For instance, we have
$$\mathrm{Diag}(\alpha_{1:2:})P^{-1} = \begin{pNiceMatrix}[margin]
  \Block[borders={bottom, right}]{5-1}{\mathrm{Diag}(\alpha_{1:2:N+1})P_{\leq N}^{-1}} &\Block[]{5-1}{-\beta_N\mathrm{Diag}(\alpha_{1:2:N+1})(P_{\leq n}^{-1})_{:,-1}(D_{>N}^{-1})_{1,:}}\\
  &\\
  &\\
  &\\
  &\\
  & \Block[borders={top, left}]{5-1}{\mathrm{Diag}(\alpha_{N+3:2:})D_{>N}^{-1}}\\
  &\\
  &\\
  &\\
  &
\end{pNiceMatrix}.$$

Recall that we have from Section~\ref{sec:bread} that for some $N_0\in \N$, we can find $c^+>1$ and $c^{-}<1$ such that for all $n\geq N_0$ 
$$\frac{\sqrt[4]{3}}{\sqrt{c^+}}n^{3/4}\leq \alpha_n\leq\frac{\sqrt[4]{3}}{\sqrt{c^-}}n^{3/4}\qquad\text{and}\qquad \beta_n\leq\frac{(c^+)^{3/2}}{3^{3/4}}.$$
Mimicking the calculations of Section~\ref{subsec:compact-est}, we then find that
\begin{align*}
    \|\left(\mathrm{Diag}(\alpha_{N+3:2:})D_{>N}^{-1}\right)_{i,:}\|_{\ell^1}&=\alpha_{N+2i+1}\sum_{j = i}^{\infty}\left|\frac{(-1)^{i-j}}{\alpha_{N+2j}}\prod_{k=i}^{j-1}\frac{\beta_{N+2k}}{\alpha_{N+2k}}\right|\\
    &\leq \sqrt{\frac{c^+}{c^-}}\left(\frac{N+2i+1}{N+2i}\right)^{3/4}\sum_{j = i}^{\infty}\theta^{j-i}\\
    &\leq\frac{\sqrt{c^+}}{\sqrt{c^-}(1-\theta)}\left(\frac{N+3}{N+2}\right)^{3/4},
\end{align*}
and
\begin{align*}
    \|\mathrm{Diag}(\alpha_{N+3:2:})(D_{>N}^{-1})_{:,j}\|_{\ell^1}&=\sum_{i = 1}^{j}\left|\alpha_{N+2i+1}\frac{(-1)^{i-j}}{\alpha_{N+2j}}\prod_{k=i}^{j-1}\frac{\beta_{N+2k}}{\alpha_{N+2k}}\right|\\
    &\leq \sqrt{\frac{c^+}{c^-}}\left(\frac{N+2+1}{N+2j}\right)^{3/4}\sum_{i = 1}^{j}\theta^{j-i}\\
    &\leq\frac{\sqrt{c^+}}{\sqrt{c^-}(1-\theta)}\left(\frac{N+3}{N+2}\right)^{3/4},
\end{align*}
such that
$$\|\mathrm{Diag}(\alpha_{N+3:2:})D_{>N}^{-1}\|_{\ell^2}\leq \|\mathrm{Diag}(\alpha_{N+3:2:})D_{>N}^{-1}\|_{\ell^1}^{1/2}\|\mathrm{Diag}(\alpha_{N+3:2:})D_{>N}^{-1}\|_{\ell^\infty}^{1/2}\leq\frac{\sqrt{c^+}}{\sqrt{c^-}(1-\theta)}\left(\frac{N+3}{N+2}\right)^{3/4}.$$
This yields that
$$\|\mathrm{Diag}(\alpha_{1:2:})P^{-1}\|_{\ell^2}\leq c_{\alpha}:=\left\|\begin{pmatrix}
    \|\mathrm{Diag}(\alpha_{1:2:N+1})P_{\leq N}^{-1}\|_{\ell^2} & \beta_n\|\mathrm{Diag}(\alpha_{1:2:N+1})(P_{\leq N}^{-1})_{:,-1}\|_{\ell^2}\|(D_{>N}^{-1})_{1,:}\|_{\ell^2}\\
    0 & \displaystyle \frac{\sqrt{c^+}}{\sqrt{c^-}(1-\theta)}\left(\frac{N+3}{N+2}\right)^{3/4}
\end{pmatrix}\right\|_{\ell^2},$$
where
$$\|(D^{-1}_{>N})_{1,:}\|_{\ell^2}\leq\frac{1}{C_{\alpha}\sqrt{1-\theta^2}}\frac{1}{N^{3/4}}.$$

Similarly, we find that
$$\|\mathrm{Diag}(\beta_{N+3:2:})D_{>N}^{-1}\|_{\ell^2}\leq\frac{(c^+)^2}{3(1-\theta)}\left(\frac{N+3}{N+2}\right)^{3/4},$$
such that
$$\|\mathrm{Diag}(\beta_{1:2:})P^{-1}\|_{\ell^2}\leq c_{\beta}:=\left\|\begin{pmatrix}
    \|\mathrm{Diag}(\beta_{1:2:N+1})P_{\leq N}^{-1}\|_{\ell^2} & \beta_n\|\mathrm{Diag}(\beta_{1:2:N+1})(P_{\leq N}^{-1})_{:,-1}\|_{\ell^2}\|(D_{>N}^{-1})_{1,:}\|_{\ell^2}\\
    0 & \displaystyle \frac{(c^+)^2}{3(1-\theta)}\left(\frac{N+3}{N+2}\right)^{3/4}
\end{pmatrix}\right\|_{\ell^2}.$$
We thus choose $C_\cJ := c_{\alpha} +c_{\beta}$. 

Now, to finalise our estimates, first note that
$$\|e^{-V/2}u\|^2_{L^2(\R)} = \cZ\|u\|^2_{L^2(\nu)}\leq \max(1,C_P)\|u\|_{H^1(\nu)}^2.$$
Then, for all $u\in H^1(\nu)$
$$\|(e^{-V/2}u)'\|_{L^2(\R)} \leq \frac{\sqrt{\cZ}}{2}\left(\|\mathcal{J}u\|_{L^2(\nu)}+\|u\|_{H^1(\nu)}\right)\\
    \leq \frac{\sqrt{\cZ}}{2}\left(\|\mathcal{J}u\|_{H^1\to L^2}+1\right)\|u\|_{H^1(\nu)}\leq \frac{\sqrt{\cZ}}{2}\left(C_{\cJ}+1\right)\|u\|_{H^1(\nu)}.$$
By~\cite[Theorem VIII.7]{Brezis2011FunctionalEquations},
\begin{align*}
    \|e^{-V/2}u\|_{\infty} &\leq \sqrt{2}\|e^{-V/2}u\|_{L^2(\R)}^{1/2}\|(e^{-V/2}u)'\|_{L^2(\R)}^{1/2}\\
    &\leq\sqrt{\cZ(C_{\cJ}+1)}\|u\|_{L^2(\nu)}^{1/2}\|u\|_{H^1(\nu)}^{1/2}\\
    &\leq\sqrt{\cZ(C_{\cJ}+1)}\max(1, C_P)^{1/4}\|u\|_{H^1(\nu)},
\end{align*}
which yields~\eqref{eqn:annoying-bounds1}. Similarly,
\begin{align*}
    \|e^{-V/2}u\|_{H^1(\R)}:&=\left(\|e^{-V/2}u\|_{L^2(\R)}^2+\|(e^{-V/2}u)'\|_{L^2(\R)}^2\right)^{1/2}\\
    &\leq \sqrt{\cZ}\left(\|e^{-V/2}u\|_{L^2(\nu)}^2+\frac{(C_{\cJ}+1)^2}{4}\|u\|_{H^1(\nu)}^2\right)^{1/2}\\
    &\leq \sqrt{\cZ}\left(\max(1, C_P)+\frac{(C_{\cJ}+1)^2}{4}\right)^{1/2}\|u\|_{H^1(\nu)},
\end{align*}
which yields~\eqref{eqn:annoying-bounds2}.

In the case $\kappa = 4$, the calculation leading to an explicit value of $C_{\cJ}$ is performed with $N=3,500$ at~\cite{Chu2025Huggzz/Freud} in the file \texttt{GP\_eq/embedding.ipynb}.

\section{Proof of Proposition~\ref{prop:GP-bounds}}\label{app:GP-bounds}
In this appendix, we prove Proposition~\ref{prop:GP-bounds}, i.e., we derive the necessary bounds for the application of Theorem~\ref{thm:N-K}.
\subsection{The bound $Y$}

We estimate the difference $T(\bar{u}) - \bar{u}$ as follows
\begin{align*}
    \|T(\bar{u}) - \bar{u}\|_{H^1}^2 &=\|AF(\bar{u})\|_{H^1}^2\\
    &=\|A(\bar{u} - \tL^{-1} f(\bar{u}))\|_{H^1}^2\\
    &=\|\fPH A (\bar{u} - \tL^{-1}f(\bar{u}))\|_{H^1}^2+\|\iPH A(\bar{u}-\tL^{-1}f(\bar{u}))\|_{H^1}^2\qquad \mbox{By Parseval's identity}\\
    &=\| A_n(\bar{u}-\fPH\tL^{-1} f(\bar{u}))\|_{H^1}^2+\|\iPH (\bar{u}-\tL^{-1}f(\bar{u}))\|_{H^1}^2\\
    &=\| A_n(\bar{u} -\fPH\tL^{-1}\fPL f(\bar{u}))\|_{H^1}^2+\|\iPH\tL^{-1}f(\bar{u})\|_{H^1}^2 \qquad \mbox{using~\eqref{eqn:tL-zero}}.
\end{align*}
The first term requires a finite computation and can therefore evaluated exactly, while the second term has to be decomposed as follows
$$\|\iPH\tL^{-1}f(\bar{u})\|_{H^1}\leq \|\iPH\tL^{-1}\fPL f(\bar{u})\|_{H^1}+\|\iPH\tL^{-1}\iPL f(\bar{u})\|_{H^1}.$$
Now, we have
\begin{align*}
    \left[\iPH\tL^{-1}\fPL f(\bar{u})\right]_{H^1} = \left[\iPH\tL^{-1}\fPL\right]_{L^2\to H^1}\left[\fPL f(\bar{u})\right]_{L^2} = -\beta_n (D^{-1}_{>n})_{1,:}(P_{\leq n}^{-1})_{:,-1}^{T}\left[\fPL f(\bar{u})\right]_{L^2},
\end{align*}
so that
$$\|\iPH\tL^{-1}\fPL f(\bar{u})\|_{H^1} = \beta_n \|(D^{-1}_{>n})_{1,:}\|_{\ell^2}\left|(P_{\leq n}^{-1})_{:,-1}^{T}\left[\fPL f(\bar{u})\right]_{L^2}\right|,$$
where we recall~\eqref{eqn:Dn-estimate}
\begin{equation}\label{eqn:Dn-estimate-2}
    \|(D^{-1}_{>n})_{1,:}\|_{\ell^2}\leq\frac{1}{C_{\alpha}\sqrt{1-\theta^2}}\frac{1}{n^{3/4}}.
\end{equation}
We also have that
\begin{align*}
    \|\iPH\tL^{-1}\iPL f(\bar{u})\|_{H^1}
    &\leq\|\iPH\tL^{-1}\iPL\|_{L^2\to H^1} \|\iPL f(\bar{u})\|_{L^2}\\
    &\leq\|\iPH\tL^{-1}\iPL\|_{L^2\to H^1} \|\iPL (\omega \bar{u} + \bar{u}_0-e^{-V}\bar{u}^3)\|_{L^2}\\
    &\leq\|\iPH\tL^{-1}\iPL\|_{L^2\to H^1} \|\iPL (e^{-V}\bar{u}^3)\|_{L^2}\\
    &\leq\|\iPH\tL^{-1}\iPL\|_{L^2\to H^1} \left(\| e^{-V}\bar{u}^3\|_{L^2}^2-\|\fPL (e^{-V}\bar{u}^3)\|_{L^2}^2\right)^{1/2}\\
    &\leq \frac{C_{22}}{n^{3/4}} \left(\| e^{-V}\bar{u}^3\|_{L^2}^2-\|\fPL (e^{-V}\bar{u}^3)\|_{L^2}^2\right)^{1/2}\qquad\mbox{by~\eqref{eqn:tL-estimates}}.
\end{align*}
We can thus take

\begin{align*}
    Y^2 :=&\| A_n(\bar{u} -\fPH\tL^{-1}\fPL f(\bar{u}))\|_{H^1}^2\\
    +&\frac{1}{n^{3/2}}\left\{\frac{\beta_n}{C_{\alpha}\sqrt{1-\theta^2}}\left|(P_{\leq n}^{-1})_{:,-1}^{T}\left[\fPL f(\bar{u})\right]_{L^2}\right|+C_{22}\left(\| e^{-V}\bar{u}^3\|_{L^2}^2-\|\fPL (e^{-V}\bar{u}^3)\|_{L^2}^2\right)^{1/2}\right\}^2.
\end{align*}

\subsection{The bound $Z^{11}$}
For $\fh\in \fPH(\mathcal{H})$,
    $$\fPH DT(\bar{u})\fh = \fPH(\id - ADF(\bar{u}))\fh = (\fPH - A_n \fPH DF(\bar{u})\fPH)\fh,$$
    by construction of $A$. Then, we get the estimate
    $$\|\fPH DT(\bar{u})\fPH\|_{H^1, H^1} = \|\fPH - A_n \fPH DF(\bar{u})\fPH\|_{H^1,H^1} = \|[\fPH]_{H^1}- [A_n]_{H^1} [\fPH DF(\bar{u})\fPH]_{H^1}\|_{\ell^2}=: Z^{11}.$$
This quantity is expected to be very small as $A_n$ is constructed as a numerical inverse of $[\fPH DF(\bar{u})\fPH]_{H^1}$. This estimate is only the two-norm of a finite-dimensional matrix $M$ and we bound this quantity via the inequality $\|M\|_{\ell^2} \leq \|M\|_{\ell^1}^{1/2}\|M\|_{\ell^\infty}^{1/2}$.
\subsection{The bound $Z^{21}$}

For $\fh \in \fPH(\mathcal{H})$, we have

\begin{align*}
    \|\iPH DT(\bar{u})\fh\|_{H^1} &=\| \iPH(\fh - A DF(\bar{u})\fh\|_{H^1}\\
    &= \|\iPH DF(\bar{u})\fh\|_{H^1}\\
    &= \|\iPH(\fh - \tL^{-1}Df(\bar{u})\fh\|_{H^1}\\
    &= \|\iPH\tL^{-1}Df(\bar{u})\fh\|_{H^1}\\
    &\leq \|\iPH\tL^{-1}\fPL Df(\bar{u})\fh\|_{H^1} + \|\iPH\tL^{-1}\iPL Df(\bar{u})\fh\|_{H^1}.
\end{align*}
On the one hand, we have 
\begin{align*}
    \left[\iPH\tL^{-1}\fPL Df(\bar{u})\fh\right]_{H^1} &= \left[\iPH\tL^{-1}\fPL\right]_{L^2\to H^1}\left[\fPL Df(\bar{u})\fPH\right]_{H^1\to L^2}\left[\fh\right]_{H^1}\\
    &= -\beta_n (D_{>n}^{-1})_{1,:}((P_{\leq n}^{-1})_{:,n})^T\left[\fPL Df(\bar{u})\fPH\right]_{H^1\to L^2}\left[\fh\right]_{H^1},
\end{align*}
such that
\begin{align*}
    \|\iPH\tL^{-1}\fPL Df(\bar{u})\fh\|_{H^1}&\leq \beta_n \|(D_{>n}^{-1})_{1,:}\|_{\ell^2}\left\|((P_{\leq n}^{-1})_{:,n})^T\left[\fPL Df(\bar{u})\fPH\right]_{H^1\to L^2}\right\|_{\ell^2}\|\fh\|_{H^1}\\
    &\leq \frac{\beta_n}{C_{\alpha}\sqrt{1-\theta^2}}\frac{1}{n^{3/4}}\left\|((P_{\leq n}^{-1})_{:,n})^T\left[\fPL Df(\bar{u})\fPH\right]_{H^1\to L^2}\right\|_{\ell^2}\|\fh\|_{H^1} \qquad\mbox{using~\eqref{eqn:Dn-estimate-2}}.
\end{align*}
On the other hand, we have
$$\iPH\tL^{-1}\iPL Df(\bar{u})\fh = \iPH\tL^{-1}\iPL (\omega \fh + h_0-3e^{-V}\bar{u}^2 \fh) = -3\iPH\tL^{-1}\iPL (e^{-V}\bar{u}^2 \fh).$$
Now, writing $\fh$ as
$$\fh = \sum_{m = 0}^n a_m q_m,$$
we get by linearity
\begin{align*}
    \|\iPL (e^{-V}\bar{u}^2 \fh)\|_{L^2}^2 &= \left\|\sum_{m = 0}^n a_m\iPL(e^{-V}\bar{u}^2 q_m)\right\|_{L^2}^2\\
    &\leq \left(\sum_{m = 0}^n |a_m| \|\iPL(e^{-V}\bar{u}^2 q_m)\|_{L^2}\right)^2\\
    &\leq \left(\sum_{m = 0}^n a_m^2\right)\left(\sum_{m=0}^n\|\iPL(e^{-V}\bar{u}^2 q_m)\|_{L^2}^2\right)\qquad\mbox{By Cauchy--Schwarz inequality}\\
    &\leq \sum_{m=0}^n\|\iPL(e^{-V}\bar{u}^2 q_m)\|_{L^2}^2\qquad\mbox{Since $\|\fh\|_{H^1} = 1$}\\
    &\leq \sum_{m=0}^n\|e^{-V}\bar{u}^2 q_m\|_{L^2}^2-\|\fPL(e^{-V}\bar{u}^2 q_m)\|_{L^2}^2\qquad\mbox{By Parseval's identity.}
\end{align*}
This implies that
\begin{align*}
    \|\iPH\tL^{-1}\iPL Df(\bar{u})\fh\|L^2 &\leq\|\iPH\tL^{-1}\iPL\|_{L^2\to H^1} \|\iPL Df(\bar{u})\fh\|L^2\\
    &\leq\frac{3C_{22}}{n^{3/4}}\left(\sum_{m=0}^n\|e^{-V}\bar{u}^2 q_m\|_{L^2}^2-\|\fPL(e^{-V}\bar{u}^2 q_m)\|_{L^2}^2\right)^{1/2}\|\fh\|_{H^1}\qquad\mbox{using~\eqref{eqn:tL-estimates}.}
\end{align*}
Thus, we can choose
\begin{align*}
Z^{21} := \frac{1}{n^{3/4}}\Bigg\{&\frac{\beta_n}{C_{\alpha}\sqrt{1-\theta^2}} \left\|((P_{\leq n}^{-1})_{:,n})^T\left[\fPL Df(\bar{u})\fPH\right]_{H^1\to L^2}\right\|_{\ell^2} \\
&\quad +3C_{22}\left(\sum_{m=0}^n\|e^{-V}\bar{u}^2 q_m\|_{L^2}^2-\|\fPL(e^{-V}\bar{u}^2 q_m)\|_{L^2}^2\right)^{1/2}\Bigg\}.
\end{align*}

\subsection{The bound $Z^{12}$}

For $\ih\in\iPH(\mathcal{H})$, we have

\begin{align*}
    \| \fPH DT(\bar{u})\ih\|_{H^1}&=\| \fPH(\ih - A DF(\bar{u})\ih)\|_{H^1}\\
    &=\| A_n\fPH DF(\bar{u})\ih\|_{H^1}\\
    &= \|A_n\fPH (\ih - \tL^{-1}Df(\bar{u})\ih)\|_{H^1}\\
    &= \|A_n\fPH\tL^{-1}Df(\bar{u})\ih\|_{H^1}\\
    &= \|A_n\fPH\tL^{-1}\fPL Df(\bar{u})\ih\|_{H^1}\\
    &\leq \|A_n\fPH\tL^{-1}\fPL( Df(\bar{u})\fPL\ih)\|_{H^1}+\|A_n\fPH\tL^{-1}\fPL( Df(\bar{u})\iPL\ih)\|_{H^1}.
\end{align*}
Note that $\ih = \iPH h$, where the projection is the $H^1$ one, therefore $\fPL \ih$ is not equal to $0$. We start with the first term, for which we have
\begin{align*}
    \left[A_n\fPH\tL^{-1}\fPL( Df(\bar{u})\fPL\ih)\right]_{H^1} 
    &=\left[A_n\right]_{H^1}\left[\fPH\tL^{-1}\fPL\right]_{L^2\to H^1}\left[\fPL Df(\bar{u})\fPL\right]_{L^2}\left[\fPL\iota\iPH\right]_{H^1\to L^2}\left[\ih\right]_{H^1}\\
    &=-\left[A_n\right]_{H^1}(P_{\leq n}^{-1})^{T}\left[\fPL Df(\bar{u})\fPL\right]_{L^2}\beta_n(P_{\leq n}^{-1})_{:,-1}(D_{>n}^{-1})_{1,:}\left[\ih\right]_{H^1}.
\end{align*}
where we used that $[\iota]_{L^2\to H^1} = P^{-1}$ and its carving given by~\eqref{eq:carvedPinv}. We thus get
\begin{align*}
    \|A_n\fPH\tL^{-1}\fPL( Df(\bar{u})\fPL\ih)\|_{H^1}&\leq \beta_n\left\|\left[A_n\right]_{H^1}(P_{\leq n}^{-1})^{T}\left[\fPL Df(\bar{u})\fPL\right]_{L^2}(P_{\leq n}^{-1})_{:,-1}\right\|_{\ell^2}\left\|(D^{-1}_{>n})_{1,:}\right\|_{\ell^2}\|\ih\|_{H^1}\\
    &=\frac{\beta_n}{C_{\alpha}\sqrt{1-\theta^2}}\frac{1}{n^{3/4}}\left\|\left[A_n\right]_{H^1}(P_{\leq n}^{-1})^{T}\left[\fPL Df(\bar{u})\fPL\right]_{L^2}(P_{\leq n}^{-1})_{:,-1}\right\|_{\ell^2}\|\ih\|_{H^1},
\end{align*}
where we have used~\eqref{eqn:Dn-estimate-2}. For the second term, we have
$$\fPL Df(\bar{u})\iPL\ih = \fPL ((\omega-3e^{-V}\bar{u}^2)\iPL\ih) =-3 \fPL (e^{-V}\bar{u}^2\iPL\ih),$$
and we now estimate each coefficient of the above term
\begin{align*}
    |\langle p_m, e^{-V}\bar{u}^2\iPL\ih\rangle| &= |\langle p_me^{-V}\bar{u}^2, \iPL\ih\rangle|\\
    &= |\langle \iPL(p_me^{-V}\bar{u}^2), \iPL\ih\rangle|\\
    &\leq \|\iPL(p_me^{-V}\bar{u}^2)\|_{L^2}\|\iPL\ih\|_{L^2}\\
    &\leq \left(\|p_me^{-V}\bar{u}^2\|_{L^2}^2-\|\fPL(p_me^{-V}\bar{u}^2)\|_{L^2}^2\right)^{1/2}\|D_{>n}^{-1}\|_{\ell^2}\|\ih\|_{H^1}\\
    &\leq \frac{C_{22}}{n^{3/4}}\left(\|p_me^{-V}\bar{u}^2\|_{L^2}^2-\|\fPL(p_me^{-V}\bar{u}^2)\|_{L^2}^2\right)^{1/2}\|\ih\|_{H^1}\qquad\mbox{using~\eqref{eqn:C12-C22-def}.}
\end{align*}
Finally, writing $w_m = [-1,1]\left(\|p_me^{-V}\bar{u}^2\|_{L^2}^2-\|\fPL(p_me^{-V}\bar{u}^2)\|_{L^2}^2\right)^{1/2}$, we obtain
\begin{align*}
    \|A_n\fPH\tL^{-1}\fPL( Df(\bar{u})\iPL\ih)\|_{H^1}&\leq\frac{3C_{22}}{n^{3/4}}\left\|\left[A_n\right]_{H^1}\left[\fPH\tL^{-1}\fPL\right]_{L^2\to H^1} w\right\|_{\ell^2}\\
    &\leq \frac{3C_{22}}{n^{3/4}}\left\|\left[A_n\right]_{H^1}(P_{\leq n}^{-1})^T w\right\|_{\ell^2}.
\end{align*}
We can therefore choose
$$Z^{12}:=\frac{1}{n^{3/4}}\left\{\frac{\beta_n}{C_{\alpha}\sqrt{1-\theta^2}}\left\|\left[A_n\right]_{H^1}(P_{\leq n}^{-1})^{T}\left[\fPL Df(\bar{u})\fPL\right]_{L^2}(P_{\leq n}^{-1})_{:,-1}\right\|_{\ell^2}+3C_{22}\left\|\left[A_n\right]_{H^1}(P_{\leq n}^{-1})^T w\right\|_{\ell^2}\right\}.$$


\subsection{The bound $Z^{22}$}

For $\ih\in \iPH(\mathcal{H})$,
\begin{align*}
    \|\iPH DT(\bar{u})\ih\|_{H^1} &=\| \ih -\iPH A DF(\bar{u})\ih\|_{H^1}\\
    &= \|\ih - \iPH DF(\bar{u})\ih\|_{H^1}\\
    &= \|\ih - \iPH (\ih - \tL^{-1}Df(\bar{u})\ih)\|_{H^1}\\
    &= \|\iPH\tL^{-1}Df(\bar{u})\ih\|_{H^1}\\
    &= \|\iPH\tL^{-1}(\omega - 3e^{-V}\bar{u}^2)\ih\|_{H^1}\\
    &\leq \|\iPH\tL^{-1}\|_{L^2\to H^1}\|(\omega-3e^{-V}\bar{u}^2)\ih\|_{L^2}\\
    &\leq \frac{C}{n^{3/4}}\|(\omega-3e^{-V}\bar{u}^2)\|_{\infty}\|\ih\|_{L^2}\qquad\mbox{by H\"older's inequality and~\eqref{eqn:tL-estimates}}\\
    &\leq \frac{C^2}{n^{3/2}}\left\{3\|e^{-V/2}\bar{u}\|_{\infty}^2 +|\omega|\right\}\|\ih\|_{H^1}\qquad\mbox{by Theorem~\ref{thm:quant-compact}.}
\end{align*}
We thus choose

$$Z^{22} := \frac{C^2}{n^{3/2}}\left\{3\|e^{-V/2}\bar{u}\|_{\infty}^2 +|\omega|\right\}.$$
While other methods are possible, in practice, we simply use the estimate~\eqref{eqn:annoying-bounds1} to bound $\|e^{-V/2}\bar{u}\|_{\infty}$.

Note that our estimates are sharp in the sense that they agree with the estimates of the diagonal case (see~\cite[Section 4.1]{Breden2025ConstructiveRd}) and that we observe in practice that $Z^{22}$ is the largest of the $Z^{ij}$'s.

\subsection{The bound $Z_2$}


We have that for $h\in \mathcal{H}$,
\begin{align*}
    \|D^2T(\bar{u})h^2\|_{H^1}&= \|AD^2F(\bar{u})h^2\|_{H^1}\\
    &\leq \|A\|_{H^1}\|D^2F(\bar{u})h^2\|_{H^1}\\
    &\leq \|A\|_{H^1}\|\tL^{-1}(D^2(\bar{u})h^2)\|_{H^1}\\
    &\leq \|A\|_{H^1}\|\tL^{-1}\|_{L^2\to H^1}\|(D^2(\bar{u})h^2)\|_{L^2}\\
    &\leq 6\|A\|_{H^1}\|(P^{-1})^T\|_{\ell^2}\|e^{-V}\bar{u}h^2\|_{L^2}\\
    &\leq 6\|A\|_{H^1}\|P^{-1}\|_{\ell^2}\|\bar{u}\|_{L^2}\|e^{-V/2}h\|_{\infty}^2\qquad\mbox{By H\"older's inequality}\\
    &\leq 6\max(1,C_P)\|A\|_{H^1}\|\bar{u}\|_{L^2}(\sqrt{\cZ(C_{\cJ}+1)}\max(1, C_P)^{1/4}\|h\|_{H^1})^2\qquad \mbox{by~\eqref{eqn:annoying-bounds1}}\\
    &\leq 6Z(C_{\cJ}+1)\max(1,C_P)^{3/2}\|A\|_{H^1}\|\bar{u}\|_{L^2}\|h\|_{H^1}^2.
\end{align*}
We can thus choose
$$Z_2 := 6Z(C_{\cJ}+1)\max(1,C_P)^{3/2}\|A\|_{H^1}\|\bar{u}\|_{L^2},$$
where $\|A\|_{H^1} = \max\{\|A_n\|_{H^1}, 1\}$.

\subsection{The bound $Z_3$}

Directly from the above bound we obtain that, for $h\in \mathcal{H}$,
\begin{align*}
    \|D^3T(\bar{u})h^3\|_{H^1}&\leq 6\|A\|_{H^1}\|(P^{-1})^T\|_{\ell^2}\|e^{-V}h^3\|_{L^2}\qquad \mbox{by~\eqref{eqn:annoying-bounds1}}\\
    &\leq 6Z(C_{\cJ}+1)\max(1,C_P)^{3/2}\|A\|_{H^1}\|h\|_{L^2}\|h\|_{H^1}^2\\
    &\leq 6Z(C_{\cJ}+1)\max(1,C_P)^{5/2}\|A\|_{H^1}\|h\|_{H^1}^3.
\end{align*}
and thus choose
$$Z_3 := 6Z(C_{\cJ}+1)\max(1,C_P)^{5/2}\|A\|_{H^1}.$$

\section{Positivity of solutions}\label{app:positivity}

We explain here how we prove the positivity of the solution $\varphi_1^{*}$ to the Gross--Pitaevskii equation
\begin{equation}\label{eqn:app-GP}
    -\partial_{xx}\varphi + W\varphi +\varphi^3 - \omega \varphi = 0, \qquad W(x) = \frac{x^6}{4} -\kappa\frac{x^4}{2}+cx^2+d,
\end{equation}
in Theorem~\ref{thm:GP1}, which can be proved to be of class $\mathcal{C}^{\infty}$. That is we prove:

\begin{lem}\label{lem:pos-GP}
    Let $\bar{\varphi}_1$ be the function whose precise description is available at~\cite{Chu2025Huggzz/Freud} in the file \texttt{GP\_eq/ubar1} and let $\varphi^{\star}_1$ be a solution of Eq.~\eqref{eqn:app-GP} such that
    $$\|e^{V/2}(\varphi^{\star}_1 - \bar{\varphi}_1)\|_{H^1(\nu)}\leq \delta = 2\times 10^{-102}.$$
    Then $\varphi^{\star}_1$ is strictly positive on $\R$.
\end{lem}

To this end, we use the criterion proposed in~\cite[§4.3]{Breden2025ConstructiveRd} reformulated in the lemma below, which is a simple application of the maximum principle.

\begin{lem}\label{lem:max-princ}
    Let $c:\R\to\R$ be a $\mathcal{C}^0$ function and $\varphi:\R^d\to\R$ a $\mathcal{C}^2$ function such that
    \begin{align}\label{eqn:phi-eq}
        -\partial_{xx}\varphi + c\varphi = 0.
    \end{align}
    Assume there exists $r_0>0$ such that $c(x)>0$ for all $\vert x\vert \geq r_0$. Assume further that $\varphi(x)>0$ for all $\vert x\vert \leq r_0$, and that $\varphi(x)\underset{\vert x\vert \to \infty}{\longrightarrow}0$. Then, $\varphi(x)>0$ for all $x\in\R$.
\end{lem}

\begin{proof}[Proof of Lemma~\ref{lem:pos-GP}]
    By a standard bootstrap argument, $\varphi^{\star}_1$ is of class $\mathcal{C}^{\infty}$. Furthermore, by Lemma~\ref{lem:annoying-bounds}
    \begin{equation}\label{eqn:infty-bound}
        \|\varphi^{\star}_1 - \bar{\varphi}_1\|_{\infty} \leq \sqrt{\cZ(C_{\cJ}+1)}\max(1, C_P)^{1/4}\delta,
    \end{equation}
    and
    $$\|\varphi^{\star}_1\|_{\infty} \leq \sqrt{\cZ(C_{\cJ}+1)}\max(1, C_P)^{1/4}\|e^{V/2}\varphi^{\star}_1\|_{H^1(\nu)}\leq\sqrt{\cZ(C_{\cJ}+1)}\max(1, C_P)^{1/4}(\|e^{V/2}\bar{\varphi}_1\|_{H^1(\nu)}+\delta).$$
     In the file \texttt{GP\_eq/proof1.ipynb} of~\cite{Chu2025Huggzz/Freud}, we then find a positive real $r_0$ such that $c(x) = W(x) +\varphi^2(x) - \omega\geq W(x) - \omega>0$ for all $|x|\geq r_0$. Then using the estimate~\eqref{eqn:infty-bound} and rigorously evaluating $\bar{\varphi}_1$ on $[-r_0, r_0]$ finishes verifying the assumptions of Lemma~\ref{lem:max-princ}. 
\end{proof}

\section{Quadrature}\label{app:quad}

In order to evaluate the bounds given in Section~\ref{subsec:bounds}, we need to compute rigorously many integrals of the form

$$\frac{1}{\tilde{\cZ}}\int_{\R} \left(\prod_{k=1}^m f_k\right)e^{-v}\d x,$$
where the $f_k$'s are polynomials, $m = 4 $ or $m = 6$, $v$ is a quartic potential of the form $v = mV/2$, and $\tilde{\cZ}$ is the corresponding normalisation constant. Similarly, we denote the orthonormal polynomials with respect to $e^{-v}$ by $\{\tilde{p}_n\}_{n\in \N}$ and the coefficients of their recurrence relation~\eqref{eqn:jacobi-rec-rel_intro} by $\{\tilde{a}_n\}_{n\in\N}$. We achieve this integration by essentially following the method employed in~\cite[§3]{Breden2025ConstructiveRd} (see also~\cite{Cadiot2025ValidatedPDEs} in the case of unweighted polynomial nonlinearities). Given some $N\in\N$ such that $2N-1\geq \sum_{k=1}^m\deg{f_k}$, we denote by $(\tilde{x}_i)_{i = 1}^N$ the roots of $\tilde{p}_N$. Then, given the weights~\cite[p. 390]{Hildebrand1987IntroductionEdition}

$$\tilde{W}_i = \frac{\tilde{a}_N}{\tilde{a}_{N-1}}\frac{\displaystyle\frac{1}{\tilde{\cZ}}\int \tilde{p}_{N-1}^2e^{-v}\d x}{\tilde{p}'_N(\tilde{x}_i)p_{N_1}(\tilde{x}_i)} = \frac{\tilde{a}_N}{\tilde{a}_{N-1}}\frac{1}{\displaystyle\left(\frac{N}{\tilde{a}_N}p_{N-1}(\tilde{x}_i)+\tilde{a}_N \tilde{a}_{N-1}\tilde{a}_{N-2}p_{N-3}(\tilde{x}_i)\right)p_{N_1}(\tilde{x}_i)},\qquad i = 1, \ldots, N,$$
we have that
$$\frac{1}{\tilde{\cZ}}\int_{\R} \left(\prod_{k=1}^m f_k\right)e^{-v}\d x = \sum_{i = 1}^N \tilde{W}_i\left(\prod_{k=1}^m f_k(\tilde{x}_i)\right).$$

The $\tilde{x}_i$'s are enclosed rigorously similarly to~\cite{Breden2025ConstructiveRd} by
\begin{itemize}
    \item finding approximate roots of $\tilde{p}_N$ by numerically finding eigenvalues of the $N$-dimensional (tridiagonal) Jacobi matrix constructed from $\{\tilde{a}_{n}\}_{n=1}^{N-1}$,
    \item refining these approximate roots using Newton's method in~\texttt{BigFloat} arithmetic,
    \item and rigorously enclosing these roots via a combination of the Intermediate Value Theorem and the Fundamental Theorem of Algebra (recall that such orthogonal polynomials always have simple real roots).
\end{itemize}

As in~\cite{Breden2025ConstructiveRd}, in practice we will express a function $u$ in $\fPL(L^2(\nu))$ with respect to the basis $\{p_n\}$ (orthogonal with respect to $\d \nu = e^{-V}\d x/\cZ$ now!), i.e.
$$u = \sum_{j=0}^n c_j p_j, \qquad \mathbf{c} := (c_0, \ldots, c_n)^T,$$
such that it suffices to evaluate $p_j(\tilde{x}_i)$ via the recurrence relation
$$p_{j+1}(\tilde{x}_i) = \frac{\tilde{x}_i p_j(\tilde{x}_i)-a_jp_{j-1}(\tilde{x}_i)}{a_{j+1}}\qquad i=1,\ldots, N.$$
This defines a pseudo-Vandermonde matrix $M = (p_j(\tilde{x}_i))$. As in~\cite{Breden2025ConstructiveRd}, we regularise $M$ by replacing it by $\bar{M} = \mathrm{Diag}(W)^{1/m}M$ which can be stored in lower precision. For instance, we then have that the finite-dimensional projection of the a cubic nonlinearity can be handled by the matrix
$$G_{ij} = \langle p_i, e^{-V}u^{2}p_j\rangle,$$
which can be written as $G = \bar{M}^{T}\mathrm{Diag}(\bar{M}\mathbf{c})^{2}\bar{M}$.

\section{About the computer-assisted proof of stochastic resonance}
\label{sec:appendix_resonance}

We present here in detail some of the technical steps required for the proof of Theorem~\ref{thm:rho-intro} (and thus also of Theorem~\ref{thm:stoc-res} and Theorem~\ref{thm:stoc-res-pic}). In Section~\ref{sec:appendix:LJU}, we study the operators $L_m=\left[i\tilde{\omega} m +\mathcal{L}\right]_{\cP_{[m]}}$, and introduce a suitable Cholesky factorisation $L_m = U^T U$. Quantitative compactness estimates for $U^{-1}$ are then derived in Section~\ref{app:eltwise-est}. In Section~\ref{app:A-construction}, we explain how we construct the approximate inverse $A^{11}$ in a way which reduces the computational cost of the proof. All these ingredients are combined in Section~\ref{sec:proofZijkl}, where we provide a proof of Proposition~\ref{prop:Zijkl}, which is at the heart of the proof of Theorem~\ref{thm:rho-intro}. In Section~\ref{sec:regularity}, we give explicit estimates that enable us to relate error bounds on the solution of~\eqref{eq:Fresonance} to error bounds on the stationary periodic density $\varrho$, in suitable weighted norms, and to error bounds on the stochastic resonance indicator $\cR(\varrho)$. Finally, Section~\ref{sec:opnormblock} contains a basic lemma regarding block-wise computation of operator norms.

\subsection{More details about the structure of some important operators}
\label{sec:appendix:LJU}

We start by precisely describing several operators, introduced in Section~\ref{sec:outlineproof}, that are critical for the proof of Theorem~\ref{thm:rho-intro}.

The operators $L_m=\left[i\tilde{\omega} m +\mathcal{L}\right]_{\cP_{[m]}}$ can be written as
\begin{align}\label{eqn:Lm-even}
    L_m = 
i\tilde{\omega} m\,\id + \begin{pmatrix}
        \lambda_{2} & \mu_{2} & & &\\
        \mu_{2} & \lambda_{4} & \mu_{4} & &\\
        & \mu_{4} &\lambda_{6} &\ddots &\\
        & & \ddots &\ddots & &\\
        & & & &
    \end{pmatrix} \qquad \text{if }m\text{ is even,}
\end{align}
and
\begin{align}\label{eqn:Lm-odd}
   L_m= i\tilde{\omega} m\, \id + \begin{pmatrix}
        \lambda_{1} & \mu_{1} & & &\\
        \mu_{1} & \lambda_{3} & \mu_{3} & &\\
        & \mu_{3} &\lambda_{5} &\ddots &\\
        & & \ddots &\ddots & &\\
        & & & &
    \end{pmatrix} \qquad \text{if }m\text{ is odd,}
\end{align}
where 
\begin{equation}
 \label{eq:deflambdaandmu}
    \lambda_k :=   \frac{k^2}{a_k^2}+a_k^2a_{k-1}^2a_{k-2}^2 = \alpha_k^2 +\beta_{k-2}^2 \qquad\text{and}\qquad\mu_k :=   ka_{k+1}a_{k+2} = \alpha_k\beta_k,
\end{equation}
with $a_k$ the coefficients of the recurrence relation defining the orthogonal polynomials $p_n$, see Section~\ref{sec:Freuddef}, and $\alpha_k$ and $\beta_k$ given in~\eqref{eqn:alpha-beta-def}.
The above formulas for $i\tilde{\omega} m + \cL$ (restricted to even and odd subspaces) follow directly from Proposition~\ref{prop:L-decomp} and from the formula for $D$ given in~\eqref{eq:DandP}.

We can also give explicit expressions for the operators $J_{[m]} = [\cJ]_{\cP_{[m]}\to\cP_{[m+1]}}$. Indeed, recalling that $\cJ$ is equal to minus the adjoint of the derivative operator (see Remark~\ref{rmk:regularity}), and using once more the formula~\eqref{eq:DandP} for the matrix $D = [\partial_x]_{\cP\to\cP}$, we get
\begin{equation}\label{eqn:J-def}
J_1 =- \begin{pmatrix}
    \alpha_{2} & & & &\\
     \beta_{2}& \alpha_{4} &  & &\\
    & \beta_{4}&\ddots &\\
    & & \ddots&& &\\
    & & & &
\end{pmatrix} \qquad \text{and}\qquad J_2  =- \begin{pmatrix}
    0 & & & & \\
    \alpha_{3} & 0 & & &\\
     \beta_{3}& \alpha_{5} & \ddots & &\\
    & \beta_{5}&\ddots &\\
    & & \ddots&& &\\
    & & & &
\end{pmatrix}.\end{equation}

We now turn our attention to the factorisation $L_m = U_m^TU_m$ of each $L_m$, which is instrumental in our proof. In order to lighten the notation, we now fix $m\in\Z$ and remove the $m$-dependence from the notations. That is, we now write $U$ instead of $U_m$, and no longer highlight the dependence with respect to $m$ in the quantities introduced below. 

The operator $U$ will be constructed from the sequence $(d_k)_{k\in\N^*}$ defined by the following recurrence relation
\begin{equation}\label{eqn:defd_k}
\begin{cases}
    d_k  = \sqrt{i\tilde{\omega} m + \lambda_k} &\qquad\mbox{if $k = 1,2$},\\
    d_k  =  \sqrt{i\tilde{\omega} m+\lambda_k -\mu_{k-2}^2/d_{k-2}^2} &\qquad\mbox{if $k >2$}.
\end{cases}
\end{equation}
The choice of square root will not matter, and without loss of generality we can choose the principal one. 
\begin{rmk}
In practice, we do check that each $d_k$ is nonzero, and hence that this sequence is well defined. We do so by rigorously computing finitely many of the $d_k$'s, and then use the upcoming estimate~\eqref{eq:lowerbounddk} for $k$ large enough. 
\end{rmk}
When consider the (unbounded) operator $U:\ell^2\to\ell^2$ given by
\begin{equation}
    \label{def:Um}
    U = \begin{pmatrix}
    d_{[m]} & \mu_{[m]}/d_{[m]} & & &\\
     & d_{[m]+2} & \mu_{[m]+2}/d_{[m]+2} & &\\
    &  &\ddots &\ddots &\\
    & & & &
\end{pmatrix},
\end{equation}
for which one can easily check that $U^T U = L_m$. The proposed formula for $U$ comes from the usual formula for the LU decomposition of a tridiagonal matrix~\cite[Theorem 4.3.2]{Ciarlet1989IntroductionOptimisation}, which generalises to the infinite dimensional case (as already observed for instance in~\cite{Breden2015RigorousPart}).

In order to study $U$, and especially its inverse, it will prove useful to first carve $U$ as follows 
\begin{equation}\label{eqn:U-form}
    U = \begin{pNiceMatrix}[margin]
  \Block[borders={bottom, right}]{3-7}{U_{\leq N}} & & & & & & & & & & & &\\
  & & & & & & & & & & &\\
  & & & & & & &\mu_n/d_n  & & & &\\
  & & & & & & & \Block[borders={top, left}]{4-5}{U_{> N}}& & & &\\
  & & & & & & & & & & &\\
  & & & & & & & & & & &\\
  & & & & & & & & & & &
\end{pNiceMatrix},
\end{equation}
where $N\in\N$ is a given truncation level, $n := [m]+2N-2$ (and thus depends implicitly on $N$ and on the parity of $m$) and $U_{\leq N} :=\Pi_{\leq N}U\Pi_{\leq N}$ and $U_{> N} := \Pi_{> N}U\Pi_{>N}$. That is,
$$U_{\leq N}  = \begin{pmatrix}
    d_{[m]} &\mu_{[m]}/d_{[m]} & & & \\
    & d_{[m]+2}& \ddots& & \\
    & & \ddots& \ddots& \\
    & & & \ddots&\mu_{n-2}/d_{n-2} \\
    & & & &d_n
\end{pmatrix},\qquad U_{> N} = \begin{pmatrix}
    d_{n+2} &\mu_{n+2}/d_{n+2} & & & \\
    & d_{n+4}& \ddots& & \\
    & & \ddots& \ddots& \\
    & & & \ddots& \\
    & & & &
\end{pmatrix}.$$
Then, we formally get
\begin{equation}
\label{eq:Uinv}
U^{-1} = \begin{pNiceMatrix}[margin]
  \Block[borders={bottom, right}]{2-5}{U_{\leq N}^{-1}} & & & & &\Block[]{2-1}{-\mu_n/d_n(U_{\leq N}^{-1})_{:,-1}(U_{> N}^{-1})_{1,:}}\\
  & & & & &\\
  & & & & & \Block[borders={top, left}]{3-1}{U_{> N}^{-1}}\\
  & & & & &\\
  & & & & &
\end{pNiceMatrix},
\end{equation}
where, by backward substitution, the coefficients of $U_{>N}^{-1} = \left((U_{>N}^{-1})_{ij}\right)_{i,j\geq 1}$, are given by
\begin{equation}
\label{eq:Utailinv}
(U_{> N}^{-1})_{ij} = 
\begin{cases}
\displaystyle\frac{(-1)^{i-j}}{d_{n+2j}}\prod_{l=i}^{j-1}\frac{\mu_{n+2l}}{d_{n+2l}^2}, \qquad &\mbox{for $i\leq j$,}\\
0 &\mbox{otherwise.}
\end{cases}
\end{equation}
\begin{rmk}
    Since $U$ is triangular, $\Pi_{\leq N} U^{-1} \Pi_{\leq N} = (\Pi_{\leq N} U \Pi_{\leq N})^{-1}$, hence we can afford to use the somewhat ambiguous notation $U^{-1}_{\leq N}$, as $(U_{\leq N})^{-1} = \Pi_{\leq N} U^{-1} \Pi_{\leq N}$. The same is true for $U^{-1}_{> N}$.
\end{rmk}
The estimates contained in the upcoming Proposition~\ref{prop:boundsUinv} prove that this operator $U^{-1}$ is indeed well defined and bounded on $\ell^2$. In principle, we could have used formulas of the form~\eqref{eq:Utailinv} for getting the inverse of $U$ in its entirety, and not only of $U_{>N}$. The reason for using these different blocks is that we will then need to estimate the entries of $U^{-1}$, and that the estimates we obtain in Proposition~\ref{prop:boundsUinv} are only valid for rows of large enough index, i.e., for $U^{-1}_{>N}$. Another upshot of this block-decomposition is that it makes it immediately clear that the top right block of $U^{-1}$ can be described as a tensor product, which will prove convenient in the sequel.

\subsection{Quantitative compactness estimates for $U^{-1}$}
\label{app:eltwise-est}
We derive here some estimates on $U^{-1}$ that play a central role in the proof of Proposition~\ref{prop:Zijkl}. First, we need some explicit control on the behaviour of the coefficients $a_k$ defining the basis polynomials of $\cP$, which then allows us to control the coefficients $\alpha_k$, $\beta_k$, $\lambda_k$ and $\mu_k$ appearing in the operators $L_m$ and $J_{[m]}$.
\begin{lem}
    \label{lem:ak_resonance}
    Let $\kappa = 1$, $\sigma = 0.287129152$, $\tilde{\kappa} = \sqrt{2}\kappa/\sigma$ as in~\eqref{eqn:change-var},
$$c^-= 0.987 \qquad \text{and}\qquad c^+ = 1.10712,$$
$N_0 = 196$ and $(a_k)_{k\in \N}$ be defined by the recurrence~\eqref{eqn:intro-bread} and initial condition~\eqref{eqn:intro-bread-initial} (with $a_k = \sqrt{b_k}$). Then, for all $k\geq N_0$,
\begin{equation}
    \label{eqn:bn_bounds_resonance}
    c^-\frac{\sqrt{k}}{\sqrt{3}}\leq a_k^2\leq c^+\frac{\sqrt{k}}{\sqrt{3}},
\end{equation}
and therefore, for all $k \geq N_0 + 2$,
\begin{align*}
    \alpha_{k+1}&\leq C_{\alpha} k^{3/4},\qquad C_{\alpha} := \frac{\sqrt[4]{3}}{\sqrt{c^-}}\left(\frac{N_0+1}{N_0}\right)^{3/4},\\
    \beta_{k+1}&\leq C_{\beta}k^{3/4},\qquad C_{\beta} :=\frac{(c^+)^{3/2}[(N_0+1)(N_0+2)(N_0+3)]^{1/4}}{(3N_0)^{3/4}},\\
    \lambda_k&\geq C_\lambda k^{3/2},\qquad C_{\lambda} :=\frac{\sqrt{3}}{c^+}+\frac{\sqrt{N_0-1}\sqrt{N_0-2}}{N_0}\frac{(c^-)^3}{3\sqrt{3}},\\
    \mu_k&\leq C_{\mu}k^{3/2}, \qquad C_{\mu}:=\frac{c^+}{\sqrt{3}}\frac{[(N_0+1)(N_0+2)]^{1/4}}{\sqrt{N_0}},
\end{align*}
where $(\alpha_k)_{k\in\N}$, $(\beta_k)_{k\in\N}$, $(\lambda_{k})_{k\in\N}$ and $(\mu_k)_{k\in \N}$ are defined by~\eqref{eqn:alpha-beta-def} and~\eqref{eq:deflambdaandmu}.
\end{lem}
\begin{proof}
    The estimate~\eqref{eqn:bn_bounds_resonance} is nothing but another instance of Proposition~\ref{prop:bread-bounds}, for a different value of $\kappa$. We refer to Section~\ref{sec:bread} for the theoretical part of the proof and to the file~\texttt{stoc\_res/get\_Painleve\_bounds.jl} at~\cite{Chu2025Huggzz/Freud} for the computational part. 
    Then, the remaining estimates follow directly from~\eqref{eqn:bn_bounds_resonance} and from the definitions of $\alpha_k$, $\beta_k$, $\lambda_k$ and $\mu_k$ in terms of $a_k$ (see~\eqref{eqn:alpha-beta-def} and~\eqref{eq:deflambdaandmu}).
\end{proof}
Thanks to Lemma~\ref{lem:ak_resonance}, we can now control the behaviour of the coefficients of $U$, and in particular show:
\begin{itemize}
    \item how fast $|d_k|$ grows,
    \item that the ratio $\left|\frac{\mu_k}{d_k^2}\right|$ falls below $1$.
\end{itemize}
This is done in Lemma~\ref{lem:tridiag_rec} below, which is inspired from~\cite[Lemma 2.1]{Breden2015RigorousPart}.
Taken together, these two points will then allow us to obtain the required quantitative compactness estimates on $U_{>N}^{-1}$.
\begin{lem}
\label{lem:tridiag_rec}
    Take $\kappa$, $c^-$, $c^+$, $C_{\lambda}$, $C_{\mu}$ and $N_0$ as in Lemma~\ref{lem:ak_resonance}, and define
    \begin{align}
        \label{eq:defchiandgamma}
        \chi := C_{\mu}/C_{\lambda}\qquad\text{and}\qquad \gamma :=\frac{1}{2}+\sqrt{\frac{1}{4}-\chi^2}.
    \end{align}
    Assume that $m\in\Z$ is such that $n := [m] + 2N - 2\geq N_0$ and consider $(d_k)_{k\in \N}$ defined by~\eqref{eqn:defd_k}.
 If
\begin{equation}\label{eqn:stupid-base-case}
    \gamma \leq \left|\frac{d_{n+2}^2}{i\tilde{\omega} m + \lambda_{n+2}}\right|,
\end{equation}
we then have, for all $k\geq 1$,
\begin{equation}
\label{eq:lowerbounddk}
|d_{n+2k}|\geq C_d(n+2k)^{3/4}, \qquad C_d := \sqrt{\gamma C_{\lambda}}.
\end{equation}
Moreover, defining 
\begin{align}
    \label{eq:defconstants}
    \vartheta := \frac{C_{\mu}}{C_d^2},\qquad C_{\alpha, d} := \frac{C_{\alpha}}{C_d}\qquad \textrm{and}\qquad C_{\beta, d} := \frac{C_{\beta}}{C_d},
\end{align}
we also have, for all $k\geq l\geq 1$,
\begin{align}
\label{eq:maj_tridiag_rec}
    \left|\frac{\mu_{n+2k}}{d_{n+2k}^2}\right| \leq \vartheta, \qquad \left|\frac{\alpha_{n+2l+1}}{d_{n+2k}}\right|&\leq C_{\alpha, d} \qquad \textrm{and}\qquad\left|\frac{\beta_{n+2l+1}}{d_{n+2k}}\right|\leq C_{\beta, d} .
\end{align}
\end{lem}
\begin{proof}
According to Lemma~\ref{lem:ak_resonance}, for all $k\geq 1$,
$$\left|\frac{\mu_{n+2k}}{i\tilde{\omega} m+\lambda_{n+2k}}\right|\leq\chi \qquad\text{and} \qquad \left|\frac{\mu_{n+2k}}{i\tilde{\omega} m+\lambda_{n+2k+2}}\right|\leq\chi.$$
We then obtain the following recurrence inequality
$$\left|\frac{d_{n+2k+2}^2}{i\tilde{\omega} m +\lambda_{n+2k+2} }\right| = \left| 1- \frac{\mu_{n+2k}}{i\tilde{\omega} m+\lambda_{n+2k}}\frac{\mu_{n+2k}}{i\tilde{\omega} m+\lambda_{n+2k+2}}\frac{i\tilde{\omega} m + \lambda_{n+2k}}{d_{n+2k}^2}\right|\geq1-\chi^2\left|\frac{i\tilde{\omega} m + \lambda_{n+2k}}{d_{n+2k}^2}\right|.$$
Since $\chi \approx 0.36705 < \frac{1}{2}$, one then readily shows by induction that, for all $k\geq 1$,
\begin{equation*}
    \left|\frac{d_{n+2k}^2}{i\tilde{\omega} m +\lambda_{n+2k} }\right| \geq \gamma,
\end{equation*}
the base case being assumption~\eqref{eqn:stupid-base-case}.
Therefore, for all $k\geq 1$,
\begin{align*}
    \vert d_{n+2k}\vert \geq \sqrt{\gamma \vert \lambda_{n+2k}\vert},
\end{align*}
and the lower bound on $\lambda_k$ from Lemma~\ref{lem:ak_resonance} yields~\eqref{eq:lowerbounddk}. The remaining estimates announced in Lemma~\ref{lem:tridiag_rec} then directly follow from the lower bounds on $\alpha_k$, $\beta_k$ and $\mu_k$ from Lemma~\ref{lem:ak_resonance}.
\end{proof}
We are now finally ready to get quantitative compactness estimates on $U^{-1}$, and a quantitative bound on (the tail part of) $J_{[m]}U^{-1}$, which, as explained in Section~\ref{sec:outlineproof}, are some of the main ingredients of our proof of Theorem~\ref{thm:rho-intro}.

\begin{prop}
\label{prop:boundsUinv}
    Let $m\in\Z$, and $U$ as in~\eqref{def:Um}. Take $c^-$, $c^+$ and $N_0$ as in Lemma~\ref{lem:ak_resonance}, and assume that $n = [m]+2N -2\geq N_0$, and that~\eqref{eqn:stupid-base-case} holds. We have the following bounds
    \begin{align}
        \|U_{>N}^{-1}\|_{\ell^2}&\leq
        \frac{1}{(1-\vartheta)C_d(2N+[m])^{3/4}},\label{eqn:U-bound}\\
        \|(U_{>N}^{-1})_{1,:}(U^{-1}_{>N})^T\|_{\ell^2} &\leq \frac{1}{(1-\vartheta^2)^{3/2}C_d^2(2N+[m])^{3/2}},\label{eqn:U-vec-mat-bound}\\
        \|\Pi_{>N+[m]}J_{[m]}U_{>N}^{-1}\|_{\ell^2}&\leq\frac{C_{\alpha,d}+ C_{\beta, d}}{1-\vartheta},\label{eqn:JU-bound}
    \end{align}
where $C_\lambda$ is defined in Lemma~\ref{lem:ak_resonance}, whereas $\vartheta$, $C_d$, $C_{\alpha,d}$ and $C_{\beta,d}$ are defined in Lemma~\ref{lem:tridiag_rec}.
\end{prop}
\begin{proof}
    We start from~\eqref{eq:Utailinv}, and use~\eqref{eq:lowerbounddk} together with the first bound in~\eqref{eq:maj_tridiag_rec}, which allows to estimate the entries of $U_{>N}^{-1}$:
\begin{equation}
\label{eq:absUtailinv}
\left\vert (U_{> N}^{-1})_{ij}\right\vert \leq 
\begin{cases}
\displaystyle \dfrac{\vartheta^{j-i}}{C_d (n+2)^{3/4}}, \qquad &\mbox{for $i\leq j$,}\\
0 &\mbox{otherwise.}
\end{cases}
\end{equation}
This readily yields that
\begin{align*}
    \Vert U_{>N}^{-1}\Vert_{\ell^1} \leq \dfrac{1}{(1-\vartheta)C_d (n+2)^{3/4}} \qquad \text{and} \qquad \Vert U_{>N}^{-1}\Vert_{\ell^\infty} \leq \dfrac{1}{(1-\vartheta)C_d (n+2)^{3/4}},
\end{align*}
which gives~\eqref{eqn:U-bound}.
Using again~\eqref{eq:absUtailinv}, we also get
\begin{align*}
   \|(U_{>N}^{-1})_{1,:}(U^{-1}_{>N})^T\|_{\ell^2} &\leq \left(\sum_{i=1}^\infty \left(\sum_{j=i}^\infty \frac{\vartheta^{j-i}}{C_d(n+2)^{3/4}}\frac{\vartheta^{j-1}}{C_d(n+2)^{3/4}} \right)^2 \right)^{1/2} \\
   &= \frac{1}{C_d^2(n+2)^{3/2}} \left(\sum_{i=1}^\infty \frac{\vartheta^{2(i-1)}}{(1-\vartheta^2)^2}\right)^{1/2} \\
   &= \frac{1}{(1-\vartheta^2)^{3/2}C_d^2(n+2)^{3/2}} ,
\end{align*}
which gives~\eqref{eqn:U-vec-mat-bound}. Finally, similarly as in Appendix~\ref{app:annoying-bounds}, we use the bidiagonal structure of $J_{[m]}$ to estimate
\begin{equation*}
    \|\Pi_{>N+[m]}J_{[m]}\Pi_{>N}U_{>N}^{-1}\|_{\ell^2} \leq \|\Diag(\alpha_{2N+3+[m]:2:})U^{-1}_{>N+1}\|_{\ell^2}+\|\Diag(\beta_{2N+1+[m]:2:})U^{-1}_{>N}\|_{\ell^2}.
\end{equation*}
Going back to~\eqref{eq:Utailinv}, and this time using the first two estimates in~\eqref{eq:maj_tridiag_rec}, we get that
\begin{align*}
    \left\vert \left(\Diag(\alpha_{2N+3+[m]:2:})U^{-1}_{>N+1}\right)_{i,j}\right\vert \leq
    \begin{cases}
        C_{\alpha,d}\vartheta^{j-i}, \qquad &\mbox{for $i\leq j$,}\\
        0 &\mbox{otherwise,}
    \end{cases}
\end{align*}
and thus
\begin{align*}
    \|\Diag(\alpha_{2N+3+[m]:2:})U^{-1}_{>N+1}\|_{\ell^2} &\leq \left(\|\Diag(\alpha_{2N+3+[m]:2:})U^{-1}_{>N+1}\|_{\ell^1} \|\Diag(\alpha_{2N+3+[m]:2:})U^{-1}_{>N+1}\|_{\ell^\infty}\right)^{1/2} \\ &\leq \frac{C_{\alpha,d}}{1-\vartheta}.
\end{align*}
Similarly, 
\begin{equation*}
    \|\Diag(\beta_{2N+1+[m]:2:})U^{-1}_{>N}\|_{\ell^2} \leq \frac{C_{\beta,d}}{1-\vartheta},
\end{equation*}
which yields~\eqref{eqn:JU-bound}.
\end{proof}

\subsection{Careful construction of the approximate inverse $A^{11}$}
\label{app:A-construction}

We give here some details regarding the construction of the approximate inverse $A$, and more precisely about the finite dimensional part $A^{11}$. As explained in Section~\ref{sec:outlineproof}, $A^{11}$ should be taken as an approximate inverse of $\bar{F}^{11}$ given in~\eqref{eqn:F11}. Indeed, when deriving the $Z$ estimate we then have to compute (rigorously, using interval arithmetic) $\id - A^{11}\bar{F}^{11}$, and would like this term to be very small. However, in practice the matrices $A^{11}$ and $\bar{F}^{11}$ are rather large (of size $(2M+1)(N+1)\times(2M+1)(N+1)$), hence we would like this computation to be done as efficiently as possible. To that end, we will introduce a reformulation of the problem that makes $\bar{F}^{11}$ sparser. We note that $\bar{F}^{11}$ is already somewhat sparse thanks to its block-tridiagonal structure, but each of the blocks $B^{11}_m = \frac{\tilde{\eta}}{2}\Pi_{\leq N}J_{[m]}L_m^{-1} \Pi_{\leq N}$ is still full, and this is what we are going to improve upon.

Since $J_{[m]}$ is lower triangular, $\Pi_{\leq N}J_{[m]}L_m^{-1} \Pi_{\leq N} = \Pi_{\leq N}J_{[m]} \Pi_{\leq N} L_m^{-1} \Pi_{\leq N}$. Defining $\tilde{L}_m^{11} = (\Pi_{\leq N}L_m^{-1}\Pi_{\leq N})^{-1}$ and $\tilde{L}^{11} = \Diag(\tilde{L}^{11}_{-M},\ldots, \tilde{L}^{11}_0, \ldots,\tilde{L}^{11}_M)$, we consider
\begin{equation*}
\tilde{F}^{11} := \bar{F}^{11}\tilde{L}^{11} = \begin{pNiceMatrix}
    \tilde{L}^{11}_{-M}  &\Block[borders={left, bottom,top,right}]{1-1}{\tfrac{\tilde{\eta}}{2}J^{11}_{[1-M]}} & & & & & & &\\
    \Block[borders={bottom,top,right}]{1-1}{\tfrac{\tilde{\eta}}{2}J^{11}_{[-M]}}&\ddots &\ddots & & & & & &\\
    &\ddots &\Block[borders={left,top}]{1-1}{\tilde{L}^{11}_{-2}} &\Block[borders={bottom,left,top,right}]{1-1}{\tfrac{\tilde{\eta}}{2}J^{11}_{1}}& & & & &\\
    & & \Block[borders={bottom,left,top,right}]{1-1}{\tfrac{\tilde{\eta}}{2}J^{11}_{2}}&\tilde{L}^{11}_{-1} &  \Block[borders={bottom,left,top,right}]{1-1}{\tfrac{\tilde{\eta}}{2}J^{11}_{2}}& & & &\\
    & & & \Block[borders={bottom,left,top,right}]{1-1}{\tfrac{\tilde{\eta}}{2}J^{11}_{1}}&\tilde{L}^{11}_{0}  &\Block[borders={bottom,left,top,right}]{1-1}{\tfrac{\tilde{\eta}}{2}J^{11}_{1}} & & &\\
    & & & & \Block[borders={bottom,left,top,right}]{1-1}{\tfrac{\tilde{\eta}}{2}J^{11}_{2}}&\tilde{L}^{11}_{1}  &\Block[borders={bottom,left,top,right}]{1-1}{\tfrac{\tilde{\eta}}{2}J^{11}_{2}} & &\\
    & & & & &\Block[borders={bottom,left,top,right}]{1-1}{\tfrac{\tilde{\eta}}{2}J^{11}_{1}} &\Block[borders={bottom,left,top,right}]{1-1}{\tilde{L}^{11}_{2}}  &\ddots &\\
    & & & & & &\ddots &\ddots &\Block[borders={bottom,left,top}]{1-1}{\tfrac{\tilde{\eta}}{2}J^{11}_{[M]}}\\
    & & & & & & & \Block[borders={bottom,left,top,right}]{1-1}{\tfrac{\tilde{\eta}}{2}J^{11}_{[M-1]}} & \tilde{L}^{11}_{M} 
\end{pNiceMatrix}.
\end{equation*}
We now show that each block $\tilde{L}^{11}_m$ is at most tridiagonal, hence confirming that $\tilde{F}^{11}$ is very sparse. Recalling that $L_m = U^T\, U$, and using the formula~\eqref{eq:Uinv} for $U^{-1}$, we get that
\begin{align*}
    \Pi_{\leq N} L_m^{-1}\Pi_{\leq N} &=  U_{\leq N}^{-1}(U_{\leq N}^{-1})^T+(\mu_n/d_n)^2 \left( (U_{\leq N}^{-1})_{:,-1}(U_{> N}^{-1})_{1,:}\right) \left( \left(U_{\leq N}^{-1}\right)_{:,-1}\left(U_{> N}^{-1}\right)_{1,:}\right)^T\\
    &=U_{\leq N}^{-1}(U_{\leq N}^{-1})^T+\xi_N\left((U_{\leq N}^{-1})_{:,-1}\,((U_{\leq N}^{-1})_{:,-1})^{T}\right),
\end{align*}
where $n := [m]+2N-2$, and the complex number $\xi_N$ is given by
\begin{equation}\label{eqn:xi_N-def}
\xi_N = \left(\frac{\mu_n}{d_n}\right)^2\,(U^{-1}_{>N})_{1,:}\, ((U^{-1}_{>N})_{1,:})^{T}= \left(\frac{\mu_n}{d_n}\right)^2\sum_{k=1}^{\infty}\frac{1}{d_{n+2k}^2}\prod_{l=1}^{k-1}\frac{\mu_{n+2l}^2}{d_{n+2l}^4}.
\end{equation}
Therefore, by the Sherman--Morrison formula,
\begin{equation}
\label{eq:tildeLm}
\tilde{L}_m^{11} = (\Pi_{\leq N}L_m^{-1}\Pi_{\leq N})^{-1} = \Pi_{\leq N}L_m\Pi_{\leq N} -
\frac{\xi_N}{1+\xi_N}d_n^2 
\begin{pmatrix}
    0 & \cdots & \cdots & 0 \\
    \vdots & \ddots & & \vdots \\
    \vdots & & 0 & 0 \\
    0 & \cdots & 0 & 1
\end{pmatrix},
\end{equation}
which is indeed tridiagonal.

We then consider $\tilde{A}^{11} = (\tilde{A}^{11}_{ij})_{-M\leq i,j,\leq M} $ a numerically computed inverse of $\tilde{F}^{11}$, and define $A^{11} := \tilde{L}^{11}\tilde{A}^{11}$. Thanks to this deliberate choice of $A^{11}$, we get
\begin{equation*}
\id - A^{11}\bar{F}^{11} = \id - \tilde{L}^{11}(\tilde{A}^{11}\tilde{F}^{11})(\tilde{L}^{11})^{-1} = \tilde{L}^{11}R^{11}(\tilde{L}^{11})^{-1},
\end{equation*}
where $R^{11} := \id - \tilde{A}^{11}\tilde{F}^{11} = (R_{ij}^{11})_{-M\leq i,j\leq M}$. Because $\tilde{F}_{11}$ is sparser than $\bar{F}^{11}$, the computation of $R^{11}$ is cheaper than that of $\id - A^{11}\bar{F}^{11}$. Moreover, the entries of $R^{11}$ should have close-to-zero entries by construction. Therefore, when we need to estimate the norm of 
\begin{align*}
    (\id - A^{11}\bar{F}^{11})_{ij}  =  \tilde{L}_{i}^{11} R_{ij}^{11} (\tilde{L}_{j}^{11})^{-1},
\end{align*} 
we can get away with using the crude estimate
\begin{align*}
    \Vert \tilde{L}_{i}^{11} R_{ij}^{11} (\tilde{L}_{j}^{11})^{-1} \Vert_{\ell^2} \leq\Vert \tilde{L}_{i}^{11} \Vert_{\ell^2} \Vert  R_{ij}^{11} \Vert_{\ell^2} \Vert (\tilde{L}_{j}^{11})^{-1} \Vert_{\ell^2}.
\end{align*}
Such calculations occur in the estimates needed for Proposition~\ref{prop:Zijkl} (see for instance Appendix~\ref{sec:proofZij11}).

\begin{rmk}
    Even though $\tilde{F}^{11}$ itself is very sparse, its (numerical) inverse $\tilde{A}^{11}$ is not. In order to alleviate memory constraints, we therefore compute $\tilde{A}^{11}$ block-column-wise. That is, we compute $(\tilde{A}_{ij})_{-M\leq i\leq M}$ one $j$ at a time. Since the $(Z_{ij})_{-M\leq i\leq M}$ depends only on $(\tilde{A}_{i,j-1})_{-M\leq i\leq M}$, $(\tilde{A}_{ij})_{-M\leq i\leq M}$ and $(\tilde{A}_{i,j+1})_{-M\leq i\leq M}$, we then only need to keep in memory three block columns at a time. This allows us to use relatively little memory even if we take $M$ to be large. The banded-block-tridiagonal structure of $\tilde{F}^{11}$ allows us to make this procedure efficient, as we compute its block-tridiagonal $LU$ decomposition once, and then use it $2M+1$ times to solve for $(\tilde{A}_{ij})_{-M\leq i\leq M}$.
\end{rmk}
\begin{rmk}
We can furthermore roughly halve the cost of the proof by exploiting the ``real-preserving'' structure of the equation and by choosing $A=(A_{ij})_{i,j, \in \Z}$ such that $A_{-i, -j}  = A_{ij}^*$ (where $\cdot^*$ denotes the conjugate without transposition). Indeed, we then have that the $-j$\textsuperscript{th} block-column of $T = (T_{ij})_{i,j\in \Z}$ has $\ell^1(\ell^2)$-norm
\begin{align*}
    \|(T_{i,-j})_{i\in \Z}\|_{\ell^1(\ell^2)} &= \sum_{i \in \Z}\|T_{i,-j}\|_{\ell^2}\\
    &=\sum_{i \in \Z}\|T_{-i,-j}\|_{\ell^2}\\
    &=\sum_{i \in \Z}\|\delta_{-i, -j}\id - (A_{-i,-j-1}B_{-j} + A_{-i, -j} + A_{-i, -j+1}B_{-j})\|_{\ell^2}\\
    &=\sum_{i \in \Z}\|\delta_{i, j}\id - (A_{i,j+1}^*B_{j}^* + A_{i, j}^* + A_{i, j-1}^*B_{j}^*)\|_{\ell^2}\qquad\mbox{as }B_{-j} = B_j^*\\
    &=\sum_{i\in \Z}\|T_{ij}^*\|_{\ell^2}\\
    &=\sum_{i\in \Z}\|T_{ij}\|_{\ell^2}.
\end{align*}
\end{rmk}

\begin{rmk}
\label{rem:xi}
The real number $\xi_N$ involved in the formula~\eqref{eq:tildeLm} for $\tilde{L}_m^{11}$ is defined as an infinite sum~\eqref{eqn:xi_N-def}. However, since this sums converges geometrically, we can easily get a tight and rigorous enclosure for $\xi_N$. We simply pick some $K\in \N$ large enough, and consider
$$\bar{\xi}_N := \left(\frac{\mu_n}{d_n}\right)^2\sum_{k=1}^{K}\frac{1}{d_{n+2k}^2}\prod_{l=1}^{k-1}\frac{\mu_{n+2l}^2}{d_{n+2l}^4}.$$
Using~\eqref{eq:maj_tridiag_rec} we then have
\begin{align*}
    |\xi_N - \bar{\xi}_N| &\leq \left(\frac{\mu_n}{\vert d_n\vert }\right)^2\sum_{k=K+1}^{\infty}\frac{1}{|d_{n+2k}|^2}\prod_{l=1}^{k-1}\frac{\mu_{n+2l}^2}{|d_{n+2l}|^4}\\
    &\leq \left(\frac{\mu_n}{\vert d_n\vert }\right)^2\sum_{k=K+1}^{\infty}\frac{\vartheta^{2(k-1)}}{C_d^2(n+2k)^{3/2}}\\
    &\leq \left(\frac{\mu_n}{\vert d_n\vert }\right)^2\frac{\vartheta^{2K}}{(1-\vartheta^2)C_d^2(2N+2K)^{3/2}}.
\end{align*}
\end{rmk}

\subsection{Proof of Proposition~\ref{prop:Zijkl}}
\label{sec:proofZijkl}

We give here the proof of Proposition~\ref{prop:Zijkl}, which relies on the estimates from Proposition~\ref{prop:boundsUinv}. Recalling~\eqref{eqn:Tcarving} and~\eqref{eq:defZij}, we have to show that the constants $Z_{ij}^{11}$, $Z_{ij}^{12}$, $Z_{ij}^{21}$ and $Z_{ij}^{22}$ defined in Proposition~\ref{prop:Zijkl} satisfy
\begin{align*}
    \Vert T_{ij}^{11}\Vert_{\ell^2} \leq Z_{ij}^{11},\quad
    \Vert T_{ij}^{12}\Vert_{\ell^2} \leq Z_{ij}^{12},\quad 
    \Vert T_{ij}^{21}\Vert_{\ell^2} \leq Z_{ij}^{21},\quad 
    \Vert T_{ij}^{22}\Vert_{\ell^2} \leq Z_{ij}^{22}.
\end{align*}
To that end, we first split the operators $B_m$ which appear in $T$ as follows:
\begin{equation}
\label{eq:Bsplit}
B_m = \begin{pNiceMatrix}[margin]
  \Block[borders={bottom, right}]{3-4}{B^{11}_{m}} & & & &\Block[borders={}]{3-4}{B^{12}_{m}} & & & \\
  & & & & & & & \\
  & & & & & & & \\
  \Block[borders={}]{5-3}{B^{21}_{m}}& & & \Block[borders={top, left}]{5-5}{B^{22}_m}& & & &\\
  & & & & & & & \\
  & & & & & & &\\
  & & & & & & & \\
  & & & & & & & 
\end{pNiceMatrix},
\end{equation}
where the finite part $B_{m}^{11} = \Pi_{\leq N}B_{m}\Pi_{\leq N}$ was already introduced before, and forces us to take $B_{m}^{12} := \Pi_{\leq N} B_{m}\Pi_{> N}$. However,  we emphasise that we use a slightly different projection for the other two terms, namely $B_{m}^{21} := \Pi_{> N} B_{m}\Pi_{\leq N-[m]}$ and $B_{m}^{22} := \Pi_{> N} B_{m}\Pi_{> N-[m]}$. The rationale behind this slightly odd-looking splitting is that it allows us to better take advantage of the structure of $U^{-1}$ and to obtain sharper estimates (see Remark~\ref{rmk:weird-carving}).
 Then, recalling~\eqref{eq:Tij} and~\eqref{eqn:Tcarving}, and using that $A^{12}_{ij}=0$ and $A^{21}_{ij}=0$ for all $i$ and $j$, we have
\begin{equation}
\label{eq:Tklij}
\begin{cases}
    T^{11}_{ij} &= (\delta_{ij}\Pi_{\leq N}  - (A^{11}_{i,j-1}B^{11}_{j} +A^{11}_{i,j} + A^{11}_{i, j+1}B^{11}_{j}))\Pi_{\leq N-[j]}\\
    T^{12}_{ij} &=(\delta_{ij}\Pi_{\leq N}  - (A^{11}_{i,j-1}B^{11}_{j} +A^{11}_{i,j} + A^{11}_{i, j+1}B^{11}_{j}))\Pi_{> N-[j]}  - (A^{11}_{i,j-1} + A^{11}_{i, j+1})B^{12}_{j}\\
    T^{21}_{ij} &= - (\delta_{i,j-1}  + \delta_{i, j+1})B^{21}_{j}\\
    T^{22}_{ij} &= - (\delta_{i,j-1} + \delta_{i, j+1})B^{22}_{j}.
\end{cases}
\end{equation}
We derive the corresponding estimates $Z^{11}_{ij}$, $Z^{12}_{ij}$, $Z^{21}_{ij}$ and $Z^{22}_{ij}$ in separate subsections below.

\subsubsection{The bound $Z^{11}_{ij}$}
\label{sec:proofZij11}

According to~\eqref{eq:Tklij}, and to the definition of $Z_{ij}^{11}$ in Proposition~\ref{prop:Zijkl}, we simply have $Z_{ij}^{11} = \epsilon_{ij}^{\leq N-[j]}= \Vert T_{ij}^{11}\Vert_{\ell^2} $, and there is nothing left to prove.

\begin{rmk}
\label{rem:Rij}
In practice, the only downside of that estimate is that $\epsilon_{ij}^{\leq N-[j]}$ involves the rigorous computation of the $\ell^2$ operator norm of a matrix. As explained in Section~\ref{sec:quant-poinc}, rigorously enclosing the $\ell^2$ operator norm of a matrix is possible, but somewhat costly. In practice, we therefore sometimes use a coarser but cheaper-to-compute estimate. Indeed, when $|i|,|j|\leq M$, the construction of $A^{11}$ detailed in Section~\ref{app:A-construction} yields
$$T^{11}_{ij} = (\delta_{ij}\Pi_{\leq N}  - (A^{11}_{i,j-1}B^{11}_{j-1} +A^{11}_{i,j} + A^{11}_{i, j+1}B^{11}_{j+1}) )\Pi_{\leq N-[j]} = \tilde{L}_i^{11}R_{ij}^{11}(\tilde{L}_j^{11})^{-1}\Pi_{\leq N-[j]},$$
where we expect $R_{ij}^{11}$ to be very close to $0$. Hence, we make use of the bound
\begin{align*}
\|T^{11}_{ij}\|_{\ell^2}&\leq\|\tilde{L}_i^{11}R_{ij}^{11}(\tilde{L}_j^{11})^{-1}\Pi_{\leq N-[j]}\|_{\ell^2}\\
    &\leq \|\tilde{L}_i^{11}\|_{\ell^2}\|R_{ij}^{11}\|_{\ell^2}\|(\tilde{L}_j^{11})^{-1}\Pi_{\leq N-[j]}\|_{\ell^2}\\
    &\leq \|\tilde{L}_i^{11}\|_{\ell^2}\|(\tilde{L}_j^{11})^{-1}\Pi_{\leq N-[j]}\|_{\ell^2}\|R_{ij}^{11}\|_{\ell^1}^{1/2}\|R_{ij}^{11}\|_{\ell^{\infty}}^{1/2}.
\end{align*}
We then consider, 
\begin{align*}
    \tilde{\epsilon}_{ij}^{\leq N-[j]} := \|\tilde{L}_i^{11}\|_{\ell^2}\|(\tilde{L}_j^{11})^{-1}\Pi_{\leq N-[j]}\|_{\ell^2}\|R_{ij}^{11}\|_{\ell^1}^{1/2}\|R_{ij}^{11}\|_{\ell^{\infty}}^{1/2},
\end{align*}
and, for $\vert i\vert,\, \vert j\vert \leq M$, use $Z_{ij}^{11} := \tilde{\epsilon}_{ij}^{\leq N-[j]}$ instead of $Z_{ij}^{11} = \epsilon_{ij}^{\leq N-[j]}$. This new estimate reduces the number of $\ell^2$ operator norm computations, and we still expect it to be tight enough, as the entries of $R_{ij}$ should be close to machine precision.
\end{rmk}



\subsubsection{The bound $Z^{12}_{ij}$}

We now focus on the formula for $T_{ij}^{12}$ in~\eqref{eq:Tklij}. The $\ell^2$ operator norm of the first term is exactly given by $\epsilon_{ij}^{> N-[j]}$ from Proposition~\ref{prop:Zijkl}, and Remark~\ref{rem:Rij} also applies here.

In order to estimate the second term, namely $(A^{11}_{i,j-1}+A^{11}_{i,j+1})B^{12}_{j}$, we use the definition $B_j = \frac{\tilde{\eta}}{2}J_{[j]}L_j^{-1}$ and the formula~\eqref{eq:Uinv} for $U^{-1}$ to get
\begin{align*}
    B^{12}_j &= \frac{\tilde{\eta}}{2}\Pi_{\leq N}J_{[j]}L_j^{-1}\Pi_{>N}\\
    &= \frac{\tilde{\eta}}{2}J^{11}_{[j]}\Pi_{\leq N}U^{-1} (U^{-1})^T\Pi_{>N}\\
    &=  \frac{\tilde{\eta}}{2}J^{11}_{[j]}\left(-\frac{\mu_n}{d_n}(U_{\leq N}^{-1})_{:,-1}(U_{> N}^{-1})_{1,:}\right)\, U^{-1}_{>N} \\
    &=-\frac{\tilde{\eta}\mu_{[j]+2N-2}}{2d_{[j]+2N-2}} \left(J^{11}_{[j]}(U_{\leq N}^{-1})_{:,-1}\right)\left((U_{> N}^{-1})_{1,:}\, U^{-1}_{>N}\right).
\end{align*}
Therefore, we get
\begin{align*}
    \|(A^{11}_{i,j-1}+A^{11}_{i,j+1})B^{12}_{j}\|_{\ell^2} 
    &\leq\frac{\tilde{\eta}\vert\mu_{[j]+2N-2}\vert}{2\vert d_{[j]+2N-2}\vert}\frac{\|(A^{11}_{i,j-1}+A^{11}_{i,j+1}) J_{[j]}^{11}(U_{\leq N}^{-1})_{:,-1}\|_{\ell^2}}{(1-\vartheta^2)^{3/2}C_d^2(2N+[j])^{3/2}}\qquad \mbox{using~\eqref{eqn:U-vec-mat-bound}}\\
    &\leq\frac{\tilde{\eta}\vert\mu_{[j]+2N-2}\vert}{2\vert d_{[j]+2N-2}\vert}\frac{\|(A^{11}_{i,j-1}+A^{11}_{i,j+1}) J_{[j]}^{11}(U_{\leq N}^{-1})_{:,-1}\|_{\ell^2}}{C_d^2\left(2N(1-\vartheta^2)\right)^{3/2}}\\
    &= \frac{\tilde{\eta}}{2}|U_{N, N+1}|\frac{\|(A^{11}_{i,j-1}+A^{11}_{i,j+1}) J_{[j]}^{11}(U^{-1})_{:N,N}\|_{\ell^2}}{C_d^2\left(2N(1-\vartheta^2)\right)^{3/2}}
\end{align*}
since $\dfrac{\mu_{[j]+2N-2}}{d_{[j]+2N-2}}=\dfrac{\mu_n}{d_n}=U_{N, N+1}$ (from~\eqref{eqn:U-form}) and therefore
$$\|T^{12}_{ij}\|\leq \epsilon^{> N-[j]}_{ij} +\frac{\tilde{\eta}}{2}|U_{N, N+1}|\frac{\|(A^{11}_{i,j-1}+A^{11}_{i,j+1}) J_{[j]}^{11}U^{-1}_{:N,N}\|_{\ell^2}}{C_d^2\left(2N(1-\vartheta^2)\right)^{3/2}} = Z_{ij}^{12}.$$

\subsubsection{The bound $Z^{21}_{ij}$}

Going back to~\eqref{eq:Tklij}, we now estimate $\|T^{21}_{ij}\|_{\ell^2} = (\delta_{i,j+1}+\delta_{i,j-1})\|B^{21}_{j}\|_{\ell^2}$.
We have that 
\begin{align}
\label{eq:B21}
    B^{21}_{j} &= \Pi_{>N}B_{j}\Pi_{\leq N-[j]}\nonumber\\
    &=\frac{\tilde{\eta}}{2}\Pi_{>N}J_{[j]}L^{-1}_j\Pi_{\leq N-[j]}\nonumber\\
    &=\frac{\tilde{\eta}}{2}\Pi_{>N}J_{[j]}\Pi_{>N-[j]}L^{-1}_j\Pi_{\leq N-[j]}\qquad \mbox{given the structure of $J_{[j]}$ defined in~\eqref{eqn:J-def}}\nonumber\\
    &=\frac{\tilde{\eta}}{2}\Pi_{>N}J_{[j]}\Pi_{>N-[j]}U^{-1}(U^{-1})^T\Pi_{\leq N-[j]}\nonumber\\
    &=\frac{\tilde{\eta}}{2}\Pi_{>N}J_{[j]}\Pi_{>N-[j]}U^{-1}\Pi_{> N-[j]}(U^{-1})^T\Pi_{\leq N-[j]}\nonumber\\
    &=\frac{\tilde{\eta}}{2}\Pi_{>N}J_{[j]}U^{-1}_{>N-[j]}\Pi_{> N-[j]}(U^{-1})^T\Pi_{\leq N-[j]}\nonumber\\
    &=\frac{\tilde{\eta}}{2}\Pi_{>N}J_{[j]}U^{-1}_{>N-[j]}\left(-\frac{\mu_{2N-[j]-2}}{d_{2N-[j]-2}}((U_{> N-[j]}^{-1})_{1,:})^T((U_{\leq N - [j]}^{-1})_{:,-1})^T\right),
\end{align}
where we crucially used that the top-right block of $U^{-1}$ can be written as an outer product (see~\eqref{eq:Uinv}). Indeed, this will allow for a sharp estimate below.
\begin{rmk}\label{rmk:weird-carving}
    This reformulation is the main reason why we use a different projection on the right and on the left in~\eqref{eqn:Tcarving} when splitting $T_{ij}$ into four parts, and why we split $B_j$ as in~\eqref{eq:Bsplit}. Indeed, otherwise we would have had to control $\Pi_{>N} B_j \Pi_{\leq N}$, and hence $\Pi_{> N-[j]}(U^{-1})^T\Pi_{\leq N} = \left( \Pi_{\leq N}(U^{-1})\Pi_{> N-[j]}\right)^T$, which cannot be written as a single outer-product.
\end{rmk}
Applying~\eqref{eqn:JU-bound} to~\eqref{eq:B21}, we get
\begin{align*}
    \|B^{21}_{j}\|_{\ell^2}&\leq \frac{\tilde{\eta}}{2}\|\Pi_{>N}J_{[j]}U^{-1}_{>N-[j]}\|_{\ell^2} \left\Vert\frac{\mu_{2N-[j]-2}}{d_{2N-[j]-2}}(U_{> N-[j]}^{-1})_{1,:}\right\Vert_{\ell^2} \Vert(U_{\leq N - [j]}^{-1})_{:,-1}\Vert_{\ell^2}\\
    &\leq\frac{\tilde{\eta}}{2}\sqrt{|\xi_{N-[j]}|}\|\Pi_{>N}J_{[j]}U^{-1}_{>N-[j]}\|_{\ell^2}\|(U^{-1}_{\leq N-[j]})_{:,-1}\|_{\ell^2}\\
    &=\frac{\tilde{\eta}}{2}\sqrt{|\xi_{N-[j]}|}\|\Pi_{>N}J_{[j]}U^{-1}_{>N-[j]}\|_{\ell^2}\|(U^{-1})_{:N-[j],N-[j]}\|_{\ell^2},
\end{align*}
where $\xi_{N-[j]}$ is as in~\eqref{eqn:xi_N-def}, and thus we indeed get
$$
\Vert T^{21}_{ij}\Vert_{\ell^2} \leq (\delta_{i,j+1}+\delta_{i,j-1})\frac{\tilde{\eta}(C_{\alpha, d}+C_{\beta,d})}{2(1-\vartheta)}\sqrt{|\xi_{N-[j]}|}\|(U^{-1}_{j})_{:N-[j],N-[j]}\|_{\ell^2} = Z^{21}_{ij}.$$
Combining~\eqref{eqn:xi_N-def} and the fact that $\dfrac{\mu_{2N-[j]-2}}{d_{2N-[j]-2}}=U_{N-[j], N-[j]+1}$ yields the estimate in Proposition~\ref{prop:Zijkl}.


\subsubsection{The bound $Z^{22}_{ij}$}

For the last term in~\eqref{eq:Tklij}, namely $T^{22}_{ij} = (\delta_{i,j+1}+\delta_{i,j-1})B^{22}_{j}$, the calculation is very similar to the one made just above:
\begin{align*}
    B^{22}_{j} &= \Pi_{>N}B_{j}\Pi_{> N-[j]}\\
    &=\frac{\tilde{\eta}}{2}\Pi_{>N}J_{[j]}L^{-1}_j\Pi_{> N-[j]}\\
    &=\frac{\tilde{\eta}}{2}\Pi_{>N}J_{[j]}\Pi_{>N-[j]}L^{-1}_j\Pi_{> N-[j]}\qquad \mbox{given the structure of $J_{[j]}$ defined in~\eqref{eqn:J-def}}\\
    &=\frac{\tilde{\eta}}{2}\Pi_{>N}J_{[j]}\Pi_{>N-[j]}U^{-1}(U^{-1})^T\Pi_{> N-[j]}\\
    &=\frac{\tilde{\eta}}{2}\Pi_{>N}J_{[j]}\Pi_{>N-[j]}U^{-1}\Pi_{> N-[j]}(U^{-1})^T\Pi_{> N-[j]}\\
    &=\frac{\tilde{\eta}}{2}\Pi_{>N}J_{[j]}U^{-1}_{>N-[j]}(U^{-1}_{>N-[j]})^T,
\end{align*}
so that
$$\|B^{22}_{j}\|_{\ell^2} \leq\frac{\tilde{\eta}}{2}\|\Pi_{>N}J_{[j]}U^{-1}_{>N-[j]}\|_{\ell^2}\|U^{-1}_{>N-[j]}\|_{\ell^2}.$$
Using~\eqref{eqn:JU-bound} and~\eqref{eqn:U-bound}, we therefore get
$$\Vert T_{ij}^{22}\Vert_{\ell^2} \leq (\delta_{i,j+1}+\delta_{i,j-1})\frac{\tilde{\eta}(C_{\alpha, d}+C_{\beta,d})}{2(1-\vartheta)^2C_d(2N-2)^{3/4}} = Z^{22}_{ij}.$$
This finishes the proof of Proposition~\ref{prop:Zijkl}. \hfill$\qed$

\subsection{Some embedding constants}
\label{sec:regularity}

The first part of Theorem~\ref{thm:linearNK} provides us with a computable upper bound for $\|v - \bar{v}\|_{\mathcal{X}}$. In order to obtain the second part of Theorem~\ref{thm:linearNK} (and thus Theorem~\ref{thm:rho-intro}, Theorem~\ref{thm:stoc-res} and Theorem~\ref{thm:stoc-res-pic}), we need to control
$$e^{V/2}(\rho-\bar{\rho}) = \frac{\sqrt[4]{2}}{\sqrt{\sigma}}\frac{e^{-V/2}}{\mathcal{Z}}(\tilde{\omega}\partial_s +\mathcal{L})^{-1}(v-\bar{v}),$$
and $\cR(\varrho)-\cR(\bar\varrho)$, which is done thanks to the two following lemmas.

\begin{lem}\label{lem:regularity}
    Suppose that $u\in (\tilde{\omega} \partial_s +\cL)^{-1}\mathcal{X}$, then
    $$\left\|\frac{e^{-V/2}}{\cZ}u\right\|_{\mathcal{C}^0(\T\times\R)}\leq\sqrt{\frac{C_{\cJ}+1}{\cZ}} C_P^{3/4}\|(\tilde{\omega} \partial_s +\cL)u\|_{\mathcal{X}}.$$
\end{lem}

\begin{proof} Writing $u(s, y) = \sum_{m\in \Z}e^{ims}u_m(y)$ and using Lemma~\ref{lem:annoying-bounds}, we have that
    \begin{equation*}
    \left\|\frac{e^{-V/2}}{\cZ}u\right\|_{\mathcal{C}^0(\T\times \R)}\leq \sqrt{\frac{C_{\cJ}+1}{\cZ}}C_P^{1/4}\|u\|_{\mathcal{C}^0(H^1(\nu))}\leq \sqrt{\frac{C_{\cJ}+1}{\cZ}}C_P^{1/4}\sum_{m \in \Z}\|u_m\|_{H^1(\nu)},
\end{equation*}
where we replaced $\max(C_P,1)$ by $C_P$ in Lemma~\ref{lem:annoying-bounds}, because $u$ has mean zero (against $\nu$). Since $u_m \in H^1_0(\nu)$, reusing the operators and notations introduced in Section~\ref{subsec:diff-op}, we have
\begin{align*}
    \|u_m\|_{H^1(\nu)} &=\|[u_m]_{\{q_n\}_{n\in \N^*}}\|_{\ell^2}\\
    & =\|[u_m]_{H^1}\|_{\ell^2} \\
    &=\|(P^{-1})^TP^T[u_m]_{H^1}\|_{\ell^2}\\
    &=\|(D_{>0}^{-1})^TD_{>0}^T[u_m]_{H^1}\|_{\ell^2} \qquad \mbox{as $(u, q_0) = \langle u, p_0\rangle = 0$}\\
    &\leq\|(D_{>0}^{-1})^T\|_{\ell^2}\|P^T[u_m]_{H^1}\|_{\ell^2}\\
    &= C_P^{1/2}\|[\mathcal{L}]_{H^1\to L^2}[u_m]_{H^1}\|_{\ell^2} \qquad \mbox{thanks to~\eqref{eq:CP} and~\eqref{eqn:tL-decomp}}\\
    &=C_P^{1/2}\|[\mathcal{L}u_m]_{L^2}\|_{\ell^2}\\
    &=C_P^{1/2}\|\mathcal{L}u_m\|_{L^2(\nu)},
\end{align*}
and thus
\begin{align}
\label{eq:C0toLuX}
    \left\|\frac{e^{-V/2}}{\cZ}u\right\|_{\mathcal{C}^0(\T\times \R)}\leq \sqrt{\frac{C_{\cJ}+1}{\cZ}}C_P^{3/4}\sum_{m \in \Z}\|\cL u_m\|_{L^2(\nu)} = \sqrt{\frac{C_{\cJ}+1}{\cZ}}C_P^{3/4} \Vert \cL u\Vert_{\cX}.
\end{align}
Finally, note that
\begin{equation}
    \label{eq:LVSdtpL}
    \Vert \cL u\Vert_{\cX} \leq \Vert (\tilde{\omega}\partial_s +\cL) u\Vert_{\cX}.
\end{equation}
Indeed, since $\cL$ is diagonalisable in an orthonormal basis of $L^2(\nu)$, for all $m\in \Z$ we have
$$\left\|(i\tilde{\omega} m +\mathcal{L})u_m\right\|^2_{L^2(\nu)} = \tilde{\omega}^2m^2\|u_m\|_{L^2(\nu)}^2 + \|\mathcal{L}u_m\|^2_{L^2(\nu)} \geq\|\mathcal{L}u_m\|^2_{L^2(\nu)},$$
hence
\begin{align*}
    \Vert \cL u\Vert_{\cX} = \sum_{m\in\Z} \Vert \cL u_m\Vert_{\cX} \leq \sum_{m \in \Z}\left\|(i\tilde{\omega} m +\mathcal{L})u_m\right\|^2_{L^2(\nu)} = \Vert (\tilde{\omega}\partial_s +\cL) u\Vert_{\cX}.
\end{align*}
Combining~\eqref{eq:C0toLuX} and~\eqref{eq:LVSdtpL} finishes the proof.
\end{proof}


\begin{lem}
\label{lem:Rsigma}
Let $u\in (\tilde{\omega} \partial_s +\cL)^{-1}\mathcal{X}$, let $v = (\tilde{\omega}\partial_s + \cL)u$, let $c \in\R$ and let 
\begin{equation}
\label{eq:rhochi}
    \varrho = \varrho(t, x) = (c+u(s, y))\frac{\sqrt[4]{2}}{\sqrt{\sigma}}\frac{e^{-y^4/4+\tilde{\kappa} y^2/2}}{\cZ}, \qquad \cZ = \int_\R e^{-y^4/4+\tilde{\kappa} y^2/2} \d y.
\end{equation}
Then, the quantity $\cR(\varrho)$ defined in~\eqref{eq:defR} satisfies
\begin{equation*}
    \cR(\varrho) \leq C_P\frac{\sqrt{\sigma b_1}}{\sqrt[4]{2}} \Vert v\Vert_{\cX}.
\end{equation*}
\end{lem}
\begin{proof}
    First, from change of variable~\eqref{eqn:change-var} we have
\begin{align*}
    \cR(\varrho) &= \sup_{t\in \T/\omega}\left|\int_{\R} x\varrho(t, x)\d x\right| \\
    &= \sup_{s\in \T}\frac{\sqrt{\sigma}}{\sqrt[4]{2}}\left|\int_{\R} y(c+u(s,y))\frac{e^{-V(s)}}{\cZ}\d y\right|\qquad \mbox{where } V(y) = \dfrac{y^4}{4}-\tilde{\kappa}\dfrac{y^2}{2} \\
    &= \sup_{s\in \T}\frac{\sqrt{\sigma}}{\sqrt[4]{2}}\left|\int_{\R} yu(s,y)\frac{e^{-V(y)}}{\cZ}\d y\right|\qquad \mbox{since }y \mbox{ is odd and } e^{-V(y)} \mbox{ even}\\
    &= \sup_{s\in \T}\frac{\sqrt{\sigma b_1}}{\sqrt[4]{2}}\left|\int_{\R} p_1(y)u(s,y)\frac{e^{-V(y)}}{\cZ}\d y\right|\qquad \mbox{since }\sqrt{b_1} = \dfrac{y}{p_1(y)} \\
    &=\frac{\sqrt{\sigma b_1}}{\sqrt[4]{2}}\sup_{s\in \T}|\langle p_1, u(s, \cdot)\rangle|.
\end{align*}
Thus, using the Cauchy--Schwarz inequality and the fact that $\Vert p_1\Vert_{L^2(\nu)} = 1$, we get
\begin{align*}
    \cR(\varrho) \leq \frac{\sqrt{\sigma b_1}}{\sqrt[4]{2}} \sup_{s\in \T} \Vert u(s,\cdot)\Vert_{L^2(\nu)} \leq \frac{\sqrt{\sigma b_1}}{\sqrt[4]{2}}\Vert u\Vert_{\cX}.
\end{align*}
Then, we reuse~\eqref{eq:LVSdtpL}, or more precisely the fact that
\begin{equation*}
    \Vert (\tilde{\omega}\partial_s + \cL)^{-1} \Vert_{\cX\to\cX} \leq \Vert  \cL^{-1} \Vert_{L^2_0(\nu)\to L^2_0(\nu)},
\end{equation*}
where $L^2_0(\nu) = \overline{\mathrm{Span}\{p_n\}_{n\in \N^*}}$ is the subspace of $L^2(\nu)$ of zero mean functions with respect to $\nu$ (on which $\mathcal{L}$ is invertible). Finally, since $[\mathcal{L}]_{L^2_0(\nu)} = D_{>0}^TD_{>0}$ (see Proposition~\ref{prop:L-decomp}), we have from~\eqref{eq:defCp} that $\|\mathcal{L}^{-1}\|_{L^2_0(\nu)} = \|D_{>0}^{-1}\|_{\ell^2}^2= C_P$, and thus
\begin{align*}
    \Vert u\Vert_{\cX} = \Vert (\tilde{\omega}\partial_s + \cL)^{-1} v\Vert_{\cX} \leq C_P \Vert v\Vert_{\cX},
\end{align*}
which finishes the proof.
\end{proof}
When $c = 1$, the $\varrho$ defined in~\eqref{eq:rhochi} is nothing but the usual $\varrho$ from~\eqref{eqn:resonance-u-def}. However, for the proof of Theorem~\ref{thm:linearNK}, we need to apply Lemma~\ref{lem:Rsigma} to 
\begin{align*}
    \varrho-\bar\varrho = (u(s, y)-\bu(s,y))\frac{\sqrt[4]{2}}{\sqrt{\sigma}}\frac{e^{-y^4/4+\tilde{\kappa} y^2/2}}{\cZ},
\end{align*}
for which we need $c = 0$.

\begin{rmk}
\label{rem:Rabsolutevalue}
Writing 
$$u(s, y) = \sum_{m\in2\Z}\sum_{n\in\N^*}a_{m, n}e^{i m s}p_{2n}(y)+\sum_{m\in(2\Z+1)}\sum_{n\in\N^*} a_{m,n}e^{i m s}p_{2n-1}(y),$$
and using the orthogonality of the $p_n$, we get
$$\langle p_1, u(s, \cdot)\rangle =\sum_{m \in 2\Z+1}a_{m,1}e^{ims}.$$
Since there are only odd $m$'s in the sum, $\langle p_1, u(s+\pi, \cdot)\rangle = -\langle p_1, u(s, \cdot)\rangle$. Thus,
\begin{equation*}
    \sup_{s\in \T}|\langle p_1, u(s, \cdot)\rangle| = \sup_{s\in \T} \langle p_1, u(s, \cdot)\rangle,
\end{equation*}
and $\cR(\rho)$ is simply given by
\begin{align*}
    \cR(\rho) = \frac{\sqrt{\sigma b_1}}{\sqrt[4]{2}}\sup_{s\in \T} \sum_{m \in 2\Z+1}a_{m,1}e^{ims}.
\end{align*}
\end{rmk}

\subsection{Block-wise computation of operator norms}
\label{sec:opnormblock}

\begin{lem}\label{lem:seq-norm}Let $\mathcal{X}$ and $\mathcal{Y}$ be Banach spaces. For each $i,j \in \Z$, let $M_{ij} : \mathcal{X}\to\mathcal{Y}$ be a bounded operator. Consider the operator $M = (M_{ij})_{i,j, \in \Z}$. If $(\|M_{ij}\|_{\mathcal{X}\to \mathcal{Y}})_{i,j\in \Z} : \ell^p(\Z, \R) \to \ell^q(\Z, \R)$ is bounded for some $p, q\in [1, \infty]$, then $M : \ell^p(\Z,\mathcal{X}) \to \ell^q(\Z,\mathcal{Y})$ is bounded.
Furthermore,
$$\|M\|_{\ell^p(\mathcal{X})\to \ell^q(\mathcal{Y})}\leq \| (\|M_{ij}\|_{\mathcal{X}\to \mathcal{Y}})_{i,j\in \Z}\|_{\ell^p\to \ell^q}.$$
\end{lem}
\begin{proof}
We prove the statement for $p, q<\infty$, the other cases being similar. Let $y = (y_{j})_{j\in \Z} \in \ell^p(\mathcal{X})$, then
\begin{align*}
    \|My\|_{\ell^q(\mathcal{Y}) }^q &= \sum_{i\in \Z}\Big\|\sum_{j\in\Z}M_{ij}y_j\Big\|^q_{\mathcal{Y}}\\
    &\leq \sum_{i\in \Z}\left(\sum_{j\in\Z}\Big\|M_{ij}y_j\Big\|_{\mathcal{Y}}\right)^q\\
    &\leq \sum_{i\in \Z}\left(\sum_{j\in\Z}\|M_{ij}\|_{\mathcal{X}\to\mathcal{Y}}\|y_j\|_{\mathcal{X}}\right)^q\\
    &\leq\Big\|(\|M_{ij}\|_{\mathcal{X}\to\mathcal{Y}})_{i,j\in\Z}\Big\|^q_{\ell^p\to\ell^q}\Big\|(\|y_j\|_{\mathcal{X}})_{j\in \Z}\Big\|^q_{\ell^p}\\
    &\leq\Big\|(\|M_{ij}\|_{\mathcal{X}\to\mathcal{Y}})_{i,j\in \Z}\Big\|^q_{\ell^p\to\ell^q}\|y\|^q_{\ell^p(\mathcal{X})}. \qedhere
\end{align*}
\end{proof}

\end{appendix}

\end{document}